\documentclass[12pt,reqno]{amsart}
\usepackage[margin=1in]{geometry}
\usepackage{amsmath,amssymb,amsthm,graphicx,amsxtra, setspace}
\usepackage[utf8]{inputenc}
\usepackage{mathrsfs}
\usepackage{hyperref}
\usepackage{upgreek}
\usepackage{mathtools}
\usepackage[dvipsnames]{xcolor}
\usepackage[mathcal]{euscript}
\usepackage{amsmath,units}
\usepackage{pgfplots}
\usepackage{tikz}
\usepackage{verbatim}
\allowdisplaybreaks
\usepackage{dsfont}
\usepackage{fontenc}
\usepackage{textcomp}
\usepackage{marvosym}
\usepackage{eurosym}
\usepackage{upgreek}
\usepackage[pagewise]{lineno}

\DeclareMathAlphabet{\mathpzc}{OT1}{pzc}{m}{it}
\DeclareMathOperator*{\esssup}{ess\,sup}

\usepackage[cyr]{aeguill}

\hypersetup{colorlinks=true,%
	citecolor=red,%
	filecolor=blue,%
	linkcolor=blue,%
}
\usepackage{graphicx}

\usepackage{xpatch}
\makeatletter   
\xpatchcmd{\@tocline}
{\hfil\hbox to\@pnumwidth{\@tocpagenum{#7}}\par}
{\ifnum#1<0\hfill\else\dotfill\fi\hbox to\@pnumwidth{\@tocpagenum{#7}}\par}
{}{}
\makeatother    

\makeatletter
\def\l@subsection{\@tocline{2}{0pt}{4pc}{6pc}{}}
\def\l@subsubsection{\@tocline{3}{0pt}{8pc}{8pc}{}}
\makeatother

\makeatletter
\def\l@section{\@tocline{1}{12pt}{0pt}{}{\bfseries}}
\makeatother

\newtheorem{theorem}{Theorem}[section]
\newtheorem{lemma}[theorem]{Lemma}
\newtheorem{proposition}[theorem]{Proposition}

\newtheorem{definition}[theorem]{Definition}

\newtheorem{remark}[theorem]{Remark}
\newtheorem{hypothesis}[theorem]{Hypothesis}

\let\originalleft\left
\let\originalright\right
\renewcommand{\left}{\mathopen{}\mathclose\bgroup\originalleft}
\renewcommand{\right}{\aftergroup\egroup\originalright}

\newcommand{\Tr}{\mathop{\mathrm{Tr}}}
\renewcommand{\d}{\/\mathrm{d}\/}

\def\w{\textbf{W}^{\varepsilon}_{{\theta}^{\varepsilon}}}

\def\L{\mathbb{L}}
\def\A{\mathrm{A}}
\def\I{\mathrm{I}}

\def\C{\mathrm{C}}
\def\f{\boldsymbol{f}}

\def\B{\mathrm{B}}
\def\D{\mathrm{D}}
\def\y{\mathbf{y}}
\def\q{\boldsymbol{q}}
\def\Y{\mathbf{Y}}
\def\Z{\mathbf{Z}}
\def\E{\mathbb{E}}
\def\X{\mathbf{X}}
\def\x{\boldsymbol{x}}

\def\g{\boldsymbol{g}}
\def\p{\boldsymbol{p}}

\def\y{\boldsymbol{y}}
\def\z{\boldsymbol{z}}
\def\v{\boldsymbol{v}}
\def\w{\boldsymbol{w}}
\def\W{\mathrm{W}}

\def\Q{\mathrm{Q}}

\def\N{\mathbb{N}}

\def\V{\mathbb{V}}
\def\R{\mathbb{R}}
\def\wi{\widetilde}
\def\Q{\mathrm{Q}}
\def\U{\mathbb{U}}
\def\P{\mathbb{P}}
\def\u{\boldsymbol{u}}
\def\a{\boldsymbol{a}}
\def\H{\mathbb{H}}

\newcommand{\eps}{\varepsilon}

\renewcommand{\d}{\/\mathrm{d}\/}

\newcommand{\Addresses}{{
		\footnote{
			
		\noindent \textsuperscript{1,2}Department of Mathematics, Indian Institute of Technology Roorkee-IIT Roorkee,
		Haridwar Highway, Roorkee, Uttarakhand 247667, INDIA.\par\nopagebreak
		\noindent  \textit{e-mail:} \texttt{Manil T. Mohan: maniltmohan@ma.iitr.ac.in, maniltmohan@gmail.com}
		
		\textit{e-mail:} \texttt{Sagar Gautam: sagar\_g@ma.iitr.ac.in, sagargautamkm@gmail.com}
		
		\noindent \textsuperscript{*}Corresponding author.

			\textit{Key words:} Convective Brinkman-Forchheimer equations, damped Navier-Stokes equations, viscosity solutions, Hamilton-Jacobi-Bellman equations, dynamic programming principle.
			
			Mathematics Subject Classification (2020): Primary 93E20, 35D40, 49L25; Secondary 76B75, 76D05

}}}

\begin{document}
	
	
	\title[Hamilton-Jacobi-Bellman equations for 2D and 3D SCBF equations]{ Hamilton-Jacobi-Bellman equation and Viscosity solutions for an optimal control problem for stochastic convective Brinkman-Forchheimer equations
		\Addresses}
		\author[S. Gautam and M. T. Mohan]
	{Sagar Gautam\textsuperscript{1} and Manil T. Mohan\textsuperscript{2*}}

	\maketitle
	
	\begin{abstract}
		In this work, we consider the following two- and three-dimensional stochastic convective Brinkman-Forchheimer (SCBF) equations  in torus $\mathbb{T}^d,\ d\in\{2,3\}$: 
	\begin{align*}
		\mathrm{d}\boldsymbol{u}+\left[-\mu \Delta\boldsymbol{u}+(\boldsymbol{u}\cdot\nabla)\boldsymbol{u}+\alpha\boldsymbol{u}+\beta|\boldsymbol{u}|^{r-1}\boldsymbol{u}+\nabla p\right]\mathrm{d}t=\d\mathrm{W}, \ \nabla\cdot\boldsymbol{u}=0, 
	\end{align*}
	where $\mu,\alpha,\beta>0$, $r\in[1,\infty)$ and $\W$ is a Hilbert space valued $\mathrm{Q}-$Wiener process. The above system can be considered as damped stochastic Navier-Stokes equations.   Using the dynamic programming approach, we study the infinite-dimensional second-order Hamilton-Jacobi equation  associated with an optimal control problem for SCBF equations. For the supercritical case, that is, $r\in(3,\infty)$ for $d=2$ and $r\in(3,5)$ for $d=3$ ($2\beta\mu\geq 1$ for $r=3$ in $d\in\{2,3\}$), we first prove the existence of a viscosity solution for the infinite-dimensional HJB equation, which we identify with the value function of the associated control problem. By establishing  a comparison principle for $r\in(3,\infty)$ and $r=3$ with $2\beta\mu\geq1$ in $d\in\{2,3\}$, we prove that the value function is the unique viscosity solution and hence we resolve the global unique solvability of the HJB equation in both two and three dimensions.
	\end{abstract}
	
	\tableofcontents
	
	\section{Introduction}\setcounter{equation}{0}
	This article studies the existence and uniqueness of viscosity solutions for the Hamilton-Jacobi-Bellman (HJB) equation associated with an optimal control problem of 2D and 3D stochastic convective Brinkman-Forchheimer equations via dynamic programming approach.

		\subsection{The model}
	The convective Brinkman-Forchheimer (CBF) equations describe incompressible fluid flow motion in a saturated porous medium.  These equations are applicable when the fluid flow rate is sufficiently high and the porosity is not too small (see \cite{MTT} for detailed model physics). 
 Let us first provide a mathematical formulation of stochastic convective Brinkman-Forchheimer (SCBF) equations. 
 	For $\mathrm{L}>0$, we consider a $d$-dimensional torus $\mathbb{T}^d=\left(\frac{\R}{\mathrm{L}\mathbb{Z}}\right)^d$, $d\in\{2,3\}$. Let $t\in[0,\infty)$ be the initial time and $T\in(t,\infty)$ be the terminal time. Let  $\big(\Omega,\mathscr{F},\{\mathscr{F}_s^t\}_{s\geq t},\P\big)$ be a complete filtered probability space satisfying the usual conditions (see Section \ref{abscon}). Let us denote by $\X(\cdot):[t,T]\times\mathbb{T}^d\times\Omega\to\R^d,$ the velocity vector filed and by $p:[t,T]\times\mathbb{T}^d\times\Omega\to\R,$ the pressure. Then for incompressible fluid flow, the SCBF equations read as:
	\begin{equation}\label{a}
		\left\{
		\begin{aligned}
			\d\X(s)+[-\mu\Delta\X(s)+ (\X(s)\cdot\nabla)&\X(s)+\alpha\X(s)+\beta|\X(s)|^{r-1}\X(s)\\+\nabla p(s)]\d s&=\boldsymbol{g}(s,\a(s))\d s+\d\mathbf{W}(s), \ \text{ in } \ (t,T)\times\mathbb{T}^d, \\ \nabla\cdot\X(s)&=0, \ \text{ in } \ [t,T]\times\mathbb{T}^d, \\
			\X(t)&=\x_0 \ \text{ in } \ \mathbb{T}^d,
		\end{aligned}
		\right.
	\end{equation}
	where $\{\mathbf{W}(s)\}_{s\geq t}$ is Hilbert space valued Wiener process  with trace class covaraince defined on the filtered probability space  $\big(\Omega,\mathscr{F},\{\mathscr{F}_s^t\}_{s\geq t},\P\big)$. Here, $\boldsymbol{g}$ is an external forcing, and $\a(\cdot)$ denotes some control strategy, which will be described precisely in Section \ref{abscon}. 
	For the uniqueness of pressure, one can impose the condition $
	\int_{\mathbb{T}^d}p(s,\xi)\d\xi=0, s\in  [t,T].$ 
	The constant $\mu>0$ denotes the Brinkman coefficient (effective viscosity). The constants $\alpha,\beta>0$ are due to Darcy- Forchheimer law and we call $\alpha$ and $\beta$ as Darcy (permeability of the porous medium) and Forchheimer (proportional to the porosity of the material) coefficients, respectively. It can be easily seen that, for $\alpha=\beta=0$, the system \eqref{a} reduces to the classical stochastic Navier-Stokes equations (SNSE). Therefore, the system \eqref{a} can be considered as a damped SNSE.  The case of $r=3$ is referred in the literature as critical case (cf. \cite{KWH}).  Furthermore, we refer the case $r<3$ as \emph{subcritical} and $r>3$ as \emph{supercritical} (or fast growing nonlinearities).

	Let us mention that, the well-posedness results (existence and uniqueness of strong solutions) for the problem \eqref{a} is known in the literature for $r\in(3,\infty)$ in three dimensions also.
	\begin{table}[ht]
	\begin{tabular}{|c|c|c|c|c|}
		\hline
		\textbf{Dimension} &$ r$& Conditions on 
		$\mu$ \& $\beta$ \\
		\hline
		\textbf{d=2,3} &$r\in(3,\infty)$&  for any  
		$\mu>0$ and $\beta>0$  \\
		\hline
		\textbf{d=2,3} &$r=3$&for $\mu>0$ and
		$\beta>0$ with $2\beta\mu\geq1$ \\
		\hline
	\end{tabular}
	\vskip 0.1 cm
	\caption{Values of $\mu,\beta$ and $r$ for the comparison principle.}
	\label{Table1}
\end{table}
However, due to some technical difficulties (see Subsection \ref{sub-tech} below), we restrict the values of $r$ as mentioned in tables, to achieve the main goals of this work, that is,  \emph{the comparison principle} (Table \ref{Table1}) and the \emph{existence of viscosity solutions} (Table \ref{Table2}).
	\begin{table}[ht]
	\begin{tabular}{|c|c|c|c|c|}
		\hline
		\textbf{Dimension} &$ r$& Conditions on 
		$\mu$ \& $\beta$ \\
		\hline
		\textbf{d=2} &$r\in(3,\infty)$&  for any  
		$\mu>0$ and $\beta>0$  \\
		\hline
		\textbf{d=3} &$r\in(3,5)$&  for any  
		$\mu>0$ and $\beta>0$  \\
		\hline
		\textbf{d=2,3} &$r=3$&for $\mu>0$ and
		$\beta>0$ with $2\beta\mu\geq1$ \\
		\hline
	\end{tabular}
	\vskip 0.1 cm
	\caption{Values of $\mu,\beta$ and $r$ for the existence of viscosity solution.}
	\label{Table2}
\end{table}

	\subsubsection{Advantages of fast growing nonlinearities}\label{advnonabs}
	
    CBF equations are also known as NSE with damping (cf. \cite{MTT}). The damping arises from the resistance to the motion of the flow or by friction effects. The uniqueness of Leray-Hopf weak solutions and the global well-posedness of the strong solutions of 3D NSE are well-known open problems in Mathematical 
	Physics. The main difficulty lies in estimating the convective term under suitable norms. As an advantage, the fast growing nonlinearities (that is, $|\X|^{r-1}\X$ term with $r \geq 3$), as well as the diffusion term (that is, $-\Delta\X$) dominate the convective term and the convective term can be handled with the help of the H\"older and Young inequalities under suitable norms (see Remark \ref{BLrem}). Therefore, we are able to deal with the three dimensional case also. 

      \subsection{Deterministic and stochastic HJB equation}\label{detstchjb}
   In this work, we are mainly interested in the optimal control problem of SCBF equations \eqref{a} via \emph{dynamic programming method}. 
     The main idea of this method is to consider the family of optimal control problems with different initial times and states, and then establish a relationship among these problems via the so-called \emph{HJB equation}. These equations are of first order for the deterministic case and second order for the stochastic case. Below, we explain these equations in the context of our model, CBF equations. For more study on these equations in the finite-dimensional case, one can refer \cite{JYXYZ} and for infinite-dimensional case, see \cite{GFAS}.
    \subsubsection{Deterministic case (\cite{FGSSA1,SSS})} We consider  the functional framework $$\V\cap\widetilde{\L}^{r+1}\hookrightarrow\V\hookrightarrow\H\cong\H^*\hookrightarrow\V^*\hookrightarrow\V^*+\widetilde{\L}^{\frac{r+1}{r}},$$ as discussed  in Section \ref{Sec-2}. Let us first visit the deterministic case and consider the following CBF equations in $\H$ satisfied by the velocity vector filed $\Y:[t,T]\times\mathbb{T}^d\to\R^d$ (state variable):
    \begin{equation}\label{detHJB}
    	\left\{
    	\begin{aligned}
    		\frac{\d\Y(s)}{\d s}+\mu\A\Y(s)+\mathcal{B}(\Y(s))+\alpha\Y(s)+ \beta\mathcal{C}(\Y(s))&=\f(s,\a(s)) \  \text{ for a.e. } \ s\in[t,T],\\
    		\Y(t)&=\y\in\V,
    	\end{aligned}
    	\right.
    \end{equation}
    where, $\f:[0,T]\times\Theta\to\V$ is an external forcing, $\Theta$ is a complete metric space (a control set) and $\a(\cdot):[0,T]\to\Theta$ is a measurable function which plays the role of some control strategy. 
    Let us denote the set of all such control strategies by $\mathscr{U}$.  
    The optimal control problem consists of minimizing the following cost functional associated with the state equation \eqref{detHJB}:
    \begin{align}\label{detcost}
    	\mathcal{J}(t,\y;\a):=\int_t^T \ell(s,\Y(s),\a(s))\d s+g(\Y(T)),
    \end{align}
    for some measurable functions $\ell:[t,T]\times\V\times\Theta\to\R$ and $g:\H\to\R$ satisfying Hypothesis \ref{valueH}  below. For any given $\f\in\mathrm{L}^2(0,T;\H)$ and $\y\in\V,$ the controlled CBF equations \eqref{detHJB} has a unique strong solution 
    \begin{align*}
   \Y\in \C([0,T];\V)\cap\mathrm{L}^2(0,T;\D(\A))\cap\mathrm{L}^{r+1}(0,T;\wi{\L}^{p(r+1)}),
    \end{align*}  
    where $p\in[2,\infty)$ for $d=2$ and $p=3$ for $d=3$, continuously depending on the data (see \cite[Theorem 4.2]{SMTM}). Therefore, the cost functional defined in \eqref{detcost} is well-defined. The minimization of the cost functional \eqref{detcost} over $\mathscr{U}$ formally leads to the following HJB equation in $u:(0,T]\times\V\to\mathbb{R}$  for  $(t,\y)\in(0,T)\times\V:$
    \begin{equation}\label{detHJB1}
    	\left\{
    	\begin{aligned}
    		&u_t-(\mu\A\y+\mathcal{B}(\y)+\alpha\y+
    		\beta\mathcal{C}(\y),\D u)+F(t,\y,\D u)=0,\\
    		&u(T,\y)=g(\y), 
    	\end{aligned}
    	\right.
    \end{equation}
    where $$F(t,\y,\p)=\inf\limits_{\a\in\U} \left\{(\f(t,\a),\p)+\ell(t,\y,\a)\right\}$$ is the Hamiltonian function. 	Note that \eqref{detHJB1} is a  first order nonlinear partial differential equation (PDE). The above equation should be satisfied by the value function 
    \begin{align*}
    	\mathit{V}(t,\y):=\inf\limits_{\a(\cdot)\in\mathscr{U}} \mathcal{J}(t,\y,\a)
    \end{align*}
    in an appropriate sense. 	Let us make some comment about the PDE \eqref{detHJB1}. First note that, as the state  $\Y(\cdot)$ in \eqref{detHJB} evolves over the infinite-dimensional space $\V$, the above PDE \eqref{detHJB1} is called \emph{infinite-dimensional} or a \emph{PDE in infinitely many variables}. Solving such an infinite-dimensional PDE is the most difficult task. In general, this PDE not necessarily has a solution in the classical sense. In fact, the value function $\mathit{V}$ satisfies \eqref{detHJB1} in a generalized sense or what is known in the literature as the \emph{viscosity solution.}

    \subsubsection{Stochastic case}\label{stcoexm}
    The stochastic formulation of the problem is well-explained in Sections \ref{abscon} and \ref{valueSC}. In this case, our aim is to minimize the cost functional of the form
    \begin{align}\label{stochcost}
    	J(t,\y;\a(\cdot))=
    	\E\left\{\int_t^T \ell(s,\Y(s),\a(s))\d s+g(\Y(T))\right\},
    \end{align}
    over all admissible class $\mathscr{U}_t^\nu$ of controls strategies (see Subsection \ref{adcLa}). This leads to the following infinite-dimensional second order nonlinear PDE  for $ (t,\y)\in(0,T)\times\V$:
    \begin{equation}\label{eqn-HJBE}
    	\left\{
    	\begin{aligned}
    		&u_t+\frac12\mathrm{Tr}(\Q\D^2u)-(\mu\A\y+\mathcal{B}(\y)+\alpha\y+
    		\beta\mathcal{C}(\y),\D u)+F(t,\y,\D u)=0, \\
    		&u(T,\y)=g(\y).
    	\end{aligned}
    	\right.
    \end{equation}
    \emph{The main difference with the deterministic case is that the HJB equation \eqref{detHJB} is of first order, while  the extra second order term $\frac12\mathrm{Tr}(\Q\D^2u)$ is appearing in \eqref{eqn-HJBE}  due to the presence of the additive noise in the stochastic system \eqref{stap}.}

     Let us now briefly discuss the applicability of the control problem \eqref{stochcost}-\eqref{eqn-HJBE}. 
     Consider the running cost, terminal cost and the forcing as follows:
    \begin{equation*}
    	\left\{
    	\begin{aligned}
    		\ell(t,\y,\a)&=\|\nabla\times\y\|_{\H}^2+\frac12\|\a\|_{\H}^2,\\
    		g(\y)&=\|\y\|_{\H}^2 \ \text{ and } \ \\
    		\f(t,\a)&=\mathcal{K}\a\in\V.
    	\end{aligned}
    	\right.
    \end{equation*}
    Here $\mathcal{K}\in\mathscr{L}(\H;\V)$ is the control operator which is linear and continuous and, representing the spatial localization of the control. Further, the control $\a\in\U$, where $\U=\mathrm{B}_{\H}(0,R)\subset\H$ is the closed unit ball of radius $R$ in $\H$. For electrically conducting fluids such as liquid metals (also known as mercury or liquid sodium) and salt water (also known as electrolytes) such controllers can be realized through the action of the Lorentz force (cf. \cite[Chapter IV]{Schan}, \cite[Chapter 10]{HKM}). Since the Lorentz force is given by 
    $j\times B$, where $j$ is the current and $B$ is the magnetic field, spatial localization of such distributed control forces can be realized by choosing a suitable shape and location for the magnets. It is implemented in a boundary-layer-stabilization experiment in \cite{CHJS}.
    
    The Hamiltonian function in this case is 
    \begin{align*}
    	F(t,\y,\mathbf{p})=\|\nabla\times\y\|_{\H}^2+\mathit{h}
    	(\mathcal{K}^*\mathbf{p}),
    \end{align*}
    where $\mathit{h}(\cdot):\H\to\R$ is given by 
    \begin{align*}
    	\mathit{h}(\z):=\inf\limits_{\a\in\U}\left\{(\a,\z)+\frac12\|\a\|_{\H}^2\right\},
    \end{align*}
    and $\mathcal{K}^*$ is the adjoint of $\mathcal{K}$ considered as an operator from $\H$ to $\H$. In fact, one can  explicitly obtain $\mathit{h}(\cdot)$ as 
    \begin{align*}
    	\mathit{h}(\z)=\left\{\begin{array}{cc}-\frac{1}{2}\|\z\|_{\H}^2, &\text{ if } \|\z\|_{\H}\leq R,\\
    		-R\|\z\|_{\H}+\frac{R^2}{2}, &\text{ if } \|\z\|_{\H}>R.\end{array}\right.
    \end{align*}
    Moreover, we remark that the optimal feedback control is given \emph{formally} as 
    \begin{align}\label{optfeed}
    	\wi\a(t)=\mathfrak{M}\big(\mathcal{K}^*\D u(t,\y)\big),
    \end{align}
    where the function $\mathfrak{M}(\cdot)$ is given by
    \begin{align*}
    	\mathfrak{M}(\z):=\D\mathit{h}(\z)=
    	\begin{cases}
    		-\z, \ &\text{ if } \ \ \|\z\|_{\H}\leq R,\\
    		-\z\frac{R}{\|\z\|_{\H}}, \ &\text{ if } \ \ \|\z\|_{\H}> R.
    	\end{cases}
    \end{align*}
   Note that in the feedback formula \eqref{optfeed}, the term $\D u(\cdot,\cdot)$ needs to be justified. We cannot make sense of this term since we are working with non-smooth solutions (a viscosity solution). That is why, in general, we cannot obtain the feedback formula for HJB equation associated with stochastic CBF equations.
    
    Now, our primary focus is to discuss the solvability of the HJB equation \eqref{eqn-HJBE} for the stochastic CBF equations in detail like 
    \begin{itemize}
    	\item comparison principle, 
    	\item existence and uniqueness of viscosity solutions,
    \end{itemize}
     which is much more challenging and complicated  as compared to the deterministic one (see Sections \ref{valueSC}-\ref{extunqvisc1}). The deterministic case of  HJB equation associated with the optimal control problems for CBF equations can be explored similar to the case of stochastic CBF equations. {However, in the case of deterministic CBF equations, we can also show the existence of a viscosity solution for $r=5$ in $d=3$ with additional assumptions on the initial data, which we are not able to obtain in the case of stochastic CBF equations. The deterministic work is under preparation.} 
    
   \subsection{Literature survey}
   Let us now discuss the literature available for the stochastic CBF equations and viscosity solution of the HJB equation associated with fluid dynamic models. 
   \subsubsection{Stochastic CBF equations}
    Let us first discuss the literature available for stochastic CBF equations. For $\y\in\H$, by exploiting a monotonicity property of the linear and nonlinear operators as well as a stochastic generalization of the Minty-Browder technique, the author in \cite{KKMTM,MTM8} established the existence and uniqueness of a global strong solution $$\Y\in\C([0,T];\H)\cap\mathrm{L}^2(0,T;\V)\cap\mathrm{L}^{r+1}(0,T;\wi\L^{r+1}),\ \mathbb{P}\text{-a.s.},$$  satisfying the energy equality (It\^o's formula) for SCBF equations (in bounded domain and torus) driven by multiplicative Gaussian noise. 
	Under suitable assumptions on the initial data ($\y\in\V$) and noise coefficients, the author has also showed the following regularity result: 
	\begin{align*}
		\Y\in\C([0,T];\V)\cap\mathrm{L}^2(0,T;\D(\A))\cap\mathrm{L}^{r+1}(0,T;\wi{\L}^{p(r+1)}),\ \mathbb{P}\text{-a.s.},
	\end{align*}  
	where $p\in[2,\infty)$ for $d=2$ and $p=3$ for $d=3$. 
	The author in \cite{MT5} established, using the contraction mapping principle, the existence and uniqueness of local and global pathwise mild solutions for stochastic CBF equations perturbed by an additive L\'evy noise in $\R^d$, for $d\in\{2,3\}$. 
	For 2D and 3D  stochastic CBF equations perturbed by a multiplicative L\'evy noise, the existence of a weak martingale solution is proved in \cite{MTM9}. This is achieved by using a version of the Skorokhod embedding theorem for non-metric spaces (for $d\in\{2,3\}$ and $r\in[1,\infty)$), a compactness method, and the classical Faedo-Galerkin   approximation.
	\subsubsection{HJB equation and viscosity solution related to fluid dynamic problems}
One of the fundamental aspects of fluid mechanics is to understand the behaviour of turbulent flows. 
An interconnected problem is the question of controlling turbulence in the fluid flow. By this, we mean determining an optimal action that minimizes turbulence (\cite{F.T.}).
	 In this work, we study the \emph{dynamic programming approach for the control of turbulence in the incompressible fluid flows}. An emerging area in this field is the feedback control theory of fluid dynamic models, which has numerous industrial and engineering applications \cite{MDG,SSS}. The feedback synthesis of an optimal control problem for the stochastic Navier-Stokes equations and related models using the infinite-dimensional HJB equation represents one of the outstanding and exceptionally challenging problems in this direction. However, in the analysis of this equation, one encounters difficulties like the non-differentiability of the solutions or the non-uniqueness of the almost everywhere solutions. Therefore, a notion of a weak solution to such equations is needed. The appropriate notion of a weak solution is the notion of the viscosity solutions which resolve this problem.

The notion of viscosity solution was first introduced by the Crandall and Lions in \cite{MGL} (also see \cite{MGEL}) for the first order HJB equation in finite dimension. Later, Ishii \cite{Hsh1} proved the uniqueness of the viscosity solution (possibly unbounded) by establishing a comparison principle. 
	However, the second-order HJB is more complicated to solve than the first-order HJB, posing additional difficulty due to the presence of a higher order term. The notion of viscosity solution was extended to the second order case by Lions \cite{PL1,PL2,PL3} in the context of optimal control problem of diffusion processes. Interested readers can refer the survey article by the authors in \cite{MGHL} and the monographs \cite{WHF,JYXYZ} for a detailed study of viscosity solutions and complete treatment of finite-dimensional HJB equation. There is an increasing interest in, and a growing literature on, HJB equation in infinite dimensions. These equations were first studied by Barbu and Da Prato (for instance, see \cite{VBGDP}), setting the problem in classes of convex functions, using semigroup and perturbation methods. Further, the theory of viscosity solutions in the infinite-dimensional space was studied in a series of articles \cite{MGL1, MGL2, MGL3, MGL4} by Crandall and Lions. The author in \cite{PL4} extended the notion of viscosity solution for the unbounded case of second order HJB equation in the context of optimal control of Zakai equation (see also \cite{Hsh,Kshi}).  A comparison of our work with the works available for HJB equations associated with fluid dynamics models has been given in Subsection \ref{sub-comp} below. 
	
     Our objective is to examine the infinite-dimensional HJB equation that arise as an optimal control problem of stochastic CBF equations (or stochastic damped NSE) in two and three dimensions. \emph{The motivation behind such HJB equation comes from the problem of minimizing the enstrophy, that is,} $\int_0^T\|\nabla\times\Y(t)\|_{\H}^2\d t$ (see Subsection \ref{detstchjb}). 
     The HJB equation associated to  stochastic CBF equations (see \eqref{eqn-HJBE}) is a second order PDE whose unknown depends on time $t$ and the initial data $\y\in\V$. The linear part of the HJB equation  corresponds to the Kolmogorov equation associated to the uncontrolled stochastic CBF equations and its solution is given by the transition semigroup (see \cite{SGMM1}). In the context of the stochastic NSE it was studied by the authors \cite{FFFG,GDPAD} in two and three dimensions.
     The nonlinear part of the HJB equation involves the derivative of the unknown, and this term is more difficult to handle. 
     
	\subsection{Comparison with the works \cite{FGSSA,FGSSA1} and our contribution}\label{sub-comp}
	In the context of stochastic NSE, the associated HJB equation consists of the unbounded term (corresponding to the Stokes operator) and the nonlinear term (corresponding to convective term), which makes this problem more challenging to solve. 
	It was mentioned in \cite{SSS1} that global unique solvability of the infinite-dimensional HJB equation associated with  NSE is an open problem. It was first resolved in \cite{FGSSA1} for deterministic NSE and later in \cite{FGSSA} for the stochastic NSE in two dimensions. 
	
Meanwhile, in the present work, we are working with  CBF equations (or damped NSE), which are characterized by an additional highly nonlinear term, also known as an absorption term (as outlined in \eqref{a}). This extra term introduces significant complexity to the system, making the associated HJB equation much more complicated to solve than those associated with the stochastic NSE, which does not include this term. The nonlinear absorption term complicates the analytical approach and increases the computational difficulty, demanding more advanced techniques to solve the HJB equation associated with stochastic CBF equations.
	
The solvability of HJB equation in the previous works \cite{FGSSA, FGSSA1} is limited to two dimensions. This limitation arises because the global well-posedness of the strong solution for 3D-NSE is not yet known. Fortunately, our approach benefits from the unique mathematical properties of the absorption term $|\X|^{r-1}\X$, which has a mathematical advantage over the convective term $(\X\cdot\nabla)\X$ (see Subsection \eqref{advnonabs}). Specifically, for $r>3$ in dimension $d\in\{2,3\}$ and $r=3$ in $d=3$ with $2\beta\mu\geq1$, the absorption term along with the diffusion term (that is, $-\Delta\X$) dominate the convective term (see Remark \ref{BLrem}). This dominance leads to the global solvability of  CBF equations (or damped NSE), and thus, we extend the result from work \cite{FGSSA, FGSSA1} to 3D-CBF equations.
	 
	 The HJB equation \eqref{eqn-HJBE} considered in this work is more general than those considered in the literature (cf. \cite{gPaD}). The authors in \cite{gPaD} investigated semilinear equations associated with an optimal control problem from the mild solutions point of view. However, the viscosity solution approach is more general because we can handle  more complicated cost functionals. Following the works of \cite{FGSSA}, we apply our approach to a stochastic optimal control problem with the associated HJB equation, which is fully nonlinear in the gradient variable (see \eqref{eqn-HJBE}). Let us comment on the randomness (or noise) considered in the system \eqref{a}.  In \cite{gPaD}, the additive noise (or Wiener process) is supposed to be non-degenerate, and it must satisfy some additional assumptions that allow the corresponding Ornstein-Uhlenback (or transition semigroup) to have some smoothing properties. However, the covariance operator $\Q$  can also degenerate in the present work. But then, we must assume  Hypothesis \ref{trQ1} to avoid the non-degenerate case as in \cite{gPaD}. Moreover, since we are dealing with non-smooth (viscosity solution) solutions, we cannot derive the formula for optimal feedback control as done in \cite{gPaD}.
	 
%

\subsection{Difficulties, strategies and approaches}\label{sub-tech}
Note that, unlike the whole space and periodic domain,  the major difficulty in working with bounded domains is that $\mathcal{P}(|\Y|^{r-1}\Y)$  need not be zero on the boundary, and $\mathcal{P}$ and $-\Delta$ are not necessarily commuting (see \cite{JCR}). Therefore, the equality 
\begin{align}\label{toreq}
	&\int_{\mathcal{O}}(-\Delta\Y(\xi))\cdot|\Y(\xi)|^{r-1}\Y(\xi)\d \xi \nonumber\\&=\int_{\mathcal{O}}|\nabla\Y(\xi)|^2|\Y(\xi)|^{r-1}\d \xi+ \frac{r-1}{4}\int_{\mathcal{O}}|\Y(\xi)|^{r-3}|\nabla|\Y(\xi)|^2|^2\d \xi,
\end{align}
may not be useful in the context of bounded domains. In a $d$-dimensional torus $\mathbb{T}^d$, the Helmholtz-Hodge projection $\mathcal{P}$ and $-\Delta$ commutes (cf. \cite[Theorem 2.22]{JCR}). So, the equality \eqref{toreq} is quite useful in obtaining regularity results. 
When deriving the energy estimates for the case of strong solutions, a significant difficulty arises in handling the term $(\mathcal{B}(\Y),\A\Y)$ (see \eqref{vse7.1} of Propositions \ref{weLLp}). The authors in \cite{FGSSA} overcame this difficulty by utilizing the fact $(\mathcal{B}(\Y),\A\Y)=0$, which is valid only in two dimensions. However, in 3D, this identity does not hold, and this makes the treatment of the term $(\mathcal{B}(\Y),\A\Y)$ considerably more difficult. The same difficulty arises when we prove the comparison principle (see \textbf{Step-IV} of Proposition \ref{comparison}). Providentially in our work, the presence of the absorption term $|\Y|^{r-1}\Y$ conquer this difficulty by means of the equality \eqref{toreq} in both 2D and 3D (see Remark \ref{BLrem}). And therefore the identity $(\mathcal{B}(\Y),\A\Y)=0$ in $\mathbb{T}^2$ is no longer needed in our work. 
The well-posedness of  stochastic CBF equations for any $\y\in\V$ has been established in \cite[Theorem 3.7]{MTM8} (also see \cite[Theorem 3.7]{KKMTM}) for $r>3$ in $d\in\{2,3\}$ and for $r=3$ in $d\in\{2,3\}$ with $2\beta\mu\geq1$.

The significant difficulty of this work is in proving the comparison principle (Theorem \ref{comparison}) and the existence and uniqueness of the viscosity solutions (Section \ref{extunqvisc1}) associated with the optimal control problem of SCBF equations \eqref{a}. We follow similar  approach  performed in \cite{FGSSA, GFAS}.
However, the definition of test function (Definition \ref{testD}) considered here plays an important role in the analysis of our work. The authors in \cite{FGSSA}, consider the $\|\cdot\|_{\V}-$norm in the definition of test function which is equivalent to the $\|\nabla\cdot\|_{\H}-$norm, because of the Poincar\'e inequality. However in the present work, we are working on a torus where zero-average condition does not hold (see Subsection \ref{zerofunc} on functional setting) and hence in our definition of test function we have to deal with the full $\V-$norm, that is, $\|\y\|_{\V}^2=\|\y\|_{\H}^2+\|\nabla\y\|_{\H}^2$. Therefore, it is essential to note that, while applying the It\^o formula to the function $\delta(\cdot)(1+\|\Y_\eps(\cdot)\|_{\V}^2)^m$ (see \eqref{vdp3}-\eqref{vdp3.1}) in the proof of the existence of viscosity solutions (Theorem \ref{extunqvisc}), the (Fr\'echet) derivative of $\|\cdot\|_{\V}^2-$norm is not the same as in the case of NSE (for which zero-average condition holds, cf. \cite{FGSSA, GFAS}).
In particular, we have following:
\begin{align}
	\D(\|\y\|_{\V}^2)&=2\A\y, \ \text{ in the presence of zero-average condition (for NSE),}
	\label{frA}\\
    \D(\|\y\|_{\V}^2)&=2(\A+\I)\y, \ \text{ in the absence of zero-average condition (for CBF)}.\label{frAA}
\end{align}
Let us now mention how this fact influences the analysis of our work. 
First, since zero average condition does not hold for the SCBF system \eqref{a}, Stoke's operator $\A=-\mathcal{P}\Delta$ is not invertible as $0$ is an eigenvalue. However, $\A+\I$ (which appears in \eqref{frAA}) is invertible. This fact we are using to get
the weak convergence \eqref{wknm2} (see \eqref{wknmdiff}), which is the most important step in the proof of Theorem \ref{extunqvisc}.
Moreover, the term $(\A+\I)\y$, for $\y\in\D(\A+\I)$, helps us to avoid the restriction on $r$ (that is, for any $r>3$) in the proof of comparison theorem (see the calculations \eqref{viscdef1}-\eqref{viscdef3} in Step-IV of Theorem \ref{comparison}). 
The linear damping term $\alpha\y$ plays a crucial role to derive the energy estimates (see Proposition \ref{weLLp}) and to obtain the bound of the term involving $\|(\A+\I)\y\|_{\H}^2$ in the proof Theorem \ref{extunqvisc}. 

The additional difficulty arising in our work as compared to the works of \cite{FGSSA, GFAS} is to handle the following terms while proving Theorem \ref{extunqvisc} (existence of viscosity solution):
\begin{align}\label{extunqdff}
	\underbrace{\frac{1}{\eps}\E\int_{t_0}^{t_0+\eps}\delta(s)(\mathcal{B}(\Y_\eps(s)),(\A+\I)\Y_\eps(s))\d s}_{\text{integral with bilinear operator}} \ \text{ and } \ 
	\underbrace{\frac{1}{\eps}\E\int_{t_0}^{t_0+\eps}\delta(s)(\mathcal{C}(\Y_\eps(s)),(\A+\I)\Y_\eps(s))\d s}_{\text{integral with nonlinear operator}}.
\end{align}
In the works \cite{FGSSA, GFAS}, such terms do not appear since they considered the problem (in the context of NSE) on $\mathbb{T}^2$ where the condition $(\mathcal{B}(\y),\A\y)=0$ holds. However, in our case, since we are working in 3D also, we need to address and include these terms also. Apart from that, we have also to justify the  integral term with the nonlinear operator, as mentioned in \eqref{extunqdff}, due to the presence of absorption term in our model \eqref{a} (damped-NSE). We can justify the first term in \eqref{extunqdff} as `$\eps\to0$' with the help of uniform energy estimates given in \eqref{ssee1} and continuous dependence estimates \eqref{ctsdep0.1}-\eqref{ctsdep0.2} (see Appendix \ref{wknBA1}) in both two and three-dimensions. Meanwhile, the second term in \eqref{extunqdff} is hard to justify as `$\eps\to0$' for all absorption exponent $r>3$. This difficulty arises when applying H\"older's inequality to the expectation and then to the time integral. Specifically, the exponent $\frac{4(r+1)}{5-r}$ appears by performing this calculation, which forces us to restrict $3\leq r<5$ in three dimension. We cannot include the case $r=5$ also (see Appendix \ref{wknCA1}). Such restriction is not occurring in two dimensions due to the Sobolev embedding $\V\hookrightarrow\wi\L^{r+1}$ for any $r>3$.
Similar difficulty also occurs in the proof of comparison principle (Theorem \ref{comparison}), but there, the terms occur without integral and expectation (that is, stationary estimates; for instance, see  \eqref{viscdef3}), and therefore no restriction we are getting on the absorption exponent $r$ and hence comparison principle is valid for all $r>3$ (see \eqref{Cconv4}-\eqref{Cconv1} of Appendix \ref{useca}).

\subsection{Future scopes in this direction} 
The viscosity solution theory is not restricted only to the HJB equation and stochastic optimal control problems. It has huge applications in stochastic PDEs. One of the very interesting problem in this area is the Large deviation principle (LDP). The authors in \cite{JFTGK} proposed a new framework for large deviations based on nonlinear semigroup techniques and viscosity solutions in abstract spaces for metric space-valued Markov processes. In a nutshell, for the large deviation principle, one has to prove the exponential tightness and the existence of the so-called Laplace limit. The key idea is to identify the Lapalce limit (at a single time) as the convergence of viscosity solutions of singularly perturbed HJB equation (the skeleton equation)  and then  one can use the whole machinery of the viscosity solutions. The author in \cite{AS2} utilizes this technique for Hilbert space-valued diffusions to obtain a large deviation result for 2D stochastic NSE. Later, the authors in \cite{AS3Z} established a Laplace limit as a consequence of the convergence of the viscosity solutions of certain integro-PDE in a Hilbert space to the viscosity solution of the limiting first-order HJB equation.  It is a challenging problem in this area and in future we plan to work on this problem, that is,  the large deviation principle via the viscosity solution approach (as mentioned above), for 2D and 3D  stochastic CBF equations.

Another crucial and challenging problem in this direction is the viscosity solutions of the HJB equation for the optimal control problem for jump diffusions (or L\'evy noise). This problem has not been explored in the context of stochastic NSE or related models. It is because of the lack of exponential moments in the case of stochastic NSE. However, the authors in \cite{AS5Z, AS4Z} established the existence and uniqueness of viscosity solutions for HJB equation associated with certain integro-PDE. We are also planning to work on this challenging problem for stochastic CBF equations.

\subsection{Organization of the paper} 
The remaining sections are arranged as follows: In the upcoming section, we provide the necessary function spaces needed for the analysis of SCBF equations and the associated HJB equation. In Section \ref{abscon}, we provide the abstract formulation of the stochastic CBF equations \eqref{a} and the necessary stochastic functional framework, which are used throughout the work. In Section \ref{engdef}, we define  variational and strong solutions in the probabilistic sense (Definition \ref{def-var-strong}) of the stochastic CBF equations \eqref{stap}. We further prove the energy estimates (Proposition \ref{weLLp}) for the $p^{\mathrm{th}}$-moment of the stochastic process $\Y(\cdot)$ and also prove the estimates for the continuous dependence of the solution (Proposition \ref{cts-dep-soln}) of the system \eqref{stap}. Section \ref{valueSC} is the core of this article. We first formulate the stochastic optimal control (Subsection \ref{stoptcon}) and state the continuous dependence result for the value function and dynamic programming principle (Proposition \ref{valueP}). We then define viscosity solution for the HJB equation \eqref{HJBE} (Subsection \ref{compPR}) and prove one of the most important results of the this theory, the \emph{comparison principle} (Theorem \ref{comparison}). Finally, we prove the existence and uniqueness of the viscosity solution for the HJB equation \eqref{HJBE2} in  Section \ref{extunqvisc1} (Theorem \ref{extunqvisc}).

	\section{Functional Framework}\label{Sec-2}\setcounter{equation}{0}
	In this section, we define some necessary function spaces which are frequently used in the rest of the paper, and linear and nonlinear operators which help us to obtain the abstract formulation of the SCBF system \eqref{a}. For our analysis, we adapt the functional framework from the work \cite{JCR1}.

	\subsection{Function spaces}\label{zerofunc}
	Let $\C_{\mathrm{p}}^{\infty}(\mathbb{T}^d;\R^d)$ denote the space of all infinitely differentiable  functions $\u$ satisfying periodic boundary conditions $\u(x+\mathrm{L}e_{i},\cdot) = \u(x,\cdot)$, for $x\in \R^d$. \emph{We are not assuming the zero mean condition for the velocity field unlike the case of NSE, since the absorption term $\beta|\u|^{r-1}\u$ does not preserve this property (see \cite{MTT}). Therefore, we cannot use the well-known Poincar\'e inequality and we have to deal with the  full $\H^1$-norm.} The Sobolev space  $\H_{\mathrm{p}}^s(\mathbb{T}^d):=\mathrm{H}_{\mathrm{p}}^s(\mathbb{T}^d;\mathbb{R}^d)$ is the completion of $\C_{\mathrm{p}}^{\infty}(\mathbb{T}^d;\R^d)$  with respect to the $\H^s$-norm and the norm on the space $\H_{\mathrm{p}}^s(\mathbb{T}^d)$ is given by $$\|\u\|_{{\H}^s_{\mathrm{p}}}:=\left(\sum_{0\leq|\boldsymbol\alpha|
		\leq s} \|\D^{\boldsymbol\alpha}\u\|_{\mathbb{L}^2(\mathbb{T}^d)}^2\right)^{1/2}.$$ 	
	It is known from \cite[Proposition 5.39]{JCR1} that the Sobolev space of periodic functions $\H_{\mathrm{p}}^s(\mathbb{T}^d)$, for $s\geq0$ is the same as 
	$$\left\{\u:\u=\sum_{k\in\mathbb{Z}^d}\y_{k}\mathrm{e}^{2\pi i k\cdot x /  \mathrm{L}},\ \overline{\u}_{k}=\u_{-k}, \  \|\u\|_{{\H}^s_f}:=\left(\sum_{k\in\mathbb{Z}^d}(1+|k|^{2s})|\u_{k}|^2\right)^{1/2}<\infty\right\}.$$ We infer from \cite[Propositions 5.38]{JCR1} that the norms $\|\cdot\|_{{\H}^s_{\mathrm{p}}}$ and $\|\cdot\|_{{\H}^s_f}$ are equivalent. Let us define 
	\begin{align*} 
		\mathscr{V}:=\{\u\in\C_{\mathrm{p}}^{\infty}(\mathbb{T}^d;\R^d):\nabla\cdot\u=0\}.
	\end{align*}
	We define the spaces $\H$ and $\widetilde{\L}^{p}$ as the closure of $\mathscr{V}$ in the Lebesgue spaces $\L^2(\mathbb{T}^d):=\mathrm{L}^2(\mathbb{T}^d;\R^d)$ and $\L^p(\mathbb{T}^d):=\mathrm{L}^p(\mathbb{T}^d;\R^d)$ for $p\in(2,\infty]$, respectively. We endow the space $\H$ with the inner product and norm of $\L^2(\mathbb{T}^d),$ and are denoted by 
	\begin{align*}
		&(\u,\v):=(\u,\v)_{\L^2(\mathbb{T}^d)}=\int_{\mathbb{T}^d}\u(x)\cdot\v(x)\d x\\ \text{ and } \  &\|\u\|_{\H}^2:=\|\u\|_{\L^2(\mathbb{T}^d)}^2=\int_{\mathbb{T}^d}|\u(x)|^2\d x, \ \text{ for } \ \u,\v\in\H.
	\end{align*}
	For $p\in(2,\infty)$, the space $\widetilde{\L}^{p}$ is endowed with the norm of $\L^p(\mathbb{T}^d)$, which is defined by 
	\begin{align*}
		\|\u\|_{\widetilde{\L}^p}^p:=\|\u\|_{\L^p(\mathbb{T}^d)}^p
		=\int_{\mathbb{T}^d}|\u(x)|^p\d x \  \text{ for } \ \u\in\wi\L^p.
	\end{align*}
	For $p=\infty$, the space $\widetilde{\L}^{\infty}$ is endowed with the norm of $\L^{\infty}(\mathbb{T}^d)$, which is given by 
		\begin{align*}
		\|\u\|_{\widetilde{\L}^{\infty}}:=\|\u\|_{\L^{\infty}(\mathbb{T}^d)}
		=\esssup_{x\in\mathbb{T}^d}|\u(x)| \  \text{ for } \ \u\in\wi\L^{\infty}.
	\end{align*}
	We also define the space $\V$ as the closure of $\mathscr{V}$ in the Sobolev space $\H^1_{\mathrm{p}}(\mathbb{T}^d)$. We equip the space $\V$ with the inner product
	\begin{align*}
		(\u,\v)_{\V}&:=(\u,\v)_{\L^2(\mathbb{T}^d)}+(\nabla\u,\nabla\v)_{\L^2(\mathbb{T}^d)}\nonumber\\&=
		\int_{\mathbb{T}^d}\u(x)\cdot\v(x)\d x+\int_{\mathbb{T}^d}\nabla\u(x)\cdot\nabla\v(x)\d x\ \text{ for } \ \u,\v\in\V,
     \end{align*}
     and the norm
     \begin{align*}
		\|\u\|_{\V}^2:=\|\u\|_{\L^2(\mathbb{T}^d)}^2+\|\nabla\u\|_{\L^2(\mathbb{T}^d)}^2=\int_{\mathbb{T}^d}|\u(x)|^2\d x+\int_{\mathbb{T}^d}|\nabla\u(x)|^2\d x \ \text{ for } \ \u,\v\in\V.
	\end{align*}
%
	Let $\langle \cdot,\cdot\rangle $ represent the induced duality between the spaces $\V$  and its dual $\V^*$ as well as $\widetilde{\L}^p$ and its dual $\widetilde{\L}^{p'}$, where $\frac{1}{p}+\frac{1}{p'}=1$. Note that $\H$ can be identified with its own dual $\H^*$. From \cite[Subsection 2.1]{FKS}, we have that the sum space $\V^*+\widetilde{\L}^{p'}$ is well defined and  is a Banach space with respect to the norm 
	\begin{align*}
		\|\u\|_{\V^*+\widetilde{\L}^{p'}}&:=\inf\{\|\u_1\|_{\V^*}+\|\u_2\|_{\wi\L^{p'}}:\u=\u_1+\u_2, \u_1\in\V^* \ \text{and} \ \u_2\in\wi\L^{p'}\}\nonumber\\&=
		\sup\left\{\frac{|\langle\u_1+\u_2,\boldsymbol{\xi}\rangle|}{\|\boldsymbol{\xi}\|_{\V\cap\widetilde{\L}^p}}:\boldsymbol{0}\neq\boldsymbol{\xi}\in\V\cap\widetilde{\L}^p\right\},
	\end{align*}
	where $\|\cdot\|_{\V\cap\widetilde{\L}^p}:=\max\{\|\cdot\|_{\V}, \|\cdot\|_{\wi\L^p}\}$ is a norm on the Banach space $\V\cap\widetilde{\L}^p$. Also the norm $\max\{\|\cdot\|_{\V}, \|\cdot\|_{\wi\L^p}\}$ is equivalent to the norms  $\|\cdot\|_{\V}+\|\cdot\|_{\widetilde{\L}^{p}}$ and $\sqrt{\|\cdot\|_{\V}^2+\|\cdot\|_{\widetilde{\L}^{p}}^2}$ on the space $\V\cap\widetilde{\L}^p$. Furthermore, we have
	$$
	(\V^*+\widetilde{\L}^{p'})^*\cong	\V\cap\widetilde{\L}^p \  \text{and} \ (\V\cap\widetilde{\L}^p)^*\cong\V^*+\widetilde{\L}^{p'}.
	$$
	Moreover, we have the continuous embeddings $\V\cap\widetilde{\L}^p\hookrightarrow\V\hookrightarrow\H\cong\H^*\hookrightarrow\V^*\hookrightarrow\V^*+\widetilde{\L}^{p'},$ where the embedding $\V\hookrightarrow\H$ is compact. 
	
	Apart from these functional settings, we use the following function spaces while proving the comparison principle and showing the existence and uniqueness of a viscosity solution:

	We denote by $\mathrm{C}^2(\H)$ (and $\mathrm{C}^2(\V)$), the space of all functions which are continuous on $\H$ (and $\V$) together with all their Fr\'echet derivatives up to order $2$. For a given $0\leq t<T$, we denote by $\mathrm{C}^{1,2}((t,T)\times\H)$, the space of all functions $\varphi:(t,T)\times\H\to\R$ for which $\varphi_t$ and $\D\varphi, \D^2\varphi$ (the Fr\'echet derivatives of $\varphi$ with respect to $\x\in\H$) exist and are uniformly continuous on closed and bounded subsets of $(t,T)\times\H$.

	\subsection{Linear operator}\label{linope}
	Let $\mathcal{P}_p: \L^p(\mathbb{T}^d) \to\wi\L^p,$ $p\in[1,\infty)$ be the Helmholtz-Hodge (or Leray) projection operator (cf.  \cite{DFHM}, etc.).	Note that $\mathcal{P}_p$ is a bounded linear operator and for $p=2$,  $\mathcal{P}:=\mathcal{P}_2$ is an orthogonal projection (see \cite[Section 2.1]{JCR}). We define the Stokes operator 
	\begin{equation*}
		\left\{
		\begin{aligned}
			\A\u&:=-\mathcal{P}\Delta\u=-\Delta\u,\;\u\in\D(\A),\\
			\D(\A)&:=\V\cap{\H}^{2}_\mathrm{p}(\mathbb{T}^d). 
		\end{aligned}
		\right.
	\end{equation*}
	For the Fourier expansion $\u=\sum\limits_{k\in\mathbb{Z}^d} e^{2\pi i k\cdot x} \u_{k} ,$ we calculate by using Parseval's identity
	\begin{align*}
		\|\u\|_{\H}^2=\sum\limits_{k\in\mathbb{Z}^d} |\u_{k}|^2 \  \text{and} \ \|\A\u\|_{\H}^2=(2\pi)^4\sum_{k\in\mathbb{Z}^d}|k|^{4}|\u_{k}|^2.
	\end{align*}
	Therefore, we have 
	\begin{align*}
		\|\u\|_{\H^2_\mathrm{p}}^2=\sum_{k\in\mathbb{Z}^d}|\u_{k}|^2+\sum_{k\in\mathbb{Z}^d}|k|^{4}|\u_{k}|^2= \|\u\|_{\H}^2+\frac{1}{(2\pi)^4}\|\A\u\|_{\H}^2\leq\|\u\|_{\H}^2+\|\A\u\|_{\H}^2.
	\end{align*}
	Moreover, by the definition of $\|\cdot\|_{\H^2_\mathrm{p}}$, we have $	\|\u\|_{\H^2_\mathrm{p}}^2\geq\|\u\|_{\H}^2+\|\A\u\|_{\H}^2$ and hence it is immediate that both the norms are equivalent and  $\D(\A+\I)=\H^2_\mathrm{p}(\mathbb{T}^d)=:\V_2$.

	\begin{remark}\label{rg3L3r}
   1.) For $d=2$, by using the Sobolev embedding, $\H_{\mathrm{p}}^1(\mathbb{T}^d)\hookrightarrow\L^p(\mathbb{T}^d)$, for all $p\in[2,\infty)$, we find 
		\begin{align*}
			\|\u\|^{r+1}_{\L^{p(r+1)}(\mathbb{T}^d)}&=\||\u|^{\frac{r+1}{2}}\|_{\L^{2p}(\mathbb{T}^d)}^2\leq C\||\u|^{\frac{r+1}{2}}\|_{\H_{\mathrm{p}}^1(\mathbb{T}^d)}^2\nonumber\\&\leq C\bigg(\int_{\mathbb{T}^d}|\nabla|\u(x)|^{\frac{r+1}{2}}|^2\d x+
			\int_{\mathbb{T}^d}||\u(x)|^{\frac{r+1}{2}}|^2\d x\bigg).
		\end{align*}
		By using elementary calculus identity $\partial_k(|\u|^{r+1})=(r+1)u_k(\partial u_k)|\u|^{r-1}$, where $\partial_k$ denotes the $k$-th partial derivative, we infer that $|\nabla|\u|^{\frac{r+1}{2}}|^2\leq C_r|\u|^{r-1}|\nabla\u|^2$ (see the proof of \cite[Lemma 1]{JCRW}). Thus, we further have 
		\begin{align}\label{3a71}
			\|\u\|^{r+1}_{\L^{p(r+1)}(\mathbb{T}^d)}\leq C\bigg(\int_{\mathbb{T}^d}|\nabla\u(x)|^2|\u(x)|^{r-1}\d x+ \int_{\mathbb{T}^d}|\u(x)|^{r+1}\d x\bigg),
		\end{align}
		for all $\u\in\V_2$ and for any $p\in[2,\infty)$. 
		
		3.) Similarly, for $d=3$, by the Sobolev embedding $\H_{\mathrm{p}}^1(\mathbb{T}^d)\hookrightarrow\L^6(\mathbb{T}^d)$, we find
		\begin{align}\label{371}
			\|\u\|^{r+1}_{\L^{3(r+1)}(\mathbb{T}^d)}&=\||\u|^{\frac{r+1}{2}}\|_{\L^{6}(\mathbb{T}^d)}^2\leq C\||\u|^{\frac{r+1}{2}}\|_{\H_{\mathrm{p}}^1(\mathbb{T}^d)}^2
			\nonumber\\&\leq C\bigg(\int_{\mathbb{T}^d}|\nabla\u(x)|^2|\u(x)|^{r-1}\d x+ \int_{\mathbb{T}^d}|\u(x)|^{r+1}\d x\bigg), 
		\end{align}
		for all $\u\in\V_2$. 
	\end{remark}

	\subsection{Bilinear operator}
		Let $b(\cdot,\cdot,\cdot):\V\times\V\times\V\to\R$ be a continuous trilinear form defined by
		\begin{align*}
			b(\u,\v,\w)=\int_{\mathbb{T}^d}(\u(x)\cdot\nabla)\v(x)\cdot\w(x)\d x.
		\end{align*} 
		Using weak incompressibility condition, the trilinear form satisfies $b(\u,\v,\v)=0$ for all $\u,\v\in\V$. By the Riesz representation theorem, there exists a continuous bilinear map $\mathcal{B}(\cdot,\cdot):\V\times\V\to\R$ such that $\langle\mathcal{B}(\u,\v),\w\rangle=b(\u,\v,\w)$ for all $\u,\v,\w\in\V$, which also satisfies (see \cite {Te})
		\begin{align}\label{syymB}
			\langle\mathcal{B}(\u,\v),\w\rangle=-\langle\mathcal{B}(\u,\w),\v\rangle \ \text{ and } \
			\langle\mathcal{B}(\u,\v),\v\rangle=0,
		\end{align}
		for any $\u,\v,\w\in\V$. Without danger of ambiguity, we also denote $\mathcal{B}(\u) = \mathcal{B}(\u, \u)$. Note that if $\u,\v\in\H$ are such that $(\u\cdot\nabla)\v=\sum\limits_{j=1}^d u_j\frac{\partial v_j}{\partial x_j}\in\L^2(\mathbb{T}^d)$, then $\mathcal{B}(\u,\v)=\mathcal{P}[(\u\cdot\nabla)\v]$.

		\begin{remark}\label{BLrem}\cite[Theorem 2.5]{SMTM}
			1.) In view of \eqref{syymB}, along with H\"older's and Young's inequalities, we calculate
			\begin{align}\label{syymB1}
				|\langle\mathcal{B}(\u)-\mathcal{B}(\v),\u-\v\rangle|\leq
				\frac{\mu }{2}\|\nabla(\u-\v)\|_{\H}^2+\frac{1}{2\mu }\||\v|(\u-\v)\|_{\H}^2.
			\end{align} 
			Using H\"older's and Young's inequalities, we estimate the term $\||\v|(\u-\v)\|_{\H}^2$ as
			\begin{align}\label{syymB2}
				\int_{\mathbb{T}^d}|\v(x)|^2|\u(x)-\v(x)|^2\d x &=\int_{\mathbb{T}^d}|\v(x)|^2|\u(x)-\v(x)|^{\frac{4}{r-1}}|\u(x)-\v(x)|^{\frac{2(r-3)}{r-1}}\d x\nonumber\\&\leq\frac{\beta\mu }{2}\||\v|^{\frac{r-1}{2}}(\u-\v)\|_{\H}^2+\frac{r-3}{r-1}\left[\frac{4}{\beta\mu (r-1)}\right]^{\frac{2}{r-3}}\|\u-\v\|_{\H}^2,
			\end{align}
			for $r>3$. Using \eqref{syymB2} in \eqref{syymB1}, we find 
			\begin{align}\label{3.4}
				|\langle\mathcal{B}(\u)-\mathcal{B}(\v),\u-\v\rangle|\leq
				\frac{\mu }{2}\|\nabla(\u-\v)\|_{\H}^2 +\frac{\beta}{4}\||\v|^{\frac{r-1}{2}}(\u-\v)\|_{\H}^2 +\varrho\|\u-\v\|_{\H}^2,
			\end{align}
			where \begin{align}\label{eqn-varrho}
				\varrho:=\frac{r-3}{2\mu(r-1)}\left[\frac{4}{\beta\mu (r-1)}\right]^{\frac{2}{r-3}}.
				\end{align}
			
			2.) In a similar way, one can establish the following inequality:
			\begin{align}\label{syymB3}
				|(\B(\u),\A\u)|\leq\frac{\mu}{2}\|\A\u\|_{\H}^2+\frac{\beta}{4} \||\u|^{\frac{r-1}{2}}\nabla\u\|_{\H}^2+\varrho\|\nabla\u\|_{\H}^2.
			\end{align}
		\end{remark}

	 \subsection{Nonlinear operator}
		Let us define the operator $$\mathcal{C}(\u):=\mathcal{P}(|\u|^{r-1}\u)\ \text{ for }\ \u\in\V\cap\L^{r+1}.$$   Since the projection operator $\mathcal{P}$ is bounded from $\H^1$ into itself (cf. \cite[Remark 1.6]{Te}), the operator $\mathcal{C}(\cdot):\V\cap\widetilde{\L}^{r+1}\to\V^*+\widetilde{\L}^{\frac{r+1}{r}}$ is well-defined and we have $\langle\mathcal{C}(\u),\u\rangle =\|\u\|_{\widetilde{\L}^{r+1}}^{r+1}.$ Moreover, for all $\u\in\V\cap\L^{r+1}$, the map $\mathcal{C}(\cdot):\V\cap\widetilde{\L}^{r+1}\to\V^*+\widetilde{\L}^{\frac{r+1}{r}}$ is Gateaux differentiable with Gateaux derivative given by 
		\begin{align}\label{C}
			\mathcal{C}'(\u)\v&=\left\{\begin{array}{cl}\mathcal{P}(\v),&\text{ for }r=1,\\ \left\{\begin{array}{cc}\mathcal{P}(|\u|^{r-1}\v)+(r-1)\mathcal{P}\left(\frac{\u}{|\u|^{3-r}}(\u\cdot\v)\right),&\text{ if }\u\neq \mathbf{0},\\\mathbf{0},&\text{ if }\u=\mathbf{0},\end{array}\right.&\text{ for } 1<r<3,\\ \mathcal{P}(|\u|^{r-1}\v)+(r-1)\mathcal{P}(\u|\u|^{r-3}(\u\cdot\v)), &\text{ for }r\geq 3,\end{array}\right.
		\end{align}
		for all $\v\in\V\cap\widetilde{\L}^{r+1}$. 
		The following lemma pertains to some monotonicity estimates of the nonlinear operator $\mathcal{C}(\cdot)$, which we use frequently throughout this work.
		\begin{lemma}[{\cite[Subsection 2.4]{SMTM}}]
			For every $r\geq1$ and for all $\u,\v,\w\in\widetilde{\L}^{r+1}$, the nonlinear operator $\mathcal{C}(\cdot)$ satisfies following estimates:
			\begin{align}\label{monoC1}
				\langle\mathcal{C}(\u)-\mathcal{C}(\v),\w\rangle\leq
				r\left(\|\u\|_{\widetilde{\L}^{r+1}}+\|\v\|_{\widetilde{\L}^{r+1}}\right)^{r-1}\|\u-\v\|_{\widetilde{\L}^{r+1}}\|\w\|_{\widetilde{\L}^{r+1}}
			\end{align}
			and 
			\begin{align}\label{monoC2}
				\langle\mathcal{C}(\u)-\mathcal{C}(\v),\u-\v\rangle\geq \frac{1}{2}\||\u|^{\frac{r-1}{2}}(\u-\v)\|_{\H}^2+\frac{1}{2}\||\v|^{\frac{r-1}{2}}(\u-\v)\|_{\H}^2\geq\frac{1}{2^{r-1}}\|\u-\v\|_{\wi\L^{r+1}}^{r+1}.
			\end{align}
	\end{lemma} 
	
	\begin{remark}[{\cite[Lemma 2.1]{KWH}}]
		The following equality on a torus $\mathbb{T}^d$ is used frequently in the paper:
		\begin{align}\label{torusequ}
		(\mathcal{C}(\u),\A\u)=\||\u|^{\frac{r-1}{2}}\nabla\u\|_{\H}^{2} +4\left[\frac{r-1}{(r+1)^2}\right]\|\nabla|\u|^{\frac{r+1}{2}}\|_{\H}^{2}.
		\end{align}
	\end{remark}

	\subsection{Some useful functional inequalities and definitions}
	\label{defunfmod}
	The following inequalities and definitions are frequently used throughout the paper: 
	
	1.) \emph{Interpolation inequality:} Let $0\leq s_1\leq s\leq s_2\leq\infty$ and $0\leq\theta\leq1$ be such that $\frac{1}{s}=\frac{\theta}{s_1}+\frac{1-\theta}{s_2}$. Then for $\u\in\L^{s_2}(\mathbb{T}^d)$, we have
	\begin{align*}
		\|\u\|_{\L^s}\leq\|\u\|_{\L^{s_1}}^{\theta}\|\u\|_{\L^{s_2}}^{1-\theta}.
	\end{align*} 
	
	2.) \emph{Agmon's inequality:} For all $\u\in\H^2_{\mathrm{p}}(\mathbb{T}^d)$, $d=\{2,3\}$, we have 
	\begin{align}\label{agmon}
		\|\u\|_{\L^{\infty}(\mathbb{T}^d)}\leq C\|\u\|_{\H}^{1-\frac{d}{4}}
		\|\u\|_{\H^2_{\mathrm{p}}(\mathbb{T}^d)}^{\frac{d}{4}}
		=\|\u\|_{\H}^{1-\frac{d}{4}}
		\|(\I+\A)\u\|_{\H}^{\frac{d}{4}}.
	\end{align}

  3.) \emph{Modulus of continuity (cf. \cite[Appendix D]{GFAS}):} A function $\upomega(\cdot):[0,+\infty)\to[0,+\infty)$ is called a \emph{modulus of continuity} if $\upomega$ is continuous, non-decreasing, subadditive and $\upomega(0)=0$. The subadditivity property of $\upomega$ also implies that for all $\epsilon>0$ there exists a constant $C_\epsilon>0$ such that 
  \begin{align*}
  	\upomega(r)\leq\epsilon+C_{\epsilon}r \  \text{ for every} \ r\geq0. 
  \end{align*}
  
   The following result shows that we can also assume the modulus of continuity to be concave.
  
  \begin{lemma}\cite[Lemma 11]{SiB}\label{lem-modulus}
  	Let $\upomega$ be a modulus of continuity. Then there exists a concave modulus of continuity $\upomega'$ with $\upomega'\geq \upomega$.
  \end{lemma}
  
  4.) \emph{Local modulus:} A function $\upomega(\cdot,\cdot):[0,+\infty)\times[0,+\infty)\to[0,+\infty)$ is called a local modulus of continuity if it is continuous and non-decreasing in both variables, subadditive in first variable, and $\upomega(0,r)=0$ for every $r\geq0$.
  
  5.) \emph{Modulus of continuity for uniformly continuous functions
  	(cf. \cite[Appendix D]{GFAS}):}
        Let $\upphi:\mathbb{X}_1\to\mathbb{X}_2$ be a uniformly continuous function, where $\mathbb{X}_1$ and $\mathbb{X}_2$ are two Banach spaces. We define its modulus of continuity, denoted by $\upomega_{\upphi}(\cdot)$, as follows:
        \begin{align*}
        \upomega_{\upphi}(\epsilon):=\sup\limits_{x,y\in\mathbb{X}_1}
        \{\|\upphi(x)-\upphi(y)\|_{\mathbb{X}_2}:\|x-y\|_{\mathbb{X}_1}
        \leq\epsilon\},  \  \text{ for } \ \epsilon\geq0.
        \end{align*}
  The above definition also implies that $\|\upphi(x)-\upphi(y)\|_{\mathbb{X}_2}\leq\upomega_{\upphi}(\|x-y\|_{\mathbb{X}_1})$ for every $x,y\in\mathbb{X}_1$. Note that for uniformly continuous functions, modulus of continuity always exist and so there exist two constants $c_0,c_1$ such that 
  \begin{align*}
  	\|\upphi(x)\|_{\mathbb{X}_1}\leq c_0+c_1\|x\|_{\mathbb{X}_1},  \ \ 
  	\text{ for every } \ x\in\mathbb{X}_1. 
  \end{align*}

	\section{Controlled SCBF equations}\label{abscon} \setcounter{equation}{0}
	In this section, we provide  the abstract formulation of the model \eqref{a} and provide some necessary hypothesis and assumptions that are used through out the article.

 Let $(\Omega,\mathscr{F},\P)$ be a complete probability space equipped with the filtration $\{\mathscr{F}_s^t\}_{s\geq t}$ such that 
	\begin{itemize}
		\item $\{\mathscr{F}_s^t\}_{s\geq t}$ is \emph{right-continuous}, that is, for all $s\geq t$, we have $\mathscr{F}_{s+}^t:=\bigcup\limits_{r>s}\mathscr{F}_r^t=\mathscr{F}_s^t$.
		\item $\mathscr{F}_t^t$ contains all $\P$-null sets of $\mathscr{F}$.
	\end{itemize}
	
	Let $\Q$ be a bounded linear operator on $\H$ which is non-negative, self-adjoint and such that $\Tr(\Q)<\infty$. Let $\mathbf{W}$ be an $\H$-valued $\Q$-Wiener process with $\mathbf{W}(t)=0,$ $\P-$a.s.
	\emph{As discussed in the introduction of this article, the operator $\Q$ can be totally degenerate also, which is contrary to the case of \cite{gPaD}.} We will require the following hypothesis throughout the article:
	\begin{hypothesis}\label{trQ1}
	 $\A^{\frac12}\Q^{\frac12}$ is a Hilbert-Schimdt operator.
	\end{hypothesis}
	Hypothesis \ref{trQ1} also says that $\Q_1:=\A^{\frac12}\Q\A^{\frac12}$ is densely defined and it can be extended to a bounded linear operator, still denoted by $\Q_1$ which is of trace-class. 
	
A 5-tuple $\nu:=(\Omega,\mathscr{F},\{\mathscr{F}_s^t\}_{s\geq t},\P,\mathbf{W})$, described above is called a \emph{generalized reference probability space} (see \cite[Definition 1.100, Chapter 1, pp. 35]{GFAS}). 
We denote by $M^2_\nu(t,T;\H)$ (a subset of $\mathrm{L}^2((t,T)\times\Omega;\H)$) the space of all $\H$-valued progressively measurable processes $\Y(\cdot)$ such that
	\begin{align*}
		\|\Y(\cdot)\|_{M^2_\nu(t,T;\H)}:=\left(\E\left[\int_t^T \|\Y(s)\|_{\H}^2\d s\right]\right)^{\frac12}<+\infty.
	\end{align*}
	The notation $M^2_\nu(t,T;\H)$ emphasizes the dependence on the generalized reference probability space $\nu$. Processes in $M^2_\nu(t,T;\H)$ are identified if they are equal $\P\otimes\d t$-a.e.
    \subsection{Admissible class}\label{adcLa}
	Let $\U$ be a complete separable metric space. Let $\a(\cdot):[0,T]\times\Omega\to\U$ be a stochastic process which indicate some control strategy. We say $\a(\cdot)$ is an \emph{admissible control} on $[t,T]$ if $\a(\cdot)$ is an $\mathscr{F}_s^t-$progressively measurable. Given the initial time $t$, we denote the set of all such control strategies, called the \emph{admissible class of controls}, by $\mathscr{U}_t^\nu$ and it is given by
	\begin{align*}
		\mathscr{U}_t^\nu:=\{\a(\cdot):[t,T]\times\Omega\to\U\big|\a(\cdot) \ \text{ is } \ \mathscr{F}_s^t-\text{progressively measurable}\}. 
	\end{align*} 
	The notation $\mathscr{U}_t^\nu$ emphasizes the dependence on the generalized reference probability space $\nu$. Note that as mentioned in \cite[Remark 2.5, Chapter 2, pp. 96]{GFAS}, the controls in $\mathscr{U}_t^\nu$ can be measurable and adapted also, since every adapted process has a progressively measurable modification (see \cite[Lemma 1.72, Chapter 1, pp. 24]{GFAS}). In view of \cite[Lemma 1.99, Chapter 1, pp. 34]{GFAS}, one could also work with predictable controls.
	
	\subsection{Abstract formulation} Let us set $\Y(\cdot):=\mathcal{P}\X(\cdot)$, $\mathcal{P}\x_0:=\y$,  $\mathcal{P}\g=\f$ and $\mathbf{W}(\cdot):=\mathcal{P}\mathbf{W}(\cdot)$. On projecting the first equation in \eqref{a}, we obtain, for $\a(\cdot)\in\mathscr{U}_t^\nu$, the abstract stochastic controlled convective Brinkman-Forchheimer equations describing the evolution of the velocity vector field $\mathbf{Y}(\cdot):[t,T]\times\mathbb{T}^d\times\Omega\to\R^d$ and satisfying the following stochastic evolution system:
	\begin{equation}\label{stap}
		\left\{
		\begin{aligned}
			\d\Y(s)&=[-\mu\A\Y(s)-\mathcal{B}(\Y(s))-\beta\mathcal{C}(\Y(s))+
			\f(s,\a(s))]\d s+\d\mathbf{W}(s) ,  \  \text{ in } \ (t,T]\times\H, \\
			\Y(t)&=\y\in\H,
		\end{aligned}
		\right.
	\end{equation}
	where $\f:[0,T]\times\U\to\V$. The following assumption on $\f$ is taken in this article:
	\begin{hypothesis}\label{fhyp}
	    The function $\f:[0,T]\times\U\to\V$ is continuous and there is $R\geq0$ such that 
		\begin{align*}
			\|\f(t,\a)\|_{\V}\leq R, \  \text{ for all } \ t\in[0,T], \ \a\in\U.
		\end{align*}
	\end{hypothesis} 
	
	\subsection{Stochastic optimal control problem}\label{stgform}
	Let us define a cost functional associated with the system \eqref{stap} as
	\begin{align}\label{costF}
		J(t,\y;\a(\cdot))=
		\E\left\{\int_t^T \ell(s,\Y(s;t,\y,\a(\cdot)),\a(s))\d s+ g(\Y(T;t,\y,\a(\cdot)))\right\},
	\end{align}
	where $\ell:[t,T]\times\H\times\U\to\R$ and $g:\H\to\R$ are given measurable functions. The function $\ell$ is the so-called \emph{running cost} and $g$ is the \emph{terminal cost} (see Subsection \ref{stcoexm} for the explicit example of $\ell$ and $g$). Since the system \eqref{stap} has a unique variational solution $\Y(\cdot)$ for any $\a(\cdot)\in\mathscr{U}_t^\nu$ (see Theorem \ref{weLLp}), the cost functional \eqref{costF} is well-defined. 
	The optimal control problem consists of minimizing the cost functional \eqref{costF} over all admissible controls $\a(\cdot)\in\mathscr{U}_t^\nu$. 
	The stochastic optimal control problem \eqref{costF} is referred as the \emph{strong formulation} in the sense that the generalized reference probability space is fixed (see \cite[Chapter 2]{GFAS}).

	\section{Energy estimates and continuous dependence results}\label{engdef}\setcounter{equation}{0}
	In this section, we provide some uniform energy estimates and prove continuous dependence results. We first define the following notions of solutions for the controlled SCBF equations \eqref{stap} (cf. \cite{FGSSA,GFAS}):
	\begin{definition}\label{def-var-strong} Assume that the Hypothesis \ref{trQ1}-\ref{fhyp} be satisfied. Let $\xi$ be a $\mathscr{F}_t^t$-measurable $\H$-valued random variable such that $\E\|\xi\|_{\H}^2<+\infty$, and let $\a(\cdot)\in\mathscr{U}_t^\nu$.
		
		(i) A process $\Y(\cdot)\in M^2_\nu(t,T;\H)$ is called a \emph{variational solution} of \eqref{stap} with initial condition $\Y(t)=\xi$ if 
		$$\E\bigg[\sup_{s\in[t,T]}\|\Y(s)\|_{\H}^2+\int_t^T \|\Y(s)\|_{\V}^2\d s+\int_t^T\|\Y(s)\|_{\wi\L^{r+1}}^{r+1}\d s \bigg]<+\infty,$$ the process $\Y$ having a modification with paths in $\mathrm{C}([0,T];\H)\cap\mathrm{L}^2(0,T;\V)\cap\mathrm{L}^{r+1}(0,T;\wi\L^{r+1})$, $\mathbb{P}-$a.s., 
		and for every $\phi\in\V\cap\wi\L^{r+1}$ and every $s\in[t,T]$, $\P-\text{a.s.}$, we have
		\begin{align*}
			\langle\Y(s),\phi\rangle&=\langle\xi,\phi\rangle+\int_t^s  \langle-\mu\A\Y(\tau)-\mathcal{B}(\Y(\tau))-\beta\mathcal{C}(\Y(\tau))+\f(\tau,\a(\tau)),\phi\rangle\d\tau\nonumber\\&\quad+
			\int_t^s \langle\d\mathbf{W}(r),\phi\rangle.
		\end{align*} 
		 
		(ii) A process $\Y(\cdot)\in M^2_\nu(t,T;\H)$ is called a \emph{strong solution} of \eqref{stap} with initial condition $\Y(t)=\xi$ if 
		$$\E\bigg[\sup_{s\in[t,T]}\|\Y(s)\|_{\V}^2+\int_t^T \|\A\Y(s)\|_{\H}^2\d s+\int_0^T\|\Y(s)\|_{\wi\L^{p(r+1)}}^{r+1}\d s\bigg]<+\infty,$$ where $p\in[2,\infty)$ for $d=2$ and $p=3$ for $d=3$, the process $\Y$ having a modification with paths in $\mathrm{C}([0,T];\V)\cap\mathrm{L}^2(0,T;\V)\cap\mathrm{L}^{r+1}(0,T;\wi\L^{p(r+1)})$, $\mathbb{P}-$a.s., and  for every $s\in[t,T]$, $\P-\text{a.s.}$, we have
		\begin{align*}
			\Y(s)&=\xi+\int_t^s  \left(-\mu\A\Y(\tau)-\mathcal{B}(\Y(\tau))-\beta\mathcal{C}(\Y(\tau))+\f(\tau,\a(\tau))\right)\d\tau\nonumber\\&\quad+
			\int_t^s \d\mathbf{W}(r), \ \text{ in }\H. 
		\end{align*} 
	\end{definition}
	\begin{remark}
		In the SPDE literature, a \emph{variational solution} is also called a probabilistically strong (analytically weak) solution. Moreover, the definition of strong solution given in Definition \ref{def-var-strong} coincides with the strong solution definition in the PDE sense.   
		\end{remark}
    
   We now recall some existence, uniqueness and continuous dependence results for the variational and strong solutions of the controlled SCBF equations \eqref{stap} and derive some energy estimates.
     \begin{proposition}\label{weLLp}
     	Let $t\in[0,T]$ and let $p\geq2$ be fixed. Let $\nu=(\Omega,\mathscr{F},\{\mathscr{F}_s^t\}_{s\geq t},\P,
     	\mathbf{W})$, be a generalized reference probability space, and let Hypothesis \ref{trQ1}-\ref{fhyp} be satisfied. Let $\xi$ be a $\mathscr{F}_t^t-$measurable $\H-$valued random variable such that $\E\left[\|\xi\|_{\H}^p\right]<+\infty$, and let $\a(\cdot)\in\mathscr{U}_t^\nu$. 
     	
     	(i) There exists a \emph{unique variational solution} $\Y(\cdot)=\Y(\cdot;t,\xi,\a(\cdot))$ of \eqref{stap} with initial condition $\Y(t)=\xi$. Moreover, $\Y(\cdot)$ has continuous trajectories in $\H$ and satisfies the following energy estimates	 for $s\in[t,T]$:
     	 \begin{align}\label{eqn-conv-1}
     		&\E\left[\|\Y(s)\|_{\H}^p\right]+p\mu\E\bigg[\int_t^{s}\|\nabla\Y(\tau)\|_{\H}^2\|\Y(\tau)\|_{\H}^{p-2}\d\tau\bigg]+ p\beta\E\bigg[\int_t^{s}
     		\|\Y(\tau)\|_{\widetilde{\L}^{r+1}}^{r+1}
     		\|\Y(\tau)\|_{\H}^{p-2}\d\tau\bigg]\nonumber\\&\leq \E\left[\|\xi\|_{\H}^p\right]+C(p,r,R,\Q,\alpha)(s-t),
     	\end{align}
     	and 
     	\begin{align}\label{vsee1}
     		&\E\bigg[\sup\limits_{s\in[t,T]}\|\Y(s)\|_{\H}^p\bigg]+\E\bigg[\int_t^T \|\nabla\Y(\tau)\|_{\H}^2 \|\Y(\tau)\|_{\H}^{p-2}\d\tau\bigg]+\E\bigg[\int_t^T \|\Y(\tau)\|_{\wi\L^{r+1}}^{r+1} \|\Y(\tau)\|_{\H}^{p-2}\d\tau\bigg]
     		\nonumber\\&\leq
     		\E\left[\|\xi\|_{\H}^p\right]+C(p,r,R,\Q,\mu,\alpha,\beta,T). 		
     	\end{align}
     	
     	(ii) If $\E\left[\|\nabla\xi\|_{\H}^p\right]<+\infty$, then the {variational solution} $\Y(\cdot)=\Y(\cdot;t,\xi,\a(\cdot))$ is a  \emph{strong solution} with continuous trajectories in $\V$. Moreover, for $s\in[t,T]$, we have following energy estimates:
     	  \begin{align}\label{eqn-conv-2}
     		&\E\left[\|\nabla\Y(s)\|_{\H}^p\right]+\frac{p\mu}{2} \E\bigg[\int_t^{s}\|\A\Y(\tau)\|_{\H}^2\|\nabla\Y(\tau)\|_{\H}^{p-2}\d\tau\bigg]\nonumber\\&\quad+\frac{3p\beta}{4} \E\bigg[\int_t^{s}
     		\||\Y(\tau)|^{\frac{r-1}{2}}\nabla\Y(\tau)\|_{\H}^{2} 
     		\|\nabla\Y(\tau)\|_{\H}^{p-2}\d\tau\bigg]\nonumber\\&\leq \big(\E\|\nabla\xi\|_{\H}^p+C(p,\mathrm{Tr}(\Q_1),R,\alpha)(s-t)\big) e^{p\varrho (s-t)},
     	\end{align}
     	and 
     		\begin{align}\label{ssee1}
     		&\E\left[\sup\limits_{s\in[t,T]}\|\Y(s)\|_{\V}^p\right]+
     		   \E\bigg[\int_t^T \|\A\Y(\tau)\|_{\H}^2 \|\nabla\Y(\tau)\|_{\H}^{p-2}\d\tau\bigg]\nonumber\\&\quad+\E\bigg[\int_t^T \||\Y(\tau)|^{\frac{r-1}{2}}\nabla\Y(\tau)\|_{\H}^2 \|\nabla\Y(\tau)\|_{\H}^{p-2}\d\tau\bigg]
     		\nonumber\\&\leq
     		\E\left[\|\nabla\xi\|_{\H}^p\right]+C(p,r,R,\Q_1,\mu,\beta,T).  		
     	\end{align}

(iii) If $\nu_1:=(\Omega_1,\mathscr{F},\{\mathscr{F}_s^t\}_{s\geq t},\P_1,\mathbf{W}_1)$ is another generalized reference probability space, $\xi_1$ is an $\mathscr{F}_t^{t,\nu_1}-$measurable $\H$ valued random variable such that $\E[\|\xi_1\|_{\H}^p]<\infty$, $\a_1(\cdot)\in\mathscr{U}_t^{\nu_1}$, and 
\begin{align*}
	\mathcal{L}_{\P_1}\big(\xi_1,\a_1(\cdot),\mathbf{W}_1(\cdot)\big)=
	\mathcal{L}_{\P}\big(\xi,\a(\cdot),\mathbf{W}(\cdot)\big),
\end{align*} 
then
\begin{align*}
	\mathcal{L}_{\P_1}\big(\a_1(\cdot),\Y_1(\cdot)\big)=
	\mathcal{L}_{\P}\big(\a(\cdot),\Y(\cdot)\big),
\end{align*}
where $\Y_1(\cdot)=\Y_1(\cdot;\xi_1,\a_1(\cdot))$ is the variational solution of \eqref{stap} in $\nu_1$ with control $\a_1(\cdot)$ and initial condition $\xi_1$. 
     \end{proposition}
     
     \begin{proof}
     	\vskip 0.25cm
     	\noindent
     	\emph{\textbf{Proof of (i):}}
     	The existence and uniqueness results of variational  solutions of the system \eqref{stap} is available in \cite[Theorem 3.7]{MTM8} (also see \cite[Theorem 3.7]{KKMTM}). Therefore, here we are providing a proof of the energy estimate \eqref{vsee1} only. Let us consider a function $\Psi_1(\y)=\|\y\|_{\H}^p=(\|\y\|_{\H}^2)^{\frac{p}{2}}, \ \y\in\H$. Then, $\Psi_1\in\mathrm{C}^2(\H)$ and it has Fr\'echet derivatives
     	\begin{align*}
     		\Psi_1'(\y)\z&=p\|\y\|_{\H}^{p-2}(\y,\z), \  \text{ for } \ \z\in\H,\nonumber\\
     		\Psi_1''(\y)(\z_1,\z_2)&=p(p-2)\|\y\|_{\H}^{p-4} \underbrace{(\y,\z_1)(\y,\z_2)}_{=(\y\otimes\y)(\z_1,\z_2)}+p\|\y\|_{\H}^{p-2}(\z_1,\z_2),  \  \text{ for } \ \z_1,\z_2\in\H,
     	\end{align*}
     	with the property 
     	\begin{align*}
     		\mathrm{Tr}(\Q\Psi_1''(\y))=p(p-1)\|\y\|_{\H}^{p-2}
     		\mathrm{Tr}(\Q).
     	\end{align*}
     	Let us define a sequence of stopping times
     	\begin{align}\label{stopt1}
     		\theta_N=\inf\{s\geq t:\|\Y(s)\|_{\H}\geq N\}, \ \text{ for any } \ N\in\mathbb{N}.
     	\end{align}
      On applying It\^o's formula (see \cite[Theorem A.1, pp. 294]{ZBSP}, \cite[Theorem 4.3, pp. 1809]{ZBSP1}) to the function $(\|\cdot\|_{\H}^2)^{\frac{p}{2}}$ and to the process $\Y(\cdot)$,  we find  $\P-\text{a.s.},$
      \begin{align}\label{vse1}
      	&	\|\Y(s\land\theta_N)\|_{\H}^p+p\mu\int_t^{s\land\theta_N}\|\nabla\Y(\tau)\|_{\H}^2\|\Y(\tau)\|_{\H}^{p-2}\d\tau+p\alpha\int_t^{s\land\theta_N}\|\Y(\tau)\|_{\H}^p \d\tau\nonumber\\&\quad+ p\beta\int_t^{s\land\theta_N}
      	\|\Y(\tau)\|_{\widetilde{\L}^{r+1}}^{r+1}
      	\|\Y(\tau)\|_{\H}^{p-2}\d\tau\nonumber\\&= \|\xi\|_{\H}^p+p\int_t^{s\land\theta_N}(\f(\tau),\Y(\tau)) \|\Y(\tau)\|_{\H}^{p-2}\d\tau+\frac{p(p-1)\mathrm{Tr}(\Q)}{2}\int_t^{s\land\theta_N}\|\Y(\tau)\|_{\H}^{p-2}\d\tau\nonumber\\&\quad+M_{s\land\theta_N}^p,
      \end{align}
     	 where $M_{s\land\theta_N}^p=p\int_t^{s\land\theta_N}\|\Y(\tau)\|_{\H}^{p-2}\left(\Y(\tau),\d\W(\tau)\right)$ is a martingale. By using Hypothesis \ref{fhyp} and Young's inequality, we calculate
     	\begin{align}
     	\big|(\f,\Y)\big|\|\Y\|_{\H}^{p-2}&\leq
     	\frac{\alpha}{4}\|\Y\|_{\H}^p+C(p,R,\alpha),\label{fcal1}\\
     	\text{ and } 
     	\frac{(p-1)\mathrm{Tr}(\Q)}{2}\|\Y\|_{\H}^{p-2}&\leq
     	\frac{\alpha}{4}\|\Y\|_{\H}^p+C(p,\mathrm{Tr}(\Q),\alpha).
     	\label{trcal1}
     	\end{align}
     	On plugging \eqref{fcal1}-\eqref{trcal1} into \eqref{vse1}, we obtain for all $s\in[t,T]$, $\P-\text{a.s.},$
     	   	\begin{align}\label{vse2.1}
     	   	&	\|\Y(s\land\theta_N)\|_{\H}^p+p\mu\int_t^{s\land\theta_N}\|\nabla\Y(\tau)\|_{\H}^2\|\Y(\tau)\|_{\H}^{p-2}\d\tau+\frac{p\alpha}{2}\int_t^{s\land\theta_N}\|\Y(\tau)\|_{\H}^p \d\tau\nonumber\\&\quad+ p\beta\int_t^{s\land\theta_N}
     	   	\|\Y(\tau)\|_{\widetilde{\L}^{r+1}}^{r+1}
     	   	\|\Y(\tau)\|_{\H}^{p-2}\d\tau\nonumber\\&\leq \|\xi\|_{\H}^p+M_{s\land\theta_N}^p+C(p,\mathrm{Tr}(\Q),R,\alpha)(s-t).
     	   \end{align}
     	   On taking the expectation in \eqref{vse2.1}, we deduce for all $s\in[t,T]$
     	   \begin{align}\label{vse2.2}
     	   	&\E\left[\|\Y(s\land\theta_N)\|_{\H}^p\right]+p\mu\E\bigg[\int_t^{s\land\theta_N}\|\nabla\Y(\tau)\|_{\H}^2\|\Y(\tau)\|_{\H}^{p-2}\d\tau\bigg]+\frac{p\alpha}{2}\E\bigg[\int_t^{s\land\theta_N}\|\Y(\tau)\|_{\H}^p \d\tau\bigg]\nonumber\\&\quad+ p\beta\E\bigg[\int_t^{s\land\theta_N}
     	   	\|\Y(\tau)\|_{\widetilde{\L}^{r+1}}^{r+1}
     	   	\|\Y(\tau)\|_{\H}^{p-2}\d\tau\bigg]\nonumber\\&\leq \E\left[\|\xi\|_{\H}^p\right]+C(p,\mathrm{Tr}(\Q),R,\alpha)T,
     	   \end{align}
     	   for all $t\in[0,T]$. Note that for the indicator function $\mathds{1}$, we have 
     	   $$\E\left[\mathds{1}_{\{\theta_N<s\}}\right]=\mathbb{P}\Big\{\omega\in\Omega:\theta_N(\omega)<s\Big\},$$
     	   and using \eqref{stopt1}, we obtain 
     	   \begin{align}\label{vse14.11}
     	   	\E\left[\|\Y(s\land\theta_N)\|_{\H}^p\right]&=
     	   	\E\left[\|\Y(\theta_N)\|_{\H}^p\mathds{1}_{\{\theta_N<s\}}\right]+\E\left[\|\Y(s)\|_{\H}^p
     	   	\mathds{1}_{\{\theta_N\geq s\}}\right]\nonumber\\&\geq
     	   	\E\left[\|\Y(\theta_N)\|_{\H}^p\mathds{1}_{\{\theta_N<s\}}\right]
     	   	\geq
     	   	N^p\mathbb{P}\Big\{\omega\in\Omega:\theta_N<s\Big\}.
     	   \end{align}
     	   Then, by the application of Markov's inequality and using \eqref{vse2.2}, we estimate
     	   \begin{align*}
     	   	\P\Big\{\omega\in\Omega:\theta_N<s\Big\}&\leq
     	   	\frac{1}{N^p}\E\left[\|\Y(s\land\theta_N)\|_{\H}^p\right]\leq
     	   	\frac{1}{N^p}
     	   	\big(\E\|\xi\|_{\H}^p+C(p,\mathrm{Tr}(\Q),R,\alpha)T\big).
     	   \end{align*}
     	   Hence, we have
     	   \begin{align*}
     	   	\lim_{N\to\infty}\P\Big\{\omega\in\Omega:\theta_N<s\Big\}=0, \ \textrm{
     	   		for all }\ s\in [t,T],
     	   \end{align*}
     	   and therefore $s\land\theta_N\to s$ as $N\to\infty$, $\P-$a.s.
     	   Taking limit $N\to\infty$ in \eqref{vse2.2} and using the \emph{monotone convergence theorem}, we arrive at  for all $s\in[t,T]$,
     	   \begin{align*}
     	   	&\E\left[\|\Y(s)\|_{\H}^p\right]+p\mu\E\bigg[\int_t^{s}\|\nabla\Y(\tau)\|_{\H}^2\|\Y(\tau)\|_{\H}^{p-2}\d\tau\bigg]+\frac{p\alpha}{2}\E\bigg[\int_t^{s}\|\Y(\tau)\|_{\H}^p \d\tau\bigg]\nonumber\\&\quad+ p\beta\E\bigg[\int_t^{s}
     	   	\|\Y(\tau)\|_{\widetilde{\L}^{r+1}}^{r+1}
     	   	\|\Y(\tau)\|_{\H}^{p-2}\d\tau\bigg]\nonumber\\&\leq \E\left[\|\xi\|_{\H}^p\right]+C(p,\mathrm{Tr}(\Q),R,\alpha)(s-t),
     	   \end{align*}
     	   which completes the proof of \eqref{eqn-conv-1}.

     	  Let us now take supremum from $t$ to $T\land\theta_N$ in \eqref{vse2.2} and then taking the expectation to obtain
     	   \begin{align}\label{vse2.3}
     	   	&\E\bigg[\sup\limits_{s\in[t,T\land\theta_N]}\|\Y(s)\|_{\H}^p\bigg]+p\mu\E\bigg[\int_t^{T\land\theta_N}\|\nabla\Y(\tau)\|_{\H}^2\|\Y(\tau)\|_{\H}^{p-2}\d\tau\bigg]\nonumber\\&\quad+\frac{p\alpha}{2}\E\bigg[\int_t^{T\land\theta_N}\|\Y(\tau)\|_{\H}^p \d\tau\bigg]+ p\beta\E\bigg[\int_t^{T\land\theta_N}
     	   	\|\Y(\tau)\|_{\widetilde{\L}^{r+1}}^{r+1}
     	   	\|\Y(\tau)\|_{\H}^{p-2}\d\tau\bigg]\nonumber\\&\leq \E\left[\|\xi\|_{\H}^p\right]+p\E\bigg[\sup\limits_{s\in[t,T\land\theta_N]}\bigg|
     	   	\int_t^s\|\Y(\tau)\|_{\H}^{p-2}\left(\Y(\tau),\d\W(\tau)\right)\bigg|\bigg]
     	   	+C(p,\mathrm{Tr}(\Q),R,\alpha)T.
     	   \end{align} 
      By the application of Burkholder-Davis-Gundy inequality (see \cite[Theorem 1.1]{CMMR}),  H\"older's and Young's inequalities, we calculate
      \begin{align}\label{vse6}
      	&\E\bigg[\sup\limits_{s\in[t,T\land\theta_N]}\bigg|\int_t^s\|\Y(\tau)\|_{\H}^{p-2}
      	\big(\Y(\tau),\d\W(\tau)\big)\bigg|\bigg]\nonumber\\&\leq C_p
        \E\bigg[\int_t^{T\land\theta_N}\mathrm{Tr}(\Q)\|\Y(\tau)\|_{\H}^{2(p-1)}\d\tau
        \bigg]^{\frac12}\nonumber\\&\leq C_p
        \E\bigg[\sup\limits_{s\in[t,T\land\theta_N]}\|\Y(\tau)\|_{\H}^{p-1}
        \left(\int_t^{T\land\theta_N}\mathrm{Tr}(\Q)\d\tau\right)^{\frac12}
        \bigg]\nonumber\\&\leq
        \frac12\E\bigg[\sup\limits_{s\in[t,T\land\theta_N]}\|\Y(s)\|_{\H}^p
        \bigg]+C_p\big(\mathrm{Tr}(\Q)T\big)^{\frac{p}{2}}.
      \end{align}
      Substituting \eqref{vse6} into \eqref{vse2.3} yields that
     \begin{align}\label{vse7}
     	&\frac12\E\bigg[\sup\limits_{s\in[t,T\land\theta_N]}\|\Y(s)\|_{\H}^p\bigg]+p\mu\E\bigg[\int_t^{T\land\theta_N}\|\nabla\Y(\tau)\|_{\H}^2\|\Y(\tau)\|_{\H}^{p-2}\d\tau\bigg] \nonumber\\&\quad+\frac{p\alpha}{2}\E\bigg[\int_t^{T\land\theta_N}\|\Y(\tau)\|_{\H}^p \d\tau\bigg]+ p\beta\E\bigg[\int_t^{T\land\theta_N}
     	\|\Y(\tau)\|_{\widetilde{\L}^{r+1}}^{r+1}
     	\|\Y(\tau)\|_{\H}^{p-2}\d\tau\bigg]\nonumber\\&\leq \E\left[\|\xi\|_{\H}^p\right]+C(p,\mathrm{Tr}(\Q),R,\alpha,T).
     \end{align}
     Since $T\wedge \theta_N\to T $ as $N\to\infty$, using the monotone convergence theorem, one can finally conclude the proof of \eqref{vsee1}.
     \vskip 0.25cm
     \noindent
     \emph{\textbf{Proof of (ii):}} The existence and uniqueness for strong  solutions of the system \eqref{stap} is established in \cite[Theorem 3.11]{MTM8}. 
     Let us now operate by $\A^{\frac12}$ in \eqref{stap} to obtain the following stochastic differential satisfied the stochastic process $\A^{\frac12}\Y(\cdot)$:
     \begin{align*}
     	\d\A^{\frac12}\Y(s)+\A^{\frac12}[\mu\A\Y(s)+\mathcal{B}(\Y(s))
     	+\alpha\Y(s)+\beta\mathcal{C}(\Y(s))]\d s=\A^{\frac12}\f(s,a(s))\d s+ \A^{\frac12}\d\mathbf{W}(s),
     \end{align*}
      for a.e. $s\in[t,T]$. Let us now consider the function $\Psi_2(\y)=\|\nabla\y\|_{\H}^p=(\|\nabla\y\|_{\H}^2)^{\frac{p}{2}}, \ \y\in\V$. Then, $\Psi_2\in\mathrm{C}^2(\V)$ and it has Fr\'echet derivatives
      \begin{align*}
      	\Psi_2'(\y)\z&=p\|\nabla\y\|_{\H}^{p-2}\langle\A\y,\z\rangle, \  \text{ for } \ \z\in\V,\nonumber\\
      	\Psi_2''(\y)(\z_1,\z_2)&=p(p-2)\|\nabla\y\|_{\H}^{p-4} {\langle\A\y,\z_1\rangle\langle\A\y,\z_2\rangle}
      	+p\|\nabla\y\|_{\H}^{p-2}\langle\A\z_1,\z_2\rangle,  \  \text{ for } \ \z_1,\z_2\in\V,
      \end{align*}
      with the property 
      \begin{align*}
      	\mathrm{Tr}(\Q\Psi_2''(\y))=p(p-1)\|\nabla\y\|_{\H}^{p-2}
      	\mathrm{Tr}(\Q_1).
      \end{align*}
      We define a sequence of  stopping times
      \begin{align*}
      	\wi\theta_N=\inf\{s\geq t:\|\nabla\Y(s)\|_{\H}>N\}, \ \text{ for any } \ N>0.
      \end{align*}
      Similar to the proof of part (i), let us now apply It\^o's formula (see \cite[Theorem A.1, pp. 294]{ZBSP}, \cite[Theorem 4.3, pp. 1809]{ZBSP1}) to the function $(\|\cdot\|_{\H}^2)^{\frac{p}{2}}$ and to the process $\A^{\frac12}\Y(\cdot)$ to find for all $s\in[0,T]$, $\P-\text{a.s.},$
      \begin{align}\label{vse7.1}
      	&\|\nabla\Y(s\land\wi\theta_N)\|_{\H}^p+p\mu
      	\int_t^{s\land\wi\theta_N}\|\A\Y(\tau)\|_{\H}^2\|\nabla\Y(\tau)\|_{\H}^{p-2}\d\tau\nonumber\\&\quad+p\alpha\int_t^{s\land\wi\theta_N}\|\nabla\Y(\tau)\|_{\H}^p \d\tau+p\beta \int_t^{s\land\wi\theta_N}
      	\big(\mathcal{C}(\Y(\tau)),\A\Y(\tau)\big)
      	\|\nabla\Y(\tau)\|_{\H}^{p-2}\d\tau\nonumber\\&= \|\nabla\xi\|_{\H}^p+p\int_t^{s\land\wi\theta_N} \big(\A^{\frac12}\f(\tau),\A^{\frac12}\Y(\tau)\big)
      	\|\nabla\Y(\tau)\|_{\H}^{p-2}\d\tau\nonumber\\&\quad-\int_t^{s\land\wi\theta_N}\big(\mathcal{B}(\Y(\tau)),\A\Y(\tau)\big)\|\nabla\Y(\tau)\|_{\H}^{p-2}\d\tau+
      	\frac{p(p-1)\mathrm{Tr}(\Q_1)}{2}\int_t^{s\land\wi\theta_N}
      	\|\nabla\Y(\tau)\|_{\H}^{p-2}\d\tau\nonumber\\&\quad+M_{s\land\wi\theta_N}^p, 
      \end{align}
      where ${M}_{s\land\wi\theta_N}^p=
      p\int_t^{s\land\wi\theta_N}\|\nabla\Y(\tau)\|_{\H}^{p-2} \left(\A\Y(\tau),\d\W(\tau)\right)$ is a martingale.
       From the Hypothesis \ref{fhyp}, Young's inequality, equality \eqref{torusequ} and the estimate \eqref{syymB3}, we calculate following:
      \begin{align}
      	\big(\A^{\frac12}\f(\tau),\A^{\frac12}
      	\Y(\tau)\big)\|\nabla\Y(\tau)\|_{\H}^{p-2}&\leq
      	R\|\nabla\Y(\tau)\|_{\H}^{p-1}\leq\frac{\alpha}{4}\|\nabla\Y(\tau)\|_{\H}^p+C(p,R,\alpha),\label{vse8}\\
      		\frac{(p-1)\mathrm{Tr}(\Q_1)}{2}\|\nabla\Y(\tau)\|_{\H}^{p-2}&\leq
      		\frac{\alpha}{4}\|\nabla\Y(\tau)\|_{\H}^p+C(p,\mathrm{Tr}(\Q_1),R,\alpha),\label{vse8.1}\\
      		\big(\mathcal{C}(\Y(\tau)),\A\Y(\tau)\big)&\geq
      		\||\Y(\tau)|^{\frac{r-1}{2}}\nabla\Y(\tau)\|_{\H}^{2},\label{vse8.2}
      		\\ |(\mathcal{B}(\Y),\A\Y)|&\leq\frac{\mu}{2}\|\A\Y\|_{\H}^2
      		+\frac{\beta}{4}\||\Y|^{\frac{r-1}{2}}\nabla\Y\|_{\H}^2 +\varrho\|\nabla\Y\|_{\H}^2.\label{vse8.3}
      \end{align}
      Utilizing the inequalities \eqref{vse8}-\eqref{vse8.3} in \eqref{vse7.1}, we obtain for all $s\in[0,T]$, $\P-\text{a.s.},$
        \begin{align}\label{vse11}
        	&\|\nabla\Y(s\land\wi\theta_N)\|_{\H}^p+\frac{p\mu}{2}
        	\int_t^{s\land\wi\theta_N}\|\A\Y(\tau)\|_{\H}^2\|\nabla\Y(\tau)\|_{\H}^{p-2}\d\tau\nonumber\\&\quad+\frac{p\alpha}{2}\int_t^{s\land\wi\theta_N}\|\nabla\Y(\tau)\|_{\H}^p \d\tau+\frac{3p\beta}{4} \int_t^{s\land\wi\theta_N}
        	\||\Y(\tau)|^{\frac{r-1}{2}}\nabla\Y(\tau)\|_{\H}^{2} 
        	\|\nabla\Y(\tau)\|_{\H}^{p-2}\d\tau\nonumber\\&\leq \|\nabla\xi\|_{\H}^p+M_{s\land\wi\theta_N}^p+
        	p\varrho\int_t^{s\land\wi\theta_N}\|\nabla\Y(\tau)\|_{\H}^p\d\tau+C(p,\mathrm{Tr}(\Q_1),R,\alpha)(s-t). 
        \end{align}
        On taking  expectation in \eqref{vse11}, we deduce  for all $s\in[t,T]$
        \begin{align}\label{vse12}
        &\E\left[\|\nabla\Y(s\land\wi\theta_N)\|_{\H}^p\right]+\frac{p\mu}{2} \E\bigg[\int_t^{s\land\wi\theta_N}\|\A\Y(\tau)\|_{\H}^2\|\nabla\Y(\tau)\|_{\H}^{p-2}\d\tau\bigg]\nonumber\\&\quad+\frac{p\alpha}{2}
        \E\bigg[\int_t^{s\land\wi\theta_N}\|\nabla\Y(\tau)\|_{\H}^p \d\tau\bigg]+\frac{3p\beta}{4} \E\bigg[\int_t^{s\land\wi\theta_N}
        \||\Y(\tau)|^{\frac{r-1}{2}}\nabla\Y(\tau)\|_{\H}^{2} 
        \|\nabla\Y(\tau)\|_{\H}^{p-2}\d\tau\bigg]\nonumber\\&\leq \E\left[\|\nabla\xi\|_{\H}^p\right]+p\varrho\E\bigg[\int_t^{s\land\wi\theta_N}\|\nabla\Y(\tau)\|_{\H}^p\d\tau\bigg]+C(p,\mathrm{Tr}(\Q_1),R,\alpha)(s-t),
        \end{align}
        for all $t\in[0,T]$. By the application of Fubini's theorem, we write from \eqref{vse12}
        \begin{align}\label{vse13}
       & \E\left[\|\nabla\Y(s\land\wi\theta_N)\|_{\H}^p\right]\nonumber\\&\leq
        \E\left[\|\nabla\xi\|_{\H}^p\right]+p\varrho\E\bigg[\int_t^{s\land\wi\theta_N}\|\nabla\Y(\tau)\|_{\H}^p\d\tau\bigg]+C(p,\mathrm{Tr}(\Q_1),R,\alpha)(s-t)
        \nonumber\\&=
        \E\left[\|\nabla\xi\|_{\H}^p\right]+p\varrho\int_t^{s}\E\left[\|\nabla\Y(\tau\land\wi\theta_N)\|_{\H}^p\right]\d\tau+C(p,\mathrm{Tr}(\Q_1),R,\alpha)(s-t).
        \end{align}
        On employing Gr\"onwall's inequality, we conclude from \eqref{vse13} that for all $s\in[t,T]$
        \begin{align}\label{vse14}
        	\E\left[\|\nabla\Y(s\land\wi\theta_N)\|_{\H}^p\right]\leq
        	\big(\E\left[\|\nabla\xi\|_{\H}^p\right]+C(p,\mathrm{Tr}(\Q_1),R,\alpha)(s-t)\big) e^{p\varrho (s-t)},
        \end{align}
      for all $t\in[0,T]$. 
On arguing similalry, as we discussed in the proof of the previous part, and then taking the limit as $N\to\infty$, we finally obtain  for all $s\in[t,T]$
      \begin{align*}
      	 &\E\left[\|\nabla\Y(s)\|_{\H}^p\right]+\frac{p\mu}{2} \E\bigg[\int_t^{s}\|\A\Y(\tau)\|_{\H}^2\|\nabla\Y(\tau)\|_{\H}^{p-2}\d\tau\bigg]\nonumber\\&\quad+\frac{p\alpha}{2}
      	\E\bigg[\int_t^{s}\|\nabla\Y(\tau)\|_{\H}^p \d\tau\bigg]+\frac{3p\beta}{4} \E\bigg[\int_t^{s}
      	\||\Y(\tau)|^{\frac{r-1}{2}}\nabla\Y(\tau)\|_{\H}^{2} 
      	\|\nabla\Y(\tau)\|_{\H}^{p-2}\d\tau\bigg]\nonumber\\&\leq \big(\E\|\nabla\xi\|_{\H}^p+C(p,\mathrm{Tr}(\Q_1),R,\alpha)(s-t)\big) e^{p\varrho (s-t)},
      \end{align*}
      which completes the proof of \eqref{eqn-conv-2}.

  Let us now take supremum over $t$ to $T\land\wi\theta_N$ in \eqref{vse12} followed by  expectation, we find 
  \begin{align}\label{vse15}
  &\E\bigg[\sup\limits_{s\in[t,T\land\theta_N]}\|\nabla\Y(s)\|_{\H}^p\bigg]+\frac{p\mu}{2}\E\bigg[\int_t^{T\land\theta_N}\|\A\Y(\tau)\|_{\H}^2\|\nabla\Y(\tau)\|_{\H}^{p-2}\d\tau\bigg]\nonumber\\&\quad+\frac{p\alpha}{2}\E\bigg[\int_t^{T\land\theta_N}
  \|\nabla\Y(\tau)\|_{\H}^p \d\tau\bigg]+ \frac{3p\beta}{4}\E\bigg[\int_t^{T\land\theta_N}
  \||\Y(\tau)|^{\frac{r-1}{2}}\nabla\Y(\tau)\|_{\H}^{2} 
  \|\nabla\Y(\tau)\|_{\H}^{p-2}\d\tau\bigg]\nonumber\\&\leq \E\left[\|\nabla\xi\|_{\H}^p\right]+p\varrho\E\bigg[\int_t^{T\land\wi\theta_N}\|\nabla\Y(\tau)\|_{\H}^p\d\tau\bigg]\nonumber\\&\quad+ p\E\bigg[\sup\limits_{s\in[t,T\land\theta_N]}\bigg|
  \int_t^s\|\nabla\Y(\tau)\|_{\H}^{p-2}\left(\A\Y(\tau),\d\W(\tau)\right)\bigg|\bigg]+C(p,\mathrm{Tr}(\Q_1),R,\alpha,T).
  \end{align}
   Similar to \eqref{vse6}, we calculate
   \begin{align}\label{vse15.1}
   	&\E\bigg[\sup\limits_{s\in[t,T\land\theta_N]}\bigg|\int_t^s\|\nabla\Y(\tau)\|_{\H}^{p-2}
   	\big(\A\Y(\tau),\d\W(\tau)\big)\bigg|\bigg]
   \nonumber\\&	\leq
   	\frac12\E\bigg[\sup\limits_{s\in[t,T\land\theta_N]}\|\nabla\Y(s)\|_{\H}^p
   	\bigg]+C_p\big(\mathrm{Tr}(\Q_1)T\big)^{\frac{p}{2}}.
   \end{align}
      On substituting \eqref{vse15.1} into \eqref{vse15}, we obtain
      \begin{align*}
    	&\frac12\E\bigg[\sup\limits_{s\in[t,T\land\theta_N]}\|\nabla\Y(s)\|_{\H}^p\bigg]+\frac{p\mu}{2}\E\bigg[\int_t^{T\land\theta_N}\|\A\Y(\tau)\|_{\H}^2\|\nabla\Y(\tau)\|_{\H}^{p-2}\d\tau\bigg]\nonumber\\&\quad+\frac{p\alpha}{2}\E\bigg[\int_t^{T\land\theta_N}
      	\|\nabla\Y(\tau)\|_{\H}^p \d\tau\bigg]+ \frac{3p\beta}{4}\E\bigg[\int_t^{T\land\theta_N}
      	\||\Y(\tau)|^{\frac{r-1}{2}}\nabla\Y(\tau)\|_{\H}^{2} 
      	\|\nabla\Y(\tau)\|_{\H}^{p-2}\d\tau\bigg]\nonumber\\&\leq \E\left[\|\nabla\xi\|_{\H}^p\right]+p\varrho\E\bigg[\int_t^{T\land\wi\theta_N}\|\nabla\Y(\tau)\|_{\H}^p\d\tau\bigg]+C(p,\mathrm{Tr}(\Q_1),R,\alpha,T).
      \end{align*}
      By using  Gr\"onwall's inequality, and then passing to the limit $N\to\infty$ together with monotone convergence theorem, we finally obtain for all $s\in[0,T]$
      \begin{align*}
       &\E\bigg[\sup\limits_{s\in[t,T]}\|\nabla\Y(s)\|_{\H}^p\bigg]+
         \E\bigg[\int_t^{T}\|\A\Y(\tau)\|_{\H}^2\|\nabla\Y(\tau)\|_{\H}^{p-2}\d\tau\bigg]+\E\int_t^{T}
         \|\nabla\Y(\tau)\|_{\H}^p \d\tau\nonumber\\&\quad+ \E\bigg[\int_t^{T}
       \||\Y(\tau)|^{\frac{r-1}{2}}\nabla\Y(\tau)\|_{\H}^{2} 
       \|\nabla\Y(\tau)\|_{\H}^{p-2}\d\tau\bigg]\nonumber\\&\leq \big(\E\left[\|\nabla\xi\|_{\H}^p\right]+C(p,\mathrm{Tr}(\Q_1),R,\alpha,T)\big)e^{2p\varrho T},
      \end{align*}
      for all $t\in[0,T]$. 
      \vskip 0.25cm
      \noindent
       \emph{\textbf{Proof of (iii):}} From \cite[Remark 1.10]{MRBZ}, we infer that the existence of a pathwise unique variational solution of the system \eqref{stap}  ensures the weak uniqueness, which completes the proof of (iii). 
     \end{proof}
     
     \begin{proposition}\label{cts-dep-soln}[Continuous dependence of solutions]
     Let $t\in[0,T]$ and $p\geq2$ be fixed. Let $\nu=(\Omega,\mathscr{F},\{\mathscr{F}_s^t\}_{s\geq t},\P,
     \mathbf{W})$ be a generalized reference probability space, and let Hypothesis \ref{trQ1}-\ref{fhyp} be satisfied. Let $\xi$ and $\eta$ be $\mathscr{F}_t^t-$measurable $\H-$valued random variables such that $\E\big[\|\xi\|_{\H}^p+\|\eta\|_{\H}^p\big]<+\infty$, and let $\a(\cdot)\in\mathscr{U}_t^\nu$. Then:
     
     (i) There exists a constant $C$ independent of $t,\xi,\eta,\a(\cdot)$ and $\nu$ such that for all $s\in[t,T]$, $\P-$a.s., we have
     \begin{align}\label{ctsdep0}
     	&\|\Y_1(s)-\Y_2(s)\|_{\H}^2+\int_t^s \|\nabla(\Y_1-\Y_2)(\tau)\|_{\H}^2\d\tau+\int_t^s \|\Y_1(\tau)-\Y_2(\tau)\|_{\wi\L^{r+1}}^{r+1}\d\tau\nonumber\\&\leq
       \|\xi-\eta\|_{\H}^2 e^{C(s-t)},
     \end{align} 
     where $\Y_1(\cdot)=\Y_1(\cdot;t,\xi,\a(\cdot))$ and $\Y_2(\cdot)=\Y_2(\cdot;t,\eta,\a(\cdot))$ are two variational  solutions of \eqref{stap} with initial conditions $\Y_1(t)=\xi$ in $\H$ and $\Y_2(t)=\eta$ in $\H$. 
     
     (ii) If $\|\y\|_{\V}\leq R_1$, where $R_1$ is arbitrary, then there exists a constant $$C=C(p,\mu,\alpha,\beta,T,R,R_1,\mathrm{Tr}(\Q),\mathrm{Tr}(\Q_1))$$ such that for all $s\in[t,T]$, 
     \begin{align}\label{ctsdep0.1}
     	\E\left[\|\Y(s)-\y\|_{\H}^p\right]\leq C(p,\mu,\alpha,\beta,T,R,R_1, \mathrm{Tr}(\Q),\mathrm{Tr}(\Q_1))(s-t), 
     \end{align} 
   where $\Y(\cdot)=\Y(\cdot;t,\y,\a(\cdot))$.
     
     (iii) For every initial condition $\y\in\V,$ there exists a modulus $\omega$, independent of the reference probability spaces $\nu$ and controls $\a(\cdot)\in\mathscr{U}_t^\nu$, such that
     \begin{align}\label{ctsdep0.2}
     	\E\left[\|\Y(s)-\y\|_{\V}^p\right]\leq\omega_{\y}(s-t), \ \text{ for all } \ s\in[t,T],
     \end{align}
     where $\Y(\cdot)=\Y(\cdot;t,\y,\a(\cdot))$.
     \end{proposition}
     
     \begin{proof}
     	\vskip 0.25cm
     \noindent
     \emph{\textbf{Proof of (i):}}  Let us set $\Z(\cdot):=\Y_1(\cdot)-\Y_2(\cdot)$. Then $\Z(\cdot)$ satisfies the following system for each fixed $\omega\in\Omega$:
     	\begin{equation}\label{ctsdep}
     	\left\{
     	\begin{aligned}
    	\frac{\d\Z(s)}{\d t}&=-\mu\A\Z(s)-\alpha(\Y_1(\tau)-\Y_2(\tau))- \big(\mathcal{B}(\Y_1(s))-\mathcal{B}(\Y_2(s))\big)\\&
    	\quad-\beta\big(\mathcal{C}(\Y_1(s))-\mathcal{C}(\Y_2(s))\big), 
    	 \  \text{ in } \ (t,T]\times\H, \\
     	\Z(t)&=\xi-\eta\in\H.
     	\end{aligned}
     	\right.
     	\end{equation}
      By using \eqref{3.4} and \eqref{monoC2}, we find
     	\begin{align}
     &	|\langle\mathcal{B}(\Y_1)-\mathcal{B}(\Y_2),\Y_1-\Y_2\rangle|\nonumber\\&\quad\leq
     	\frac{\mu }{2}\|\nabla(\Y_1-\Y_2)\|_{\H}^2 +\frac{\beta}{4}\||\Y_2|^{\frac{r-1}{2}}(\Y_1-\Y_2)\|_{\H}^2 +\varrho\|\Y_1-\Y_2\|_{\H}^2,\label{ctsdep1}\\
     &	\langle\mathcal{C}(\Y_1)-\mathcal{C}(\Y_2),\Y_1-\Y_2\rangle\nonumber\\&\quad\geq \frac{1}{2}\||\Y_2|^{\frac{r-1}{2}}(\Y_1-\Y_2)\|_{\H}^2+
     	\frac{1}{4}\||\Y_1|^{\frac{r-1}{2}}(\Y_1-\Y_2)\|_{\H}^2.\label{ctsdep11}
     \end{align}
    On taking the inner product with $\Z(\cdot)$ in \eqref{ctsdep} and utilizing the estimates \eqref{ctsdep1} and \eqref{ctsdep11}, and using \eqref{monoC2}, we conclude   for a.e. $s\in[t,T],$ $ \P-\text{a.s.},$
    \begin{align}\label{ctsdep2}
    	\frac12\frac{\d}{\d s}\|\Z(s)\|_{\H}^2+\frac{\mu }{2} \|\nabla\Z(s)\|_{\H}^2+\alpha\|\Z(s)\|_{\H}^2+
    	\frac{\beta}{2^r}\|\Z(s)\|_{\wi\L^{r+1}}^{r+1} \leq\varrho\|\Z(s)\|_{\H}^2.
    \end{align}
    Then it follows from  Gronwall's inequality that   for all $s\in[t,T]$, $ \P-\text{a.s.},$
      \begin{align}\label{ctsdep3}
      	\|\Z(s)\|_{\H}^2\leq\|\xi-\eta\|_{\H}^2 e^{2\varrho(s-t)}.
      \end{align}
    Substituting \eqref{ctsdep3}  into \eqref{ctsdep2}, we arrive at \eqref{ctsdep0}.
     	
     	\vskip 0.25cm
     \noindent
     \emph{\textbf{Proof of (ii):}}  Let us take $\Z(\cdot):=\Y(\cdot)-\y$. Then, we rewrite   from \eqref{stap}, for $s\in[t,T]$, $\P-$a.s., 
     	\begin{align*}
     	\Z(s)=\int_t^s\big[-\mu\A\Y(\tau)-\alpha\Y(\tau)-\mathcal{B}(\Y(\tau))-\beta\mathcal{C}(\Y(\tau))+\f(\tau,\a(\tau))\big]\d\tau+\int_t^s \d\W(\tau).
     	\end{align*}
     	On applying the infinite-dimensional It\^o formula to the function $\|\cdot\|_{\H}^p$ and to the process $\Z(\cdot)$ and then taking expectation, we get
     	\begin{align}\label{ctsdep4}
     		&\E\left[\|\Z(s)\|_{\H}^p\right]\nonumber\\&=p\E\bigg[\int_t^s \big(-\mu\A\Y(\tau)-\alpha\Y(\tau)- \mathcal{B}(\Y(\tau))-\beta\mathcal{C}(\Y(\tau))+\f(\tau,\a(\tau)),\Z(\tau)\big)\|\Z(\tau)\|_{\H}^{p-2}\d\tau\bigg]
     		\nonumber\\&\quad+
     		\frac{p(p-1)\mathrm{Tr}(\Q)}{2}\E\bigg[\int_t^s \|\Z(\tau)\|^{p-2}\d\tau\bigg].
     	\end{align}
        By using the Cauchy Schwarz inequality  and Young's inequality, \eqref{ctsdep4} reduces to
        \begin{align}\label{ctsdep5}
        \E\left[\|\Z(s)\|_{\H}^p\right]\leq&-\frac{p\mu}{2}\E\bigg[\int_t^s \|\nabla\Y(\tau)\|_{\H}^2\|\Z(\tau)\|_{\H}^{p-2}\d\tau\bigg]
        -\frac{p\alpha}{2}\E\bigg[\int_t^s \|\Y(\tau)\|_{\H}^2\|\Z(\tau)\|_{\H}^{p-2}\d\tau\bigg]
        \nonumber\\&+p\E\bigg[\int_t^s
        \big(\mathcal{B}(\Y(\tau)),\y\big)\|\Z(\tau)\|_{\H}^{p-2}\d\tau\bigg]-
        p\beta\E\bigg[\int_t^s
        \|\Y(\tau)\|_{\wi\L^{r+1}}^{r+1}\|\Z(\tau)\|_{\H}^{p-2}\d\tau\bigg]
        \nonumber\\&+p\beta\E\bigg[\int_t^s
        \big(\mathcal{C}(\Y(\tau)),\y\big)\|\Z(\tau)\|_{\H}^{p-2}\d\tau\bigg]
        \nonumber\\&+C(p,R,R_1,\mathrm{Tr}(\Q),\mu)
         \E\bigg[\int_t^s\left(\|\Z(\tau)\|_{\H}^{p-1}+ \|\Z(\tau)\|_{\H}^{p-2}\right)\d\tau\bigg].
        \end{align}
     	Let us now estimate the terms $\big(\mathcal{B}(\Y(\tau)),\y\big)$ by using the  Cauchy Schwarz, H\"older's, and Young's inequalities as
     	\begin{align}\label{ctsdep6}
     	\big|\big(\mathcal{B}(\Y(\tau)),\y\big)\big|&\leq
     	\|\mathcal{B}(\Y(\tau))\|_{\H}^2+\frac14\|\y\|_{\H}^2
     	\nonumber\\&\leq
     	\||\Y(\tau)|^{\frac{r-1}{2}}\nabla\Y(\tau)\|_{\H}^{\frac{4}{r-1}}
     	\|\nabla\Y(\tau)\|_{\H}^{\frac{2(r-3)}{r-1}}+\frac14\|\y\|_{\H}^2
     	\nonumber\\&\leq
     	\frac{\mu}{4}\|\nabla\Y(\tau)\|_{\H}^2+
     	\varrho_1\||\Y(\tau)|^{\frac{r-1}{2}}\nabla\Y(\tau)\|_{\H}^2
     	+\frac14\|\y\|_{\H}^2,
     	\end{align}
     	where $\varrho_1=\left(\frac{4(r-3)}{\mu(r-1)}\right)^{\frac{r-3}{2}}
     	\frac{2}{r-1}$. Similarly, we estimate $\big(\mathcal{C}(\Y(\tau)),\y\big)$  as
     	\begin{align}\label{ctsdep7}
     	\big|\big(\mathcal{C}(\Y(\tau)),\y\big)\big|
     	&\leq\|\Y(\tau)\|_{\wi\L^{2r}}^r\|\y\|_{\H}\nonumber\\&\leq
     	\|\Y(\tau)\|_{\wi\L^{r+1}}^{\frac{r+3}{4}}\|\Y(\tau)\|_{\wi\L^{3(r+1)}}^{\frac{3(r-1)}{4}}\|\y\|_{\H}\nonumber\\&\leq
     	\frac12\|\Y(\tau)\|_{\wi\L^{r+1}}^{r+1}+C_1(r)\|\Y(\tau)\|_{\wi\L^{3(r+1)}}^{r+1}+C_2(r)\|\y\|_{\H}^{r+1}.
     	\end{align}
     	Plugging \eqref{ctsdep6}-\eqref{ctsdep7} into \eqref{ctsdep5} yields that
     	\begin{align*}
     	 \E\left[\|\Z(s)\|_{\H}^p\right]\leq&-\frac{p\mu}{4}\E\bigg[\int_t^s \|\nabla\Y(\tau)\|_{\H}^2\|\Z(\tau)\|_{\H}^{p-2}\d\tau\bigg]
     	-\frac{p\alpha}{2}\E\bigg[\int_t^s \|\Y(\tau)\|_{\H}^2\|\Z(\tau)\|_{\H}^{p-2}\d\tau\bigg]
     	\nonumber\\&+p\varrho_1\E\bigg[\int_t^s
     	\||\Y(\tau)|^{\frac{r-1}{2}}\nabla\Y(\tau)\|_{\H}^2\|\Z(\tau)\|_{\H}^{p-2}\d\tau\bigg]\nonumber\\&-
     	\frac{p\beta}{2}\E\bigg[\int_t^s
     	\|\Y(\tau)\|_{\wi\L^{r+1}}^{r+1}\|\Z(\tau)\|_{\H}^{p-2}\d\tau\bigg]
     	\nonumber\\&+p\beta C_1(r)\E\bigg[\int_t^s
     	\|\Y(\tau)\|_{\wi\L^{3(r+1)}}^{r+1}\|\Z(\tau)\|_{\H}^{p-2}\d\tau\bigg]
     	\nonumber\\&+
         C(p,R,R_1,\mathrm{Tr}(\Q),r,\mu,\beta)
     	\E\bigg[\int_t^s\left(\|\Z(\tau)\|_{\H}^{p-1}+ \|\Z(\tau)\|_{\H}^{p-2}\right)\d\tau\bigg].
     \end{align*}
   By making use of Remark \ref{rg3L3r} and the uniform energy estimates \eqref{vsee1} and \eqref{ssee1}, we finally conclude \eqref{ctsdep0.1}.
   
	\vskip 0.25cm
\noindent
\emph{\textbf{Proof of (iii):}} To prove \eqref{ctsdep0.2}, we need to show that
       \begin{align}\label{ctsdep8}
       	\sup\limits_{s\in[t,t+\eps], \ \a(\cdot)\in\mathscr{U}_t^{\nu}}
       	\E\left[\|\Y(s;t,\y,\a(\cdot))-\y\|_{\V}^p\right]\to0 \ \text{ as } \ \eps\to0.
       \end{align}
       Suppose \eqref{ctsdep8} does not hold true. Then, there exist sequences $s_n$ and $\a_n(\cdot)\in\mathscr{U}_t^\nu$ such that
       \begin{align}\label{contra1}
       	s_n\to t \ \text{ and }  \ \E\left[\|\Y_n(s_n)-\y\|_{\V}^p\right]\geq\eps
       \end{align}
       for all $n\geq1$ where $\Y_n(s_n)=\Y(s_n;t,\y,\a_n(\cdot))$. In view of Proposition \ref{weLLp} (part (iii)) and \cite[Corollary 2.21, pp. 108]{GFAS}, we can assume that $\a_n(\cdot)$ are defined on the same reference probability space. However, it follows from \eqref{ctsdep0.1} and \eqref{ssee1}, we have the following, respectively, strong and weak convergences (along a subsequence):
       \begin{align*}
       	\Y_n(s_n)\to\y \ \text{ in } \ \mathrm{L}^p(\Omega;\H)  \  \text{ and } \
       	\Y_n(s_n)\rightharpoonup\y \ \text{ in } \ \mathrm{L}^p(\Omega;\V).
       \end{align*}
       The weak sequential convergence in $\mathrm{L}^p(\Omega;\V)$ implies that
       \begin{align*}
       	\|\y\|_{\V}^p\leq\liminf_{n\to\infty}\E\left[\|\Y_n(s_n)\|_{\V}^p\right],
       \end{align*}
       while the uniform energy estimate in \eqref{eqn-conv-2} provides
       \begin{align*}
       		\|\y\|_{\V}^p\geq\limsup_{n\to\infty}\E\left[\|\Y_n(s_n)\|_{\V}^p
       		\right].
       \end{align*}
       These inequalities together yields that $\|\y\|_{\V}^p=\lim\limits_{n\to\infty}\E\left[\|\Y_n(s_n)\|_{\V}^p\right]$ and hence the Radon-Riesz property (\cite[Proposition 3.32, pp. 78]{HB}) assures $\Y_n(s_n)\to\y$ in $\mathrm{L}^p(\Omega;\V)$, which contradicts \eqref{contra1}, and the proof of \eqref{ctsdep0.2} is completed. 
     \end{proof}

     \section{Value function and the comparison principle} \label{valueSC}\setcounter{equation}{0}
     In this section, we study the continuity properties of the value function of the stochastic optimal control problem, and we prove the comparison principle. First, we discuss the dynamic programming principle and formulate the stochastic optimal control problem.
      \subsection{The dynamic programming principle}\label{stoptcon}
     The dynamic programming principle is one of the fundamental results of the stochastic optimal control problems. Since we work with an infinite-dimensional system, its proof and formulation are very technical. Let us first give a stochastic setup and the main assumptions needed for the dynamic programming principle.
     \begin{definition}\label{refprobsp}\cite[Definition 2.7, Chapter 2, pp. 97-98]{GFAS}
     A reference probability space is a generalized reference probability space $\nu:=(\Omega,\mathscr{F},\{\mathscr{F}_s^t\}_{s\geq t},\P,\mathbf{W})$, where 
     \begin{itemize}
     	\item $\mathbf{W}(t)=0,$ $\P-$a.s.,
     	\item $\mathscr{F}_s^t=\sigma(\mathscr{F}_s^{t,0},\mathscr{N})$,  where $\mathscr{F}_s^{t,0}=\sigma\big(\mathbf{W}(q):t\leq q\leq s\big)$ is the filtration generated by $\mathbf{W}(\cdot)$ and $\mathscr{N}$ is the collection of the $\P$-null sets in $\mathscr{F}$.
     \end{itemize} 
         \end{definition} 
     We now consider our stochastic control problems \eqref{stap}-\eqref{costF} where the generalized reference probability space is replaced by the reference probability space. Therefore we restricting the set of admissible controls. The set of all admissible controls is defined by
     \begin{align*}
     	\mathscr{U}_t:=\bigcup_{\nu}\mathscr{U}_t^{\nu},
     \end{align*}
where the union is taken over all reference probability spaces $\nu$ (see Definition \ref{refprobsp}). We say that the control $\a(\cdot)$ is an admissible control if there exists a reference probability space $\nu:=(\Omega^{\nu},\mathscr{F},\{\mathscr{F}_s^{t,\nu}\}_{s\geq t},\P^{\nu},\mathbf{W}^{\nu})$, such that $\a(\cdot):[t,T]\times\Omega^{\nu}\to\U$ is $\mathscr{F}_s^{t,\nu}-$progressively measurable. Here the superscript $\nu$ indicates the dependence on the reference probability space $\nu$.
Using the reference probability spaces allows us to represent the control processes as functions of Wiener processes. It will enable one to pass from one reference probability space to another; hence, there are restrictions on the reference probability space. With this setup, we have the following definition of the value function.

     \begin{definition}
     	The value function $\mathcal{V}(\cdot)$ of the stochastic optimal control problem \eqref{stap} and \eqref{costF}, with initial time $t$, is defined as  
     	\begin{align}\label{valueF}
     		\mathcal{V}(t,\y):=\inf\limits_{\a(\cdot)\in\mathscr{U}_t} J(t,\y;\a(\cdot)),
     	\end{align}
     	where $J$ is defined in \eqref{costF} for $\a(\cdot)\in\mathscr{U}_t$. 
     \end{definition}

    We assume that the running cost $\ell$ is independent of $t$. This is done to minimize the non-essential technical difficulty which might obscure the main points of the proof. 
    
     \begin{remark}
     	
     	1.) If we use the reference probability space, then the formulation specified above for DPP (weak formulation) and the strong formulation discussed in Section  \ref{stgform} are equivalent in the sense that they have the same value function (see \cite[Theorem 2.22, Chapter 2]{GFAS}). 
     
     	2.) The approach for formulating the stochastic optimal control problem considered in this work is closely related to \cite[Chapter 2 and 4]{JYXYZ} (see also \cite[Chapter 2]{GFAS}). However, plenty of literature is available regarding the formulations of stochastic optimal control problems, with various notions of control processes, in both finite and infinite dimensions. For instance, see \cite{VSB,WHF,WHF1,NVK}. 
     \end{remark}

     We assume the following assumptions on the cost functions $\ell$ and $g$, and the forcing functions $\f$.
     \begin{hypothesis}\label{valueH} 
     	 The functions $\ell:\V\times\U\to\R$ and $g:\V\to\R$ are continuous, and there exist $k\geq0$, and for every $r>0$, a modulus $\sigma_r$ such that 
     	\begin{align}
     		|\ell(\y,\a)|, |g(\y)|&\leq C(1+\|\y\|_{\V}^k), \  \ \text{ for all } \ \y\in\V, \ \a\in\U,\label{vh1}\\
     		|\ell(\y,\a)-\ell(\x,\a)|&\leq\sigma_r(\|\y-\x\|_{\V})\  \ \text{ if } \ \|\y\|_{\V},\|\x\|_{\V}\leq r, \ \a\in\U,\label{vh2}\\
     		|g(\y)-g(\x)|&\leq\sigma_r(\|\y-\x\|_{\H}) \  \ \text{ if } \ \|\y\|_{\V},\|\x\|_{\V}\leq r.\label{vh3}
     	\end{align}
     \end{hypothesis}
     The example given in the Subsection \ref{stcoexm} satisfies all the conditions of the above Hypothesis \ref{valueH}. The following proposition refers the continuity properties of the value function \eqref{valueF} and the dynamic programming principle. The proof of this proposition is standard, and one can refer \cite[Section 3.13, Chapter 3]{GFAS} for a detailed and complete explanation.
     \begin{proposition}\label{valueP}
     	Let us suppose that  Hypothesis \ref{fhyp} and \ref{valueH} are satisfied. Then, we have the following results: 
     	
     	(i) \textbf{(Continuous dependence of cost functional)} For every $r>0,$ there exists a modulus $\omega_r$ such that for every $t\in[0,T]$ and $\a(\cdot)\in\mathscr{U}_t$
     	\begin{align}\label{conJ}
     		|J(t,\y;\a(\cdot))-J(t,\x;\a(\cdot))|\leq\omega_r(\|\y-\x\|_{\H}), \ \text{ if } \ \|\y\|_{\V},\|\x\|_{\V}\leq r.
     	\end{align}
     	
     	(ii) \textbf{(Dynamic programming principle)} For every 
     	$0\leq t\leq\eta\leq T$ and $\y\in\V$, the value function $\mathcal{V}(\cdot)$ satisfies the following identity:
     	\begin{align}\label{vdpp}
     		\mathcal{V}(t,\y)=\inf\limits_{\a(\cdot)\in\mathscr{U}_t}
     		\E\left\{\int_t^\eta \ell(\Y(s;t,\y,\a(\cdot)),\a(s))\d s+ \mathcal{V}(\eta,\Y(\eta;t,\y,\a(\cdot)))\right\}.
     	\end{align}
     	
     	(iii) \textbf{(Locally Lipschitz property of the value function)} For every $r>0,$ there exists a modulus $\omega_r$ such that
     	\begin{align}\label{conV}
     		|\mathcal{V}(t_1,\y)-\mathcal{V}(t_2,\x)|\leq\omega_r(|t_1-t_2|+
     		\|\y-\x\|_{\H}),
     	\end{align}
     	for all $t_1,t_2\in[0,T]$ and $\|\y\|_{\V},\|\x\|_{\V}\leq r$. Moreover, there exists  $C\geq0$ and $k\geq 0$ such that
     	\begin{align}\label{bddV}
     		|\mathcal{V}(t,\y)|\leq C(1+\|\y\|_{\V}^k),
     	\end{align}
     	for all $t\in[0,T]$ and $\y\in\V$.
     \end{proposition}
   
      \begin{remark}
      	Other methods for proving the DPP are available in the literature (see \cite{VSB, WHF, NVK}).
      	As discussed in \cite{FGSSA}, one can work with a \emph{canonical reference probability space} for the controlled SCBF equations. That is, for $ 0 \leq t \leq T$, one can take $\Omega_t = \{\omega\in\C([t,T ];\H): \omega(t)= 0\},$ so that the Wiener process $\mathbf{W}$ can be defined on $\Omega_t$ by $\mathbf{W}(\tau)(\omega)= \omega(\tau)$. Let $\mathscr{F}_{t,s}$ be the $\sigma$-algebra generated by paths of $\mathbf{W}$ up to time $s$ in $\Omega_t$, and let $\mathbb{P}_t$ be the Wiener	measure on $\Omega_t,$ (\cite{gdp,HHK}). Therefore, $(\Omega_t,\mathscr{F}_{t,T},\{\mathscr{F}_{t,s}\}_{t\geq s},\mathbb{P}_t)$ is the \emph{canonical sample} space for the Wiener process $\mathbf{W}$. With the above setting, one can define  $\a(\cdot): [t,T ] \times\Omega_t\to\U$  is an \emph{admissible control} on $[t,T ]$ if $\a(\cdot)$	is an $\mathscr{F}_{t,s}$-progressively-measurable process. In this case, one can take $\mathscr{U}_t$ as  the set of all admissible controls $\a(\cdot)$ on $[t,T]$. 
      \end{remark}
     
	 \subsection{The HJB equation}\label{compPR}
     One of the main aim of this work is to study the value function \eqref{valueF} in details via dynamic programming approach and its characterization as a \emph{`solution'} of the following HJB equation associated with \eqref{stap}:
	 \begin{equation}\label{HJBE2}
	 	\left\{
	 	\begin{aligned}
	 		u_t+\frac12\mathrm{Tr}(\Q\D^2u)-&(\mu\A\y+\mathcal{B}(\y)+\alpha\y+
	 		\beta\mathcal{C}(\y),\D u)\\+\inf\limits_{\a\in\U} \left\{(\f(t,\a),\D u) +\ell(t,\y,\a)\right\}&=0, \ \text{ for } (t,\y)\in(0,T)\times\H,\\
	 		u(T,\y)&=g(\y), \ \text{ for } \ \y\in\H.
	 	\end{aligned}
	 	\right.
	 \end{equation}
	Here the Hamiltonian function $F$ is defined by
	 \begin{align}\label{hamfunc}
	 	F(t,\y,\p):=\inf\limits_{\a\in\U} \left\{(\f(t,\a),\p)+\ell(t,\y,\a)\right\}.
	 \end{align}
	 In this work, we study the following class of infinite-dimensional HJB equation with more general Hamiltonian $F$ satisfying certain assumptions mentioned below (see Hypothesis \ref{hypF14}):
	  \begin{equation}\label{HJBE}
	 	\left\{
	 	\begin{aligned}
	 		u_t+\frac12\mathrm{Tr}(\Q\D^2u)-&(\mu\A\y+\mathcal{B}(\y)+\alpha\y+
	 		\beta\mathcal{C}(\y),\D u)\\+F(t,\y,\D u)&=0, \ \text{ for } (t,\y)\in(0,T)\times\H,\\
	 		u(T,\y)&=g(\y), \ \text{ for } \ \y\in\H.
	 	\end{aligned}
	 	\right.
	 \end{equation}
	We prove the comparison principle for the general class of HJB equation \eqref{HJBE} (see Theorem \ref{comparison}). However, the proof of the existence of viscosity solution requires the stochastic optimal control techniques, and therefore, it is limited to the equation \eqref{HJBE2} (see Theorem \ref{extunqvisc}).
	  \subsection{Viscosity solution}
	  In general, the value function is not smooth enough to satisfy the HJB equation in the classical sense. However, under certain assumptions on the Hamiltonian (Hypothesis \ref{valueH}) and cost functional (including running and terminal cost; Hypothesis \ref{hypF14}), we show that the value function \eqref{valueF} is the unique viscosity solution of the HJB equation \eqref{HJBE2}, with some appropriate boundary conditions.
	  
	From Proposition \ref{valueP}, since we only have continuity properties of the value function on the space $[0, T]\times\V$, it is more appropriate to restrict the definition of viscosity solution to this space. The setup in this space might be better from the perspective of the HJB equation, however, because of our associated control problem \eqref{HJBE},  it is essential to keep $\H$ as our reference space. To achieve this, one has to choose the test function more appropriately. As mentioned in \cite[Subsection 3.13.3, Chapter 3, pp. 344]{GFAS} (also see \cite{FGSSA}), we are going to use the following approach:
	\begin{itemize}
	\item[(i)] By using a special radial function of ${\V}-$norm as a test function, we first restrict the points where maxima or minima occur in the definition of the viscosity sub/super solution to be in 
	$(0,T)\times\V$.
	\item[(ii)] Then, to make sense of all terms appearing in the HJB equation \eqref{HJBE}, we require the points where maxima or minima occur to be in $(0, T)\times\V_2$.
	\end{itemize} 
	  Using the above requirements and the properties of the CBF equations, we show that the value function \eqref{valueF} is a viscosity solution of \eqref{HJBE}.  Let us first define the test function as follows:

	\begin{definition}\label{testD}
		A function $\uppsi:(0,T)\times\H\to\R$ is called a \emph{test function} for \eqref{HJBE} if
		\begin{align}\label{tf}
		\uppsi(t,\y)=\upvarphi(t,\y)+\mathfrak{h}(t)(1+\|\y\|_{\V}^2)^m,
		\end{align} 
		 where
		\begin{itemize}
			\item $\upvarphi\in\mathrm{C}^{1,2}((0,T)\times\H)$ and is such that $\upvarphi_t,\D\upvarphi$ and $\D^2\upvarphi$ are uniformly continuous on $[\eps,T-\eps]\times\H$ for every $\eps>0$;
			\item $\mathfrak{h}\in\mathrm{C}^1(0,T)$ and $\mathfrak{h}(t)>0$ for every $t\in(0,T)$.
		\end{itemize}
	\end{definition}
	
	\begin{remark}
		
	Note that the function $$h(t,\y):=\mathfrak{h}(t)(1+\|\y\|_{\V}^2)^m$$ is not Fr\'echet differentiable in $\H$ (see \cite[Appendix A]{GFAS}). Therefore, the terms involving $\D h$ and $\D^2 h$, in particular $(\mu\A\y+ \mathcal{B}(\y)+\alpha\y+\beta\mathcal{C}(\y),\D h(t,\y))$ and $\mathrm{Tr}(\Q\D^2h(t,\y))$ have to be understood in a proper way. Following \cite[pp. 345]{GFAS}(see also \cite{FGSSA}), with an abuse  of notation, we define
		\begin{align*}
			\D h(t,\y):= -\mathfrak{h}(t)\left(2m(1+\|\y\|_{\V}^2)^{m-1}\A\y\right)
		\end{align*}
		and we write $\D\psi:=\D\varphi+\D h$. Now, if $(t,\y)\in(0,T)\times\V_2$, then $\D\psi(t,\y)$ makes sense and so does the term $(\mu\A\y+ \mathcal{B}(\y)+\alpha\y+\beta\mathcal{C}(\y),\D \psi(t,\y))$. We interpret the term $\mathrm{Tr}(\Q\D^2\psi(t,\y))$, without defining $\D^2h(t,\y)$, as 
		\begin{align*}
			\mathrm{Tr}(\Q\D^2\psi(t,\y)):=&\mathrm{Tr}(\Q\D^2\varphi(t,\y))+
			\mathfrak{h}(t)\bigg(2m(1+\|\y\|_{\V}^2)^{m-1}\mathrm{Tr}(\Q_1)\\&+4m(m-1)(1+\|\y\|_{\V}^2)^{m-2}\|\Q^{\frac12}\A\y\|_{\H}^2\bigg).
		\end{align*}
		It will be seen in the proof of comparison principle, the above interpretations will appear as direct consequences of It\^o' formula applied to $h$. 
	\end{remark}
       
	   We now define the viscosity solution, which has been adopt from \cite{FGSSA} (also see \cite[Definition 3.133, pp. 345]{GFAS}). Note that the definition of viscosity solution given below is for the general HJB equation \eqref{HJBE} where the Hamiltonian $F$ need not be necessarily of the form \eqref{hamfunc}. We assume that the Hamiltonian $F:[0, T]\times\V\times\H\to\R$ is any function.
   	\begin{definition}\label{viscsoLndef}
		A weakly sequentially upper-semicontinuous (respectively, lower-semicontinuous) function $u:(0,T]\times\V\to\R$ is called \emph{a viscosity subsolution} (respectively, \emph{supersolution}) of \eqref{HJBE} if $u(T,\wi\y)\leq h(\wi\y)$ (respectively, $u(T,\wi\y)\geq h(\wi\y)$) for all $\wi\y\in\V$ and if, for every test function $\uppsi$, whenever $u-\uppsi$ has global maximum (respectively, $u+\uppsi$ has a global minimum) at a point $(t,\y)\in(0,T)\times\V$ then $\y\in\V_2$ and 
		\begin{align}\label{eqn-subso}
		&\uppsi_t(t,\y)+\frac12\mathrm{Tr}(\Q\D^2\uppsi(t,\y))\nonumber\\&\quad-
	(\mu\A\y+ \mathcal{B}(\y)+\alpha\y+\beta\mathcal{C}(\y),\D\uppsi(t,\y))
			+F(t,\y,\D\uppsi(t,\y))\geq0
		\end{align}
		(respectively, 
		 \begin{align}\label{eqn-superso}
		-&\uppsi_t(t,\y)-\frac12\mathrm{Tr}(\Q\D^2\uppsi(t,\y))\nonumber\\&+
		(\mu\A\y+\mathcal{B}(\y)+\alpha\y+\beta\mathcal{C}(\y),\D\uppsi(t,\y))+F(t,\y,-\D\uppsi(t,\y))\leq0.)
			   	\end{align}
		
		A \emph{viscosity solution} of \eqref{HJBE} is a function which is both viscosity subsolution and viscosity supersolution.
	\end{definition}
	
	\subsection{Comparison principle}\label{COMPvisc}
		 This subsection proves the comparison principle for the equation \eqref{HJBE}. The comparison principle ensures that under certain conditions, a viscosity subsolution is always less than or equal to a viscosity supersolution. It is usually the more difficult part of the theory. As a consequence, this gives the uniqueness of the viscosity solutions.
	\begin{hypothesis}\label{hypF14}
		Let $F:[0,T]\times\V\times\H\to\R$ and there exists a modulus of continuity $\omega$ and moduli $\omega_r$ such that for every $r>0$ we have 
		\begin{align}
			|F(t,\y,\p)-F(t,\x,\p)|&\leq\omega_r(\|\y-\x\|_{\V})+\omega(\|\y-\x\|_{\V}\|\p\|_{\H}), \ \text{ if } \ \|\y\|_{\V}\leq r,\|\x\|_{\V}\leq r,\label{F1}\\
			|F(t,\y,\p)-F(t,\y,\q)|&\leq\omega((1+\|\y\|_{\V})\|\p-\q\|_{\H}),\label{F2}\\
			|F(t,\y,\p)-F(s,\y,\p)|&\leq\omega_r(|t-s|),\ \text{ if } \ \|\y\|_{\V}\leq r,\|\y\|_{\V}\leq r, \|\p\|_{\V}\leq r,\label{F3}\\
			|g(\y)-g(\x)|&\leq\omega_r(\|\y-\x\|_{\H}), \ \text{ if } \ \|\y\|_{\V}\leq r,\|\x\|_{\V}\leq r.\label{F4}
		\end{align}
	\end{hypothesis}
	
Let us now state the comparison principle, which is essential to prove the uniqueness of the viscosity solutions.
	\begin{theorem}\label{comparison}
		Assume that Hypothesis \ref{hypF14} holds. Let $u:(0,T]\times\V\to\R$ be a viscosity subsolution of \eqref{HJBE} and $v:(0,T]\times\V\to\R$ be a viscosity supersolution of \eqref{HJBE}. Let 
		\begin{align}\label{bdd}
			u(t,\y), -v(t,\y), |g(\y)|\leq C(1+\|\y\|_{\V}^k), 
		\end{align}
		for some $k>0$. Then, for $r$ in Table \ref{Table1}, we have 
		$$u\leq v\ \text{ on }\ (0,T]\times\V.$$
	\end{theorem}
	 \begin{proof}
	We observe that the weak sequential upper-semicontinuity of $u$ and weak sequential lower-semicontinuity of $v$ (since they are viscosity sub- and supersolution of \eqref{HJBE}, respectively) imply that 
			\begin{equation}\label{posneg}
				\left\{
				\begin{aligned}
					\lim\limits_{t\uparrow T} \ (u(t,\y)-g(\y))^+=0,\\
					\lim\limits_{t\uparrow T} \ (v(t,\y)-g(\y))^-=0,
				\end{aligned}
				\right.
			\end{equation}
			uniformly on bounded subsets of $\V$, where for any real-valued function $f$, $f^+=\max\{f,0\}$ and $f^-=\max\{-f,0\}$. For $\upgamma>0$, we define
		\begin{align*}
			u_{\upgamma}(t,\y):=u(t,\y)-\frac{\upgamma}{t}\  \text{ and } \ v_{\upgamma}(t,\y):=v(t,\y)+\frac{\upgamma}{t}. 
		\end{align*}
	
		Since $u$ is a viscosity subsolution of \eqref{HJBE}, therefore $u_{\upgamma}$ is a viscosity subsolution of 
		\begin{equation}\label{newsub}
		\left\{
		\begin{aligned}
			&(u_\upgamma)_t+\frac12\mathrm{Tr}(\Q\D^2 u_{\upgamma})-
		(\mu\A\y+\alpha\y+\mathcal{B}(\y)+\beta\mathcal{C}(\y),\D u_{\upgamma})
		+F(t,\y,\D u_{\upgamma})\geq\frac{\upgamma}{T^2},\\
		&u_\upgamma(T,\y)=g(\y)-\frac{\upgamma}{T}.
		\end{aligned}
		\right.
		\end{equation}
		Similalry, $v_{\upgamma}$ is a viscosity supersolution of 
			\begin{equation}\label{newsup}
			\left\{
			\begin{aligned}
				&(v_\upgamma)_t+\frac12\mathrm{Tr}(\Q\D^2 u_{\upgamma})-
				(\mu\A\y+\alpha\y+\mathcal{B}(\y)+\beta\mathcal{C}(\y),\D v_{\upgamma})
				+F(t,\y,\D u_{\upgamma})\leq-\frac{\upgamma}{T^2},\\
				&v_\upgamma(T,\y)=g(\y)+\frac{\upgamma}{T}.
			\end{aligned}
			\right.
		\end{equation}
	If we can prove that $u_{\upgamma}\leq v_{\upgamma}$, then we can obtain $u\leq v$ by letting $\upgamma\to0$. Let us assume that $u_{\upgamma}\not\leq v_{\upgamma}$ on $(0,T]\times\V$. Then,  there is a $\kappa>0$ such that for sufficiently small $\upgamma>0$, we have (\cite[Theorem 3.50]{GFAS})
	\begin{align*}
		0<\upgamma<\mathfrak{m}:=\lim\limits_{R\to+\infty}\lim\limits_{q\to0}\lim\limits_{\upnu\to0}\sup\bigg\{& u_{\upgamma}(t,\y)-v_{\upgamma}(s,\x): \|\y-\x\|_{\H}<q, \ \|\y\|_{\H},\|\x\|_{\H}\leq R, \\& |t-s|<\upnu, \ \kappa<t,s\leq T\bigg\}.
	\end{align*}
	We also define
	\begin{align*}
		\mathfrak{m}_{\delta}:=\lim\limits_{q\to0}\lim\limits_{\upnu\to0}\sup\bigg\{&u_{\upgamma}(t,\y)-v_{\upgamma}(s,\x)-\delta e^{K_\upgamma(T-t)} (1+\|\y\|_{\V}^2)^m-\delta e^{K_\upgamma(T-s)} (1+\|\x\|_{\V}^2)^m:\\& \|\y-\x\|_{\H}<q, \ \y,\x\in\H, \ \ |t-s|<\upnu, \ \kappa<t,s\leq T\bigg\},\\
		\mathfrak{m}_{\delta,\eps}:=\lim\limits_{\upnu\to0}\sup\bigg\{&u_{\upgamma}(t,\y)-v_{\upgamma}(s,\x)-\delta e^{K_\upgamma(T-t)} (1+\|\y\|_{\V}^2)^m-\delta e^{K_\upgamma(T-s)} (1+\|\x\|_{\V}^2)^m\\&-\frac{\|\y-\x\|_{\H}^2}{2\eps}: \y,\x\in\H, \ \ |t-s|<\upnu, \ \kappa<t,s\leq T\bigg\},\\
		\mathfrak{m}_{\delta,\eps,\eta}:=\sup\bigg\{&u_{\upgamma}(t,\y)-v_{\upgamma}(s,\x)-\delta e^{K_\upgamma(T-t)} (1+\|\y\|_{\V}^2)^m-\delta e^{K_\upgamma(T-s)} (1+\|\x\|_{\V}^2)^m\\&-\frac{\|\y-\x\|_{\H}^2}{2\eps}-\frac{(t-s)^2}{2\eta}: \ \y,\x\in\H, \ \kappa<t,s\leq T\bigg\}.
	\end{align*}
	From the above definitions, we have the following convergences:
	\begin{align}\label{comconv}
		\mathfrak{m}\leq\lim\limits_{\delta\to0}\mathfrak{m}_\delta, \ \  \mathfrak{m}_\delta=
		\lim\limits_{\eps\to0}\mathfrak{m}_{\eps,\delta}, \ \ \text{ and } \ 
		\mathfrak{m}_{\delta,\eps}=\lim\limits_{\eta\to0}\mathfrak{m}_{\delta,\eps,\eta}.
	\end{align}
		For $\eps,\delta,\eta>0$, we define the function $\Phi$ on $(0,T]\times\H$ by
		\begin{equation*}
			\Phi(t,s,\y,\x)=\left\{
			\begin{aligned}
				&u_{\upgamma}(t,\y)-v_{\upgamma}(s,\x)-
				\frac{\|\y-\x\|_{\H}^2}{2\eps}-\delta e^{K_\upgamma(T-t)} (1+\|\y\|_{\V}^2)^m\\&-\delta e^{K_\upgamma(T-s)} (1+\|\x\|_{\V}^2)^m-\frac{(t-s)^2}{2\eta}, \  &&\text{ if } \ \y,\x\in\V\\
				&-\infty, \  &&\text{ if } \ \y,\x\notin\V.
			\end{aligned}
			\right.
		\end{equation*}
		Clearly, $\Phi\to-\infty$ as $\max\{\|\y\|_{\V},\|\x\|_{\V}\}\to+\infty$.
	
		\vskip 0.2cm
		\noindent	\textbf{Step-I:} \emph{$\Phi$ is weakly sequentially upper-semicontinuous on $(0,T]\times(0,T]\times\H\times\H$.}  
		Since norm is known to be a weakly sequentially lower-semicontinuous function, the functions $\y\mapsto(1+\|\y\|_{\V}^2)^m$, $\x\mapsto(1+\|\x\|_{\V}^2)^m$, and $(\y,\x)\mapsto\|\y-\x\|_{\H}^2$ are in $\H$ and $\H\times\H$, respectively, are weakly sequentially lower-semicontinuous. Moreover, $u_{\upgamma}$ is a weakly sequentially upper-semicontinuous function in $(0,T)\times\V$ as $u$ is the viscosity subsolution of \eqref{HJBE}. We will now show that
			$u_{\upgamma}(t,\y)-\delta e^{K_\upgamma(T-t)} (1+\|\y\|_{\V}^2)^m $ \text{is weakly sequentially upper-semicontinuous on} $ (0,T)\times\H.$
	
		Suppose this is not true. Then, there exists a sequence $(t_n)_{n\geq1}$ in $(0,T)$ with $t_n\to t\in(0,T)$ and a sequence $(\y_n)_{n\geq1}$ in $\H$ with $\y_n\rightharpoonup\y\in\H$ such that
		\begin{align}\label{contra}
			\limsup\limits_{n\to\infty}\left(u_{\upgamma}(t_n,\y_n)-\delta e^{K_\upgamma(T-t_n)} (1+\|\y_n\|_{\V}^2)^m\right)>
			u_{\upgamma}(t,\y)-\delta e^{K_\upgamma(T-t)} (1+\|\y\|_{\V}^2)^m.
		\end{align}
		Now, if $\liminf\limits_{n\to\infty}\|\y_n\|_{\V}=+\infty$, then \eqref{contra} is impossible because of the assumption \eqref{bdd} on $u$. Therefore, $\liminf\limits_{n\to\infty}\|\y_n\|_{\V}<+\infty$ and by the properties of limit inferior, there exists a subsequence (still denoted by $(t_n,\y_n)$) such that $\limsup\limits_{n\to\infty}\|\y_n\|_{\V}<+\infty$. By an application of the Banach-Alaoglu theorem, we then have $\y_n\rightharpoonup\y$ in $\V$ (along a subsequence), which implies $\|\y\|_{\V}\leq\liminf\limits_{n\to\infty}\|\y_n\|_{\V}$ and $$\liminf_{n\to\infty}\delta e^{K_\upgamma(T-t_n)} (1+\|\y_n\|_{\V}^2)^m\geq \delta e^{K_\upgamma(T-t)} (1+\|\y\|_{\V}^2)^m. $$
		Therefore, from \eqref{contra}, we further have
		\begin{align*}
	\limsup\limits_{n\to\infty}u_{\upgamma}(t_n,\y_n)>u_{\upgamma}(t,\y),
		\end{align*}
		 which yields a contradiction to the fact that $u_{\upgamma}$ is weakly sequentially upper-semicontinuous.  Similarly, one can show that $v_{\upgamma}(s,\x)-\delta e^{K_\upgamma(T-s)} (1+\|\x\|_{\V}^2)^m$ weakly sequentially lower-semicontinuous on $(0,T)\times\H$. Finally, we conclude our required claim. Therefore by the definition of viscosity solution, $\Phi$ has global maximum over $(0,T]\times(0,T]\times\H\times\H$ at some point $(\bar{t},\bar{s},\bar{\y},\bar{\x})\in(0,T]\times(0,T]\times\V\times\V$.
		
		 \vskip 0.2cm
	 \noindent	 \textbf{Step-II:} \emph{ $\|\bar{\y}\|_{\V},\|\bar{\x}\|_{\V}$ are bounded independently of $\eps$, for a fixed $\delta>0$.}  
		 Indeed, for any $\y\in\V$, we have 
		 \begin{align*}
		  \Phi(\bar{t},\bar{s},\y,\y)\leq\Phi(\bar{t},\bar{s},\bar{\y},\bar{\x}).
		 \end{align*}
		 Using the definition of $\Phi$ and \eqref{bdd}, we obtain
		 \begin{align*}
		 &\delta e^{K_\upgamma(T-\bar{t})} (1+\|\bar{\y}\|_{\V}^2)^m+\delta e^{K_\upgamma(T-\bar{s})} (1+\|\bar{\x}\|_{\V}^2)^m
		 \nonumber\\&\leq
		 u_{\upgamma}(\bar{t},\bar{\y})-u_{\upgamma}(\bar{t},\y)+
		 v_{\upgamma}(\bar{s},\y)-v_{\upgamma}(\bar{s},\bar{\x})
		 \nonumber\\&\quad+
		 [\delta e^{K_\upgamma(T-\bar{t})} +\delta e^{K_\upgamma(T-\bar{s})}](1+\|\y\|_{\V}^2)^m
		 \nonumber\\&\leq
		 C(1+\|\bar{\y}\|_{\V}^k)+C(1+\|\bar{\x}\|_{\V}^k)+C(1+\|\y\|_{\V}^k)
		 \nonumber\\&\quad+
		  [\delta e^{K_\upgamma(T-\bar{t})} +\delta e^{K_\upgamma(T-\bar{s})}](1+\|\y\|_{\V}^2)^m,
		 \end{align*}
		 for all $\y\in\V$ and for all $0<\bar{t},\bar{s}\leq T$. In particular for $\y=\boldsymbol{0}$ and using the fact that $1\leq e^{K_{\upgamma}(T-t)}\leq e^{K_{\upgamma}T}$ for any $t\in[0,T]$, and then using the Young's inequality, we find that
		 \begin{align*}
		 \delta(1+\|\bar{\y}\|_{\V}^2)^m+\delta(1+\|\bar{\x}\|_{\V}^2)^m
		 \leq C,
		 \end{align*}
		 provided $2m\geq k+1$. Thus for a fixed $\delta$, we conclude that $\bar{\x}$ and $\bar{\y}$ are bounded indepdently of $\eps$ in $\V$. By using \eqref{comconv}, we have the following:
		\begin{align}\label{copm1}
			\lim\limits_{\eta\to0}\frac{(\bar{t}-\bar{s})^2}{2\eta}&=0 \ \text{ for fixed } \ \delta>0, \eps>0,
		\end{align}
		and
		\begin{align}\label{copm2}
			\lim\limits_{\eps\to0}\limsup\limits_{\eta\to0} \frac{\|\bar{\y}-\bar{\x}\|_{\H}^2}{2\eps}&=0 \ \text{ for fixed } \ \delta>0.
		\end{align}
		Interested readers can refer \cite[Chapter 3, pp. 207-209]{GFAS} for a detailed explanation of \eqref{copm1}-\eqref{copm2} (in the context of parabolic equations). We can assume this maximum point to be strict (for instance, see \cite[Lemma 3.37, Chapter 3]{GFAS}). Moreover, by the definition of viscosity solution $\bar{\y},\bar{\x}\in\V_2$. 
       
         \vskip 0.2cm
         \noindent  \textbf{Step-III:} \emph{ For small $\upgamma$ and $\delta$, we have $\bar{t},\bar{s}<T$ if $\eta$ and $\eps$ are sufficiently small.}  
		Note that $\Phi(\bar{t},\bar{s},\bar{\x},\bar{\y})>0$ for small $\upgamma,\delta>0$. If either of $\bar{s}$ or $\bar{t}$ equal to $T$, then from \eqref{F4}, \eqref{posneg}, \eqref{copm1} and \eqref{copm2}, we have
		\begin{align}\label{tseT}
		\Phi(\bar{t},\bar{s},\bar{\x},\bar{\y})\leq-
		\frac{2\upgamma}{T}+\omega_r(\|\bar{\y}-\bar{\x}\|_{\H})-
		\delta(1+\|\bar{\y}\|_{\V}^2)^m-\delta(1+\|\bar{\x}\|_{\V}^2)^m.
		\end{align}
         For fixed $\upgamma>0$, we choose a constant  $C_{\upgamma}$ such that $\omega_r(z)\leq\frac{\upgamma}{T}+C_{\upgamma}z$. 
         Then, from \eqref{tseT}, we have
         \begin{align}\label{tseT1}
         \Phi(\bar{t},\bar{s},\bar{\x},\bar{\y})\leq-
         \frac{\upgamma}{T}+
         (C_{\upgamma}-\delta)(1+\|\bar{\y}\|_{\V}^2)^m+(C_{\upgamma}-\delta)(1+\|\bar{\x}\|_{\V}^2)^m.
         \end{align}
         Now, if the constant $C_{\upgamma}$ is such that $C_{\upgamma}<\delta$, then the right hand side of the inequality \eqref{tseT1} becomes less than $0$ while $ \Phi(\bar{t},\bar{s},\bar{\x},\bar{\y})>0$, which is not possible. Hence, we must have $\bar{t},\bar{s}<T$, for sufficiently small $\eta$ and $\eps$.

		\vskip 0.2cm
		\noindent	\textbf{Step-IV:} \emph{Reduction to finite-dimensional space.} 
		Let $\H_1\subset\H_2\subset\ldots$ be finite-dimensional subspaces of $\H$ generated by the eigenfunctions of $\A$ such that $\overline{\bigcup_{N=1}^\infty}\H_N=\H$. Given $N\in\N$, $N>1$, denote by $\mathbf{P}_N,$ the orthogonal projection onto $\H_N$. It is clear that $\mathbf{P}_N$  is a bounded linear operator on $\H$ and so is $\mathbf{Q}_N:=\I-\mathbf{P}_N$. Let us denote  $\H_N^{\perp}=\mathbf{Q}_N\H$. We then have an orthogonal decomposition $\H=\H_N\times\H_N^{\perp}$. For $\y\in\H$, we will denote by $\y_N:=\mathbf{P}_N\y,$ an element of $\H_N$ and by $\y_N^{\perp}:=\mathbf{Q}_N\y,$ an element of $\H_N^{\perp}$. Then, we write $\y=(\mathbf{P}_N\y,\mathbf{Q}_N\y)=(\y_N,\y_N^{\perp})$. 
		
		Let us now fix $N\in\N$. By using the properties of $\mathbf{P}_N$ and $\mathbf{Q}_N$, we have following straightforward identities:
		\begin{align}
			\|\y-\x\|_{\H}^2&=\|\mathbf{P}_N(\y-\x)\|_{\H}^2
			+\|\mathbf{Q}_N(\y-\x)\|_{\H}^2,\label{fnPQ1}\\
			\|\mathbf{Q}_N(\y-\x)\|_{\H}^2
			&\leq2(\mathbf{Q}_N(\bar{\y}-\bar{\x}),\y-\x)-
			\|\mathbf{Q}_N(\bar{\y}-\bar{\x})\|_{\H}^2+2\|\mathbf{Q}_N(\y-\bar{\y})\|_{\H}^2\nonumber\\&\quad
			+2\|\mathbf{Q}_N(\x-\bar{\x})\|_{\H}^2\nonumber,
		\end{align}
		with the equality in the second inequality if $\y=\bar{\y}$ and $\x=\bar{\x}$.
		Let us now define 
		\begin{equation*}
			\wi u(t,\y):=
			\left\{
			\begin{aligned}
				&u_{\upgamma}(t,\y)-\frac{(\y,\mathbf{Q}_N(\bar{\y}-\bar{\x}))}{\eps}+\frac{\|\mathbf{Q}_N(\bar{\y}-\bar{\x})\|_{\H}^2}{2\eps}-
				\frac{\|\mathbf{Q}_N(\y-\bar{\y})\|_{\H}^2}{\eps}\nonumber\\&-
				\delta e^{K_\upgamma(T-t)} (1+\|\y\|_{\V}^2)^m, \ \text{ when } \ \y,\x\in\V,\\
				&-\infty, \ \text{ when } \ \y,\x\notin\V,
			\end{aligned}
			\right.
		\end{equation*}
		and 
		\begin{equation*}
			\wi v(s,\x):=
			\left\{
			\begin{aligned}
				&v_{\upgamma}(s,\x)-\frac{(\x,\mathbf{Q}_N(\bar{\y}-\bar{\x}))}{\eps}+
				\frac{\|\mathbf{Q}_N(\x-\bar{\x})\|_{\H}^2}{\eps}\nonumber\\&+
				\delta e^{K_\upgamma(T-s)} (1+\|\x\|_{\V}^2)^m, \ \text{ when } \ \y,\x\in\V,\\
				&+\infty, \ \text{ when } \ \y,\x\notin\V.
			\end{aligned}
			\right.
		\end{equation*}
	We emphasize that such extended $\wi\u$ and $\wi\v$ are weakly sequentially upper-semicontinuous and lower-semicontinuous on $(0,T]\times\H$, respectively. It follows that the function
		\begin{align*}
			\wi\Phi(t,s,\y,\x):=\wi u(t,\y)-\wi v(s,\x)-\frac{\|\mathbf{P}_N(\y-\x) \|_{\H}^2}{2\eps}-\frac{(t-s)^2}{2\eta}.
		\end{align*}
		always satisfies $\wi\Phi\leq\Phi$ whenever $\y,\x\in\V$ and it attains a strict global maximum over $(0,T]\times(0,T]\times\H\times\H$ at $(\bar{t},\bar{s},\bar{\y},\bar{\x})$ where $\wi\Phi(\bar{t},\bar{s},\bar{\y},\bar{\x}) =\Phi(\bar{t},\bar{s},\bar{\y},\bar{\x})$.

		\vskip 0.2cm
		\noindent	\textbf{Step-V:} \emph{Finite-dimensional maximum principle.} 
On dividing both sides by $\|\y\|_{\H}$ in the definition of $	\wi{u}(t,\y)$, we obtain
		\begin{align*}
			\limsup\limits_{\|\y\|_{\H}\to\infty}\sup\limits_{t\in(0,T)}\frac{\wi u(t,\y)}{\|\y\|_{\H}}<0.
		\end{align*}
		In a similar way, we have 
		\begin{align*}
			\limsup\limits_{\|\x\|_{\H}\to\infty}\sup\limits_{t\in(0,T)}\frac{-\wi v(s,\x)}{\|\x\|_{\H}}<0.
		\end{align*}
		Therefore, $\wi u$ and $-\wi v$ satisfy the assumptions of \cite[Corollary 3.29, Chapter 3, pp. 194]{GFAS}. As a consequence, there exist functions $\varphi_n,\psi_n\in\mathrm{C}^{1,2}((0,T)\times\H)$ for $n=1,2,\ldots$ such that  
		$$\varphi_n, (\varphi_n)_t, \D\varphi_n, \D^2\varphi_n,\psi_n, (\psi_n)_t, \D\psi_n, \D^2\psi_n$$ are bounded and uniformly continuous, and such that
		\begin{align*}
			\wi u(t,\y)-\varphi_n(t,\y) \ \text{ has a global maximum at some point } \ (t_n,\y_n)\in(0,T)\times\V,
		\end{align*}
		and 
		\begin{align*}
			\wi v(s,\x)-\psi_n(s,\x) \ \text{ has a global minimum at some point } \ (s_n,\x_n)\in(0,T)\times\V.
		\end{align*}
		Moreover, the following convergences holds as $n\to\infty$:
		\begin{equation}\label{conv1}
			\left\{
			\begin{aligned}
				t_n&\to\bar{t} \ \text{ in } \ \R, \ &\y_n&\to\bar{\y} \ \text{ in } \ \H,\\
				\wi u(t_n,\y_n)&\to\wi u(\bar{t},\bar{\y}) \ \text{ in } \ \R, \  &(\varphi_n)_t(t_n,\y_n)&\to\frac{\bar{t}-\bar{s}}{\eta} \ \text{ in } \ \R,\\ \D\varphi_n(t_n,\y_n)&\to\frac{1}{\eps}(\bar{\y}_N-\bar{\x}_N) \ \text{ in } \ \V, \ &\D^2\varphi_n(t_n,\y_n)&\to X_N \ \text{ in } \ \mathscr{L}(\H),
			\end{aligned}
			\right.
		\end{equation}
		and 
		\begin{equation}\label{conv2}
			\left\{
			\begin{aligned}
				s_n&\to\bar{s} \ \text{ in } \ \R, \ &\x_n&\to\bar{\x} \ \text{ in } \ \H,\\
				\wi v(s_n,\x_n)&\to\wi v(\bar{s},\bar{\x}) \ \text{ in } \ \R, \  &(\varphi_n)_t(s_n,\x_n)&\to\frac{\bar{t}-\bar{s}}{\eta} \ \text{ in } \ \R,\\ \D\psi_n(t_n,\x_n)&\to\frac{1}{\eps}(\bar{\y}_N-\bar{\x}_N) \  \text{ in } \ \V, \ &\D^2\psi_n(t_n,\x_n)&\to Y_N \ \text{ in } \ \mathscr{L}(\H),
			\end{aligned}
			\right.
		\end{equation}
		where $X_N=\mathbf{P}_N X_N\mathbf{P}_N$, $Y_N=\mathbf{P}_N Y_N\mathbf{P}_N$ are bounded independently of $n$ and satisfy 
		\begin{align}\label{xnyn}
			-\frac{3}{\eps}\begin{pmatrix} \mathbf{P}_N \ \ \ 0\\0 \  \  \ \mathbf{P}_N \end{pmatrix}\leq\begin{pmatrix} X_N \ \ \ 0 \\0  \  \ \ Y_N\end{pmatrix}\leq\frac{3}{\eps}
			\begin{pmatrix} \mathbf{P}_N \ \ \ -\mathbf{P}_N \\-\mathbf{P}_N  \ \ \  \mathbf{P}_N\end{pmatrix}.
		\end{align}
		It is clear from \eqref{xnyn} that $X_N,Y_N$ satisfy $X_N\leq Y_N$, that is, $X_N\xi\cdot\xi\leq Y_N\xi\cdot\xi$ for $\xi\in\R^N$, where `$\cdot$' indicates the Euclidean product on $\R^N$. Furthermore, by the definition of $\wi u$ and the fact that $\wi u-\varphi_n$ has a global maximum at $(t_n,\y_n)\in(0,T)\times\V$, we must have $\|\y_n\|_{\V}\leq C$, for some constant $C$ independent of $n$. By a similar reasoning, we also have $\|\x_n\|_{\V}\leq C$, for some constant $C$ independent of $n$. Therefore, by an application of the Banach-Alaoglu theorem, we have the following weak convergences (along a subsequence still denoted by the same symbol) as $n\to\infty$:
			\begin{align*}
				\y_n\rightharpoonup\wi\y \ \text{ and } \ \x_n\rightharpoonup\wi\x, \ \text{ in } \ \V.
			\end{align*}
By the uniqueness of weak limits, we further have $\wi\x=\bar{\x}$ and $\wi\y=\bar{\y}$.
	Note that weak convergent sequences are bounded. 	The above weak convergence together with \eqref{conv1}-\eqref{conv2}, and the fact that  $(\bar{t},\bar{s},\bar{\y},\bar{\x})$ is a maximum point of $\Phi$ imply that 
		\begin{align*}
			\|\y_n\|_{\V}\to\|\bar{\y}\|_{\V},  \ \|\x_n\|_{\V}\to\|\bar{\x}\|_{\V},  \ \text{ as } \ n\to\infty. 
		\end{align*}
		which in turn gives the following strong convergence as $n\to\infty$ (Radon-Riesz property):
		\begin{align}\label{eqn-v-norm-conv}
			\y_n\to\bar{\y}\ \text{ and } \ \x_n\to\bar{\x} \ \text{ in } \ \V. 
		\end{align}
	
		\vskip 0.2cm
		\noindent	\textbf{Step-VI:} \emph{Applying definition of viscosity solution.} Let us define a test function 
		\begin{align*}
			\varphi(t,\y):=&\varphi_n(t,\y)+\frac{(\y,\mathbf{Q}_N(\bar{\y}-\bar{\x}))}{\eps}+\frac{\|\mathbf{Q}_N(\y-\bar{\y})\|_{\H}^2}{\eps}-\frac{\|\mathbf{Q}_N(\bar{\y}-\bar{\x})\|_{\H}^2}{2\eps} \nonumber\\&+
			\delta e^{K_\upgamma(T-t)} (1+\|\y\|_{\V}^2)^m.
		\end{align*}
		Since $u_\upgamma$ is a viscosity subsolution of \eqref{newsub} over $(0,T)\times\V$, then by the definition of viscosity subsolution with the test function $\varphi$, we have $\y_n\in\V_2$ (even $\x_n\in\V_2$ also, in the case of viscosity supersolution $v_\upgamma$), and it satisfies    
		{\small {	\begin{align}\label{supsoLdef}
&-\delta K_\upgamma e^{K_\upgamma(T-t_n)} (1+\|\y_n\|_{\V}^2)^m+
		(\varphi_n)_t(t_n,\y_n)+\frac12\mathrm{Tr}(\Q\D^2\varphi_n(t_n,\y_n))
		\nonumber\\&\quad+
		\frac{1}{2\eps}\mathrm{Tr}\bigg(\Q\D^2(\y_n,\mathbf{Q}_N(\bar{\y}-\bar{\x}))+\Q\D^2\|\mathbf{Q}_N(\y_n-\bar{\y})\|_{\H}^2-\frac12\Q\D^2\|\mathbf{Q}_N(\bar{\y}-\bar{\x})\|_{\H}^2\bigg)
		\nonumber\\&\quad+
		 \frac{\delta}{2}e^{K_\upgamma(T-t_n)}
         \bigg(2m(1+\|\y_n\|_{\V}^2)^{m-1}\big(\mathrm{Tr}(\Q)+\mathrm{Tr}(\Q_1)\big)+4m(m-1)(1+\|\y_n\|_{\V}^2)^{m-2}\|\Q^{\frac12}(\A+\I)\y_n\|_{\H}^2\bigg)\nonumber\\&\quad-
         \bigg(\mu\A\y_n,\D\varphi_n(t_n,\y_n)+\frac{1}{\eps}\mathbf{Q}_N(\bar{\y}-\bar{\x})+\frac{2}{\eps}\mathbf{Q}_N(\y_n-\bar{\y})
         +2m\delta e^{K_\upgamma(T-t_n)} (1+\|\y_n\|_{\V}^2)^{m-1}(\A+\I)\y_n\bigg)
          \nonumber\\&\quad-
           \bigg(\alpha\y_n,\D\varphi_n(t_n,\y_n)+\frac{1}{\eps}\mathbf{Q}_N(\bar{\y}-\bar{\x})+\frac{2}{\eps}\mathbf{Q}_N(\y_n-\bar{\y})
          +2m\delta e^{K_\upgamma(T-t_n)} (1+\|\y_n\|_{\V}^2)^{m-1}(\A+\I)\y_n\bigg)
          \nonumber\\&\quad-
         \bigg(\mathcal{B}(\y_n),\D\varphi_n(t_n,\y_n)+\frac{1}{\eps}\mathbf{Q}_N(\bar{\y}-\bar{\x})+\frac{2}{\eps}\mathbf{Q}_N(\y_n-\bar{\y})
         +2m\delta e^{K_\upgamma(T-t_n)} (1+\|\y_n\|_{\V}^2)^{m-1}(\A+\I)\y_n\bigg)
         \nonumber\\&\quad-
         \beta\bigg(\mathcal{C}(\y_n),\D\varphi_n(t_n,\y_n)+\frac{1}{\eps}\mathbf{Q}_N(\bar{\y}-\bar{\x})+\frac{2}{\eps}\mathbf{Q}_N(\y_n-\bar{\y})+2m\delta
         e^{K_\upgamma(T-t_n)} (1+\|\y_n\|_{\V}^2)^{m-1}(\A+\I)\y_n\bigg)
         \nonumber\\&\quad+
         F\bigg(t_n,\y_n,\D\varphi_n(t_n,\y_n)+\frac{1}{\eps}\mathbf{Q}_N(\bar{\y}-\bar{\x})+\frac{2}{\eps}\mathbf{Q}_N(\y_n-\bar{\y})
         +2m\delta
         e^{K_\upgamma(T-t_n)} (1+\|\y_n\|_{\V}^2)^{m-1}(\A+\I)\y_n\bigg)
         \nonumber\\&\geq\frac{\upgamma}{T^2}.
		\end{align}}}
		The following \emph{derivatives} are immediate 
		\begin{align}\label{derQ}
		\D^2(\y_n,\mathbf{Q}_N(\bar{\y}-\bar{\x}))=0, \  \D^2\|\mathbf{Q}_N(\y_n-\bar{\y})\|_{\H}^2=2\mathbf{Q}_N \ \text{ and } \ \D^2\|\mathbf{Q}_N(\bar{\y}-\bar{\x})\|_{\H}^2=0.
		\end{align}
		On utilizing \eqref{derQ} in \eqref{supsoLdef} and rearranging the terms, we obtain 
     	\begin{align}\label{viscdef1}
		&-\delta K_\upgamma e^{K_\upgamma(T-t_n)} (1+\|\y_n\|_{\V}^2)^m+(\varphi_n)_t(t_n,\y_n)+ \frac12\mathrm{Tr}(\Q\D^2\varphi_n(t_n,\y_n)+2\Q\mathbf{Q}_N)
		\nonumber\\&\quad+
		\frac{\delta}{2}e^{K_\upgamma(T-t_n)}
		\bigg(2m(1+\|\y_n\|_{\V}^2)^{m-1}\big(\mathrm{Tr}(\Q)+\mathrm{Tr}(\Q_1)\big)
		\nonumber\\&\qquad+4m(m-1)(1+\|\y_n\|_{\V}^2)^{m-2}\|\Q^{\frac12}(\A+\I)\y_n\|_{\H}^2\bigg)
		\nonumber\\&\quad-
		\bigg(\big((\mu\A+\alpha\I)\y_n,\mathfrak{J}_n\big)+2m\delta e^{K_\upgamma(T-t_n)} (1+\|\y_n\|_{\V}^2)^{m-1}((\mu\A+\alpha\I)\y_n,(\A+\I)\y_n)\bigg)
		\nonumber\\&\quad-
		\bigg(\big(\mathcal{B}(\y_n),\mathfrak{J}_n\big)
		+2m\delta e^{K_\upgamma(T-t_n)} (1+\|\y_n\|_{\V}^2)^{m-1}\big(\mathcal{B}(\y_n),(\A+\I)\y_n\big)\bigg)
		\nonumber\\&\quad-
		\beta\bigg(\big(\mathcal{C}(\y_n),\mathfrak{J}_n\big)
	   +2m\delta e^{K_\upgamma(T-t_n)} (1+\|\y_n\|_{\V}^2)^{m-1}\big(\mathcal{C}(\y_n),(\A+\I)\y_n\big)\bigg)
		\nonumber\\&\quad+
		F\bigg(t_n,\y_n,\mathfrak{J}_n
	   +2m\delta e^{K_\upgamma(T-t_n)} (1+\|\y_n\|_{\V}^2)^{m-1}(\A+\I)\y_n\bigg)
		\nonumber\\&\geq\frac{\upgamma}{T^2}.
		\end{align}
     where $\mathfrak{J}_n:=\D\varphi_n(t_n,\y_n)+ \frac{1}{\eps}\mathbf{Q}_N(\bar{\y}-\bar{\x})+ \frac{2}{\eps}\mathbf{Q}_N(\y_n-\bar{\y}).$
	   From assumption \eqref{F2} of Hypothesis \ref{hypF14}, we calculate
	{\small 
   {\begin{align}\label{F2cal}
    &\bigg|F\big(t_n,\y_n,\mathfrak{J}_n+2m\delta       e^{K_\upgamma(T-t_n)}(1+\|\y_n\|_{\V}^2)^{m-1}(\A+\I)\y_n\big)
   -F\big(t_n,\y_n,\mathfrak{J}_n\big)
	\bigg|\nonumber\\&\leq
	\omega\left((1+\|\y_n\|_{\V})\times2m\delta e^{K_\upgamma(T-t_n)}
	(1+\|\y_n\|_{\V}^2)^{m-1}\|(\A+\I)\y_n\|_{\H}\right)\nonumber\\&\leq
	 \frac{\upgamma}{2T^2}+2 C_\upgamma(1+\|\y_n\|_{\V})m\delta e^{K_\upgamma(T-t_n)}(1+\|\y_n\|_{\V}^2)^{m-1}\|(\A+\I)\y_n\|_{\H} .
		\end{align}}}
		Let us calculate 
		\begin{align}\label{trq1}
		\|\Q^{\frac12}(\A+\I)\y_n\|_{\H}^2\leq2
		\big(\|\Q^{\frac12}\A\y_n\|_{\H}^2+\|\Q^{\frac12}\y_n\|_{\H}^2\big)
		&=
		2\big(\|\Q_1^{\frac12}\A^{\frac12}\y_n\|_{\H}^2+\|\Q^{\frac12}\y_n\|_{\H}^2\big)\nonumber\\&\leq2
		\big(\mathrm{Tr}(\Q_1)+\mathrm{Tr}(\Q)\big)\|\y_n\|_{\V}^2.
		\end{align}
		Substituting \eqref{F2cal} and \eqref{trq1} in \eqref{viscdef1}, we obtain 
	\begin{align}\label{viscdef111}
		&-\delta K_\upgamma e^{K_\upgamma(T-t_n)} (1+\|\y_n\|_{\V}^2)^m+(\varphi_n)_t(t_n,\y_n)+ \frac12\mathrm{Tr}(\Q\D^2\varphi_n(t_n,\y_n)+2\Q\mathbf{Q}_N)
		\nonumber\\&\quad+
		m(4m-3)\delta e^{K_\upgamma(T-t_n)}
		 \big(\mathrm{Tr}(\Q)+\mathrm{Tr}(\Q_1)\big)
		 (1+\|\y_n\|_{\V}^2)^{m-1}
		\nonumber\\&\quad-
		\bigg(\big((\mu\A+\alpha\I)\y_n,\mathfrak{J}_n\big)+\big(\mathcal{B}(\y_n),\mathfrak{J}_n\big)+\beta\big(\mathcal{C}(\y_n),\mathfrak{J}_n\big)\bigg)
		\nonumber\\&\quad-
		2m\delta e^{K_\upgamma(T-t_n)}(1+\|\y_n\|_{\V}^2)^{m-1} 
		\underbrace{\big((\mu\A+\alpha\I)\y_n+\mathcal{B}(\y_n)+\beta\mathcal{C}(\y_n),(\A+\I)\y_n\big)}_{I_1}
		\nonumber\\&\quad+
		F\big(t_n,\y_n,\mathfrak{J}_n\big)
		\nonumber\\&\quad+
		\underbrace{2 C_\upgamma(1+\|\y_n\|_{\V})m\delta e^{K_\upgamma(T-t_n)}(1+\|\y_n\|_{\V}^2)^{m-1}
		\|(\A+\I)\y_n\|_{\H}}_{I_2}
		\nonumber\\&\geq\frac{\upgamma}{2T^2}.
	\end{align}
	
	Let us now calculate $I_1$. For this, we consider the following two cases:
	\vskip 2mm
	\noindent
	\textbf{Case-I: When $r>3$ in $d\in\{2,3\}$.}
	From \eqref{syymB3} and \eqref{torusequ}, we infer
		\begin{align}\label{rg3c}
			&-\min\{\mu,\alpha\}\|(\A+\I)\y_n\|_{\H}^2-
			2(\mathcal{B}(\y_n),(\A+\I)\y_n)- 2\beta(\mathcal{C}(\y_n),(\A+\I)\y_n)\nonumber\\&
			\leq-\frac{\min\{\mu,\alpha\}}{2}\|(\A+\I)\y_n\|_{\H}^2-
			\beta\||\y_n|^{\frac{r-1}{2}}\nabla\y_n\|_{\H}^2+ 2\varrho_1\|\nabla\y_n\|_{\H}^2-
			2\beta\|\y_n\|_{\wi\L^{r+1}}^{r+1},
		\end{align}
		where $\varrho_1=\frac{r-3}{\min\{\mu,\alpha\}(r-1)}\left[\frac{4}{\beta\min\{\mu,\alpha\} (r-1)}\right]^{\frac{2}{r-3}}$.
   \vskip 2mm
   \noindent
   \textbf{Case-II: When $r=3$ with $2\beta\mu\geq1$ in $d\in\{2,3\}$.}
   		Let $0<\theta<1.$
   		By using the Cauchy-Schwarz and Young's inequalities, we calculate
   		\begin{align*}
   				|(\mathcal{B}(\y_n),(\A+\I)\y_n)|&\leq\||\y_n|\nabla\y_n\|_{\H}
   				\|(\A+\I)\y_n\|_{\H}
   				\nonumber\\&\leq\frac{\theta\mu}{2}\|(\A+\I)\y_n\|_{\H}^2+
   				\frac{1}{2\theta\mu}\||\y_n|\nabla\y_n\|_{\H}^2.
   			\end{align*}
   		For $r=3$, we write \eqref{toreq} as
   		\begin{align*}
   		(\mathcal{C}(\y_n),(\A+\I)\y_n)=\||\y_n|\nabla\y_n\|_{\H}^{2}+ \frac{1}{2}\|\nabla|\y_n|^2\|_{\H}^{2}+\|\y_n\|_{\wi\L^{r+1}}^{r+1}.
   			\end{align*}
   		From \eqref{syymB3} and \eqref{torusequ}, we infer
   		\begin{align}\label{re3c}
   				&-\min\{\mu,\alpha\}\|(\A+\I)\y_n\|_{\H}^2-
   				2(\mathcal{B}(\y_n),(\A+\I)\y_n)- 2\beta(\mathcal{C}(\y_n),(\A+\I)\y_n)\nonumber\\&\leq -\min\{\mu,\alpha\}(1-\theta)\|(\A+\I)\y_n\|_{\H}^2-
   				\left(2\beta-\frac{1}{\theta\mu}\right)\||\y_n|\nabla\y_n\|_{\H}^2-
   				2\beta\|\y_n\|_{\wi\L^{r+1}}^{r+1}.
   			\end{align}
   				By using the Cauchy-Schwarz inequality, we estimate $I_2$ as
   			\begin{align}\label{trq2}
   				&2 m\delta e^{K_\upgamma(T-t_n)}(1+\|\y_n\|_{\V}^2)^{m-1} 
   				\underbrace{C_\upgamma(1+\|\y_n\|_{\V})\|(\A+\I)\y_n\|_{\H}}_{\text{Cauchy-Schwarz}}
   				\nonumber\\&\leq 
   				\frac{2C_\upgamma^2 m\delta}{\min\{\mu,\alpha\}} e^{K_\upgamma(T-t_n)}(1+\|\y_n\|_{\V})^m+\min\{\mu,\alpha\} m\delta e^{K_\upgamma(T-t_n)} \|(\A+\I)\y_n\|_{\H}^2(1+\|\y_n\|_{\V}^2)^{m-1}.
   			\end{align}
   			
   Let us now choose $$K_\upgamma=1+2\left(\frac{2 C_\upgamma^2 m}{\min\{\mu,\alpha\}}+m(4m-3) \big(\mathrm{Tr}(\Q)+\mathrm{Tr}(\Q_1)\big)+\mathfrak{a}_1
   	\right),$$  where
   	\begin{equation*}
   		\mathfrak{a}_1:=
   		\left\{
   		\begin{aligned}
   		2m\varrho_1, \ &\text{when} \ r>3, \ \text{ for all } \ \mu,\beta>0,\\
   		0,  \ &\text{when} \ r=3, \ \text{ with } \ 2\beta\mu\geq1.
   		\end{aligned}
   		\right.
   	\end{equation*}
   On utilizing \eqref{rg3c}-\eqref{trq2} in \eqref{viscdef111}, we obtain
   \begin{align}\label{viscdef3}
   	&- \frac{\delta}{2}K_\upgamma e^{K_\upgamma(T-t_n)} (1+\|\y_n\|_{\V}^2)^m+
   	(\varphi_n)_t(t_n,\y_n)+\frac12\mathrm{Tr}(\Q\D^2\varphi_n(t_n,\y_n)+2\Q\mathbf{Q}_N)
   	\nonumber\\&\quad-
   	\big((\mu\A+\alpha\I)\y_n,\mathfrak{J}_n\big)-
   	\big(\mathcal{B}(\y_n),\mathfrak{J}_n\big)-
   	\beta\big(\mathcal{C}(\y_n),\mathfrak{J}_n\big)+
   	F\big(t_n,\y_n,\mathfrak{J}_n\big)
   	\nonumber\\&\geq\frac{\upgamma}{2T^2}+\mathfrak{a}_3 m\delta e^{K_\upgamma(T-t_n)} (1+\|\y_n\|_{\V}^2)^{m-1}\||\y_n|^{\frac{r-1}{2}}\nabla\y_n\|_{\H}^2
   	\nonumber\\&\quad+
   		2\beta m\delta e^{K_\upgamma(T-t_n)}(1+\|\y_n\|_{\V}^2)^{m-1} \|\y_n\|_{\wi\L^{r+1}}^{r+1}
   	\nonumber\\&\quad+
   	\mathfrak{a}_2\min\{\mu,\alpha\}m\delta e^{K_\upgamma(T-t_n)} (1+\|\y_n\|_{\V}^2)^{m-1}\|(\A+\I)\y_n\|_{\H}^2,
   \end{align}
   where
   \begin{equation}\label{e3cbf1}
   	\mathfrak{a}_2:=
   	\left\{
   	\begin{aligned}
   		\frac12, \ &\text{when} \ r>3, \ \text{ for all } \ \mu,\beta>0,\\
   		1-\theta,  \ &\text{when} \ r=3, \ \text{ with } \ 2\beta\mu\geq1.
   	\end{aligned}
   	\right.
   \end{equation}
   and 
   \begin{equation}\label{e3cbf}
   	\mathfrak{a}_3:=
   	\left\{
   	\begin{aligned}
   		\beta, \ &\text{when} \ r>3, \ \text{ for all } \ \mu,\beta>0,\\
   		2\beta-\frac{1}{\theta\mu},  \ &\text{when} \ r=3, \ \text{ with } \ 2\beta\mu\geq1.
   	\end{aligned}
   	\right.
   \end{equation}
On taking  limit supremum as $n\to\infty$ in \eqref{viscdef3} and utilizing \eqref{eqn-v-norm-conv} (the detailed justification can be found in the Appendix \ref{expoca}), we deduce
	\begin{align}\label{viscdef4}
		&- \frac{\delta}{2}K_\upgamma e^{K_\upgamma(T-\bar{t})} (1+\|\bar{\y}\|_{\V}^2)^m+\frac{\bar{t}-\bar{s}}{\eta}
	     +\frac12\mathrm{Tr}(\Q X_N+2\Q\mathbf{Q}_N)\nonumber\\&
		\quad -	\bigg((\mu\A+\alpha\I)\bar{\y},\frac{\bar{\y}-\bar{\x}}{\eps}\bigg) 
		-\bigg(\mathcal{B}(\bar{\y}),\frac{\bar{\y}-\bar{\x}}{\eps}
		\bigg)\nonumber\\&\quad-\beta
		\bigg(\mathcal{C}(\bar{\y}),\frac{\bar{\y}-\bar{\x}}{\eps}
		\bigg)+
		F\bigg(\bar{t},\bar{\y},\frac{\bar{\y}-\bar{\x}}{\eps}\bigg)
		\nonumber\\&\geq\frac{\upgamma}{2T^2}.
	\end{align}
	On employing a similar procedure as above for the supersolution $v_\upgamma$, we arrive at
	\begin{align}\label{viscdef5}
		&\frac{\delta}{2}K_\upgamma e^{K_\upgamma(T-\bar{s})} (1+\|\bar{\x}\|_{\V}^2)^m+\frac{\bar{t}-\bar{s}}{\eta}
		+\frac12\mathrm{Tr}(\Q Y_N-2\Q\mathbf{Q}_N)\nonumber\\&\quad-
		\bigg((\mu\A+\alpha\I)\bar{\x},\frac{\bar{\y}-\bar{\x}}{\eps}
		\bigg)-
		\bigg(\mathcal{B}(\bar{\x}),\frac{\bar{\y}-\bar{\x}}{\eps}\bigg)
		\nonumber\\&\quad
		-\beta\bigg(\mathcal{C}(\bar{\x}),\frac{\bar{\y}-\bar{\x}}{\eps}
		\bigg)+
		F\bigg(\bar{s},\bar{\x},\frac{\bar{\y}-\bar{\x}}{\eps}\bigg)
		\nonumber\\&\leq-\frac{\upgamma}{2T^2}.
	\end{align}
	On combining \eqref{viscdef4}-\eqref{viscdef5} and using the fact that $X_N\leq Y_N$, and then passing $N\to\infty$
	\begin{align}\label{viscdef6}
		&\frac{\delta}{2} (1+\|\bar{\y}\|_{\V}^2)^m+\frac{\delta}{2} (1+\|\bar{\x}\|_{\V}^2)^m
		+\frac{\min\{\mu,\alpha\}}{\eps}\|\bar{\y}-\bar{\x}\|_{\V}^2 
		\nonumber\\&\quad+
		\bigg(\mathcal{B}(\bar{\y})-\mathcal{B}(\bar{\x}),\frac{\bar{\y}-\bar{\x}}{\eps}\bigg)+\beta
		\bigg(\mathcal{C}(\bar{\y})-\mathcal{C}(\bar{\x}),\frac{\bar{\y}-\bar{\x}}{\eps}\bigg)\nonumber\\&\quad
		+F\bigg(\bar{t},\bar{\y},\frac{\bar{\y}-\bar{\x}}{\eps}\bigg)-
		F\bigg(\bar{s},\bar{\x},\frac{\bar{\y}-\bar{\x}}{\eps}\bigg)
		\nonumber\\&\leq-\frac{\upgamma}{T^2}
	\end{align}
	
	\vskip 0.2cm
\noindent	\textbf{Step-VII:} \emph{Final conclusion by getting a contradiction.}
	For fixed $\upgamma$ and $\delta$, we have $\|\bar{\y}\|_{\V}, \|\bar{\x}\|_{\V}\leq R_\delta$ for some $R_\delta>0$. Let $D_{\upgamma,\delta}$ be a constant such that
	\begin{align*}
	\omega_{R_\delta}(s)\leq\frac{\upgamma}{4T^2}+D_{\upgamma,\delta}s.
	\end{align*}
	Now, we calculate
	\begin{align}\label{Fhm}
	&\bigg|F\bigg(\bar{t},\bar{\y},\frac{\bar{\y}-\bar{\x}}{\eps}\bigg)-
	F\bigg(\bar{s},\bar{\x},\frac{\bar{\y}-\bar{\x}}{\eps}\bigg)\bigg|
	\nonumber\\&\leq 
	\bigg|F\bigg(\bar{t},\bar{\y},\frac{\bar{\y}-\bar{\x}}{\eps}\bigg)-
	F\bigg(\bar{s},\bar{\y},\frac{\bar{\y}-\bar{\x}}{\eps}\bigg)\bigg|
	+
	\bigg|F\bigg(\bar{s},\bar{\y},\frac{\bar{\y}-\bar{\x}}{\eps}\bigg)-
	F\bigg(\bar{s},\bar{\x},\frac{\bar{\y}-\bar{\x}}{\eps}\bigg)\bigg|
	\nonumber\\&\leq
	\omega_{R_{\delta,\eps}}(|\bar{t}-\bar{s}|)
	+\omega_{R_{\delta}}(\|\bar{\y}-\bar{\x}\|_{\V})
	+\omega\left(\|\bar{\y}-\bar{\x}\|_{\V}\frac{\|\bar{\y}-\bar{\x}\|_{\H}}
	{\eps}\right)\nonumber\\&\leq
	\omega_{R_{\delta,\eps}}(|\bar{t}-\bar{s}|)+\frac{3\upgamma}{4T^2}+D_{\upgamma,\delta}\|\bar{\y}-\bar{\x}\|_{\V}+C_\upgamma
	\|\bar{\y}-\bar{\x}\|_{\V}\frac{\|\bar{\y}-\bar{\x}\|_{\H}}{\eps}
	\nonumber\\&\leq\omega_{R_{\delta,\eps}}(|\bar{t}-\bar{s}|)+\frac{3\upgamma}{4T^2}+D_{\upgamma,\delta}
	\|\bar{\y}-\bar{\x}\|_{\V}+\frac{C_\upgamma^2} {\min\{\mu,\alpha\}\eps}\|\bar{\y}-\bar{\x}\|_{\H}^2
	+\frac{\min\{\mu,\alpha\}}{4\eps}\|\bar{\y}-\bar{\x}\|_{\V}^2.
	\end{align}
	 From \eqref{3.4} and \eqref{monoC2}, we calculate
	\begin{align}\label{rg3c1}
	&\bigg(\mathcal{B}(\bar{\y})-\mathcal{B}(\bar{\x}),\frac{\bar{\y}-\bar{\x}}{\eps}\bigg)+\beta
	\bigg(\mathcal{C}(\bar{\y})-\mathcal{C}(\bar{\x}),\frac{\bar{\y}-\bar{\x}}{\eps}\bigg)\nonumber\\&\geq-
	\frac{\min\{\mu,\alpha\}}{4\eps}\|\nabla(\bar{\y}-\bar{\x})\|_{\H}^2 -\frac{\varrho}{\eps}\|\bar{\y}-\bar{\x}\|_{\H}^2,
	\end{align}
where $\varrho$ is defined in \eqref{eqn-varrho}. 	Combining \eqref{Fhm} and \eqref{rg3c1}, using in \eqref{viscdef6} and then using the fact that $\bar{\y}$ and $\bar{\x}$ are in $\V$, and by an application of  a Taylor's formula (\cite[Theorem 7.9.1]{PGC}), we conclude that
	\begin{align}\label{viscdef7}
	&\frac{\delta}{2} (1+\|\bar{\y}\|_{\V}^2)^m+\frac{\delta}{2} (1+\|\bar{\x}\|_{\V}^2)^m
	+\frac{\min\{\mu,\alpha\}}{2\eps}\|\bar{\y}-\bar{\x}\|_{\V}^2
	\nonumber\\&\leq-\frac{\upgamma}{4T^2}+
	D_{\upgamma,\delta}\|\bar{\y}-\bar{\x}\|_{\V}+\omega_{R_{\delta,\eps}}(|\bar{t}-\bar{s}|)+\left(\frac{C_\upgamma^2}{\min\{\mu,\alpha\}\eps}+\frac{\min\{\mu,\alpha\}}{4\eps}+\frac{\varrho}{\eps}\right)\|\bar{\y}-\bar{\x}\|_{\H}^2
	\nonumber\\&\quad
	+C(K_\upgamma,\delta,\varrho,T) \|\bar{\y}-\bar{\x}\|_{\V}.
   \end{align}
   On simplifying further and using \eqref{copm1}-\eqref{copm2}, we finally have
  \begin{align*}
  	\frac{\min\{\mu,\alpha\}}{2\eps}\|\bar{\y}-\bar{\x}\|_{\V}^2-
  	\left(D_{\upgamma,\delta}+C(K_\upgamma,\delta,\varrho,T)\right)
  	\|\bar{\y}-\bar{\x}\|_{\V}
  \leq-\frac{\upgamma}{4T^2}+
  	\upsigma_2(\eta,\eps;\delta,\upgamma),
  \end{align*}
	where, for fixed $\upgamma,\delta$, we have $\limsup\limits_{\eps\to0}
	\limsup\limits_{\eta\to0}\upsigma_2(\eta,\eps;\delta,\upgamma)=0$.
Now, on taking the infimum, we obtain
	\begin{align}\label{contd1}
&	\inf\limits_{\|\bar{\y}-\bar{\x}\|_{\V}>0}  
     \left\{\frac{\min\{\mu,\alpha\}}{2\eps}
	\|\bar{\y}-\bar{\x}\|_{\V}^2-
	\left(D_{\upgamma,\delta}+C(K_\upgamma,\delta,\varrho,T)\right)
	\|\bar{\y}-\bar{\x}\|_{\V}\right\}\nonumber\\&\leq-\frac{\upgamma}{4T^2}+
	\upsigma_2(\eta,\eps;\delta,\upgamma).
	\end{align} 
	Now on taking $\limsup\limits_{\eps\to0}\limsup\limits_{\eta\to0}$ in \eqref{contd1} and using the fact that $$\lim\limits_{\eps\to0}\inf\limits_{r>0}\left(\frac{\mu r^2}{4\eps}-\left(D_{\upgamma,\delta}+C(K_\upgamma,\delta,\varrho,T)\right)r\right)=0,$$
	we obtain
	\begin{align*}
	 0\leq-\frac{\upgamma}{4T^2} \ \text{ or } \ \upgamma<0,
	\end{align*}
	which gives a contradiction to the assumption $u\not\leq v$ and hence $u\leq v$ on $(0,T]\times\V$.
		\end{proof}


	\section{Existence and uniqueness of viscosity solutions}\label{extunqvisc1}\setcounter{equation}{0}
	In this section, we go back to the HJB equation \eqref{HJBE} with the Hamiltonian function $F$ defined in \eqref{hamfunc} and show that the value function of the associated stochastic optimal control problem is its viscosity solution.
	
	\begin{theorem}\label{extunqvisc}
  Let us suppose that Hypothesis \ref{valueH} is satisfied. In addition, let $\f:[0,T]\times\U\to\V$ be such that $\f(\cdot,\a)$ is uniformly continuous, uniformly for $\a\in\U$. Then, for the values of $r$ given in Table \ref{Table2}, the value function $\mathcal{V}$ defined in \eqref{valueF} is the \emph{unique viscosity solution} of the HJB equation \eqref{HJBE2} within the class of viscosity solutions $u$ satisfying
  \begin{align*}
  	|u(t,\y)|\leq C(1+\|\y\|_{\V}^k), \  \ (t,\y)\in(0,T)\times\V,
  \end{align*} 
  for some $k\geq0$.
	\end{theorem}
	
	\begin{proof}
		We first notice that under the assumptions of Hypothesis \ref{valueH}, the Hamiltonian $F$ given in \eqref{hamfunc} satisfies all the conditions of Hypothesis \ref{hypF14}. Moreover, by Proposition \ref{valueP}, the value function $\mathcal{V}$ satisfies \eqref{vdpp}-\eqref{bddV}. In particular, we infer from \eqref{bddV} that  $\mathcal{V}$ is weakly sequentially continuous on $(0,T]\times\V$. Further, we only need to show that $\mathcal{V}$ is a \emph{viscosity solution} of the problem \eqref{HJBE}, since the uniqueness is a direct consequence of the comparison principle (see Theorem \ref{comparison}). The outline of the proof is as follows:
		\begin{itemize}
			\item We first show that the points of minima (or maxima) in the definition of viscosity supersolution (or subsolution) are in $\V_2$. 
			\item We then use the dynamic programming principle and
			carefully apply various estimates for solutions of the state equation \eqref{stap} to pass to the limit and obtain the inequalities as in  Definition \ref{viscsoLndef}.
		\end{itemize}
		We will only show that the value function $\mathcal{V}$ is a viscosity supersolution. The proof that $\mathcal{V}$ is a viscosity subsolution uses the same arguments. To prove this, let $\psi(t,\y)=\varphi(t,\y)+\delta(t)(1+\|\y\|_{\V}^2)^m$ be a test function (see Definition \ref{testD}) and let $\mathcal{V}+\psi$ has a global minimum at $(t_0,\y_0)\in(0,T)\times\V$.

	\vskip 1mm
	\noindent
	\textbf{To prove $\y_0\in\V_2$.}  By dynamic programming principle \eqref{vdpp}, for every $\eps>0,$ there exists $\a_\eps(\cdot)\in\mathscr{U}_{t_0}$ such that
	\begin{align}\label{vdp1}
		\mathcal{V}(t_0,\y_0)+\eps^2>\E\left\{\int_{t_0}^{t_0+\eps}
		\ell(\Y_\eps(s),\a_\eps(s))\d s+ \mathcal{V}(t_0+\eps,\Y_\eps(t_0+\eps))\right\},
	\end{align}
	where $\Y_\eps(\cdot)=\Y(\cdot;t_0,\y_0,\a_\eps(\cdot))$. We can assume that $\a_\eps(s)$ is $\mathscr{F}_s^{t_0,0}-$predictable and thus by \cite[Corollary 2.21, Chapter 2, pp. 108]{GFAS} and Proposition \ref{weLLp}--part (iii), one can assume that all $\a_\eps(\cdot)$ are defined on the same reference probability space $\nu$, that is, $\a_\eps(\cdot)\in\mathscr{U}_{t_0}^{\nu}$. Since  $(\mathcal{V}+\psi)(t_0,\y_0)\leq(\mathcal{V}+\psi)(t,\y)$ for every $(t,\y)\in(0,T)\times\V$, that is,
	\begin{align}\label{vdp2}
		\mathcal{V}(t,\y)-\mathcal{V}(t_0,\y_0)\geq-\varphi(t,\y)
		+\varphi(t_0,\y_0)-\delta(t)(1+\|\y\|_{\V}^2)^m+
		\delta(t_0)(1+\|\y_0\|_{\V}^2)^m,
	\end{align}
	for every $(t,\y)\in(0,T)\times\V$. Then, from \eqref{vdp1} and \eqref{vdp2}, we have
	\begin{align}\label{vdp3}
	&\eps^2-\E\int_{t_0}^{t_0+\eps}\ell(\Y_\eps(s),\a_\eps(s))\d s \nonumber\\&\geq \E\left[\mathcal{V}(t_0+\eps,\Y_\eps(t_0+\eps))-\mathcal{V}(t_0,\y_0)\right]\nonumber\\&\geq
	\E\big[-\varphi(t_0+\eps,\Y_\eps(t_0+\eps))+\varphi(t_0,\y_0)
	\nonumber\\&\qquad- \delta(t_0+\eps)(1+\|\Y_\eps(t_0+\eps)\|_{\V}^2)^m+
	\delta(t_0)(1+\|\y_0\|_{\V}^2)^m\big].
	\end{align}
	Let us set $\lambda=\inf\limits_{t\in[t_0,t_0+\eps_0]}\delta(t)$ for some $\eps_0>0$ and take $\eps<\eps_0$. On applying It\^o's formula to the processes $\varphi(\cdot,\Y_\eps(\cdot))$ and $\delta(\cdot)(1+\|\Y_\eps(\cdot))\|_{\V}^2)^m$ and then dividing both sides by $\eps$, we obtain
	\begin{align}\label{vdp3.1}
	&\eps-\frac{1}{\eps}\E\int_{t_0}^{t_0+\eps}\ell(\Y_\eps(s),\a_\eps(s))\d s
	\nonumber\\&\geq
	-\frac{1}{\eps}\E\biggl[\int_{t_0}^{t_0+\eps}\bigg(\varphi_t(s,\Y_\eps(s))-\big((\mu\A+\alpha\I)\Y_\eps(s)+\mathcal{B}(\Y_\eps(s))+\beta\mathcal{C}(\Y_\eps(s)),\D\varphi(s,\Y_\eps(s))\big)
	\nonumber\\&\qquad+
	\big(\f(s,\a_\eps(s)),\D\varphi(s,\Y_\eps(s))\big)+\frac12\mathrm{Tr}(\Q\D^2\varphi(s,\Y_\eps(s)))\bigg)\d s\biggr]
	\nonumber\\&\qquad
	-\frac{1}{\eps}\E\bigg[\int_{t_0}^{t_0+\eps}\bigg(\delta'(s)(1+\|\Y_\eps(s))\|_{\V}^2)^m+m\delta(s)\big(\mathrm{Tr}(\Q_1)+\mathrm{Tr}(\Q)\big)(1+\|\Y_\eps(s))\|_{\V}^2)^{m-1}\bigg)\d s\bigg]
	\nonumber\\&\qquad+ \frac{2m}{\eps}\E\bigg[\int_{t_0}^{t_0+\eps}\bigg(\delta(s)\bigg(\big((\mu\A+\alpha\I)\Y_\eps(s),(\A+\I)\Y_\eps(s)\big)+
	(\mathcal{B}(\Y_\eps(s)),(\A+\I)\Y_\eps(s))
	\nonumber\\&\qquad+
		\beta(\mathcal{C}(\Y_\eps(s)),(\A+\I)\Y_\eps(s))-\big(\f(s,\a_\eps(s)),
	(\A+\I)\Y_\eps(s)\big)\big)\bigg)(1+\|\Y_\eps(s))\|_{\V}^2)^{m-1}
	\bigg)\d s\bigg]
	\nonumber\\&\qquad-
	\frac{2m(m-1)}{\eps}\E\bigg[\int_{t_0}^{t_0+\eps}\delta(s)
	\|\Q^{\frac12}(\A+\I)\Y_\eps(s)\|_{\H}^2
	(1+\|\Y_\eps(s)\|_{\V}^2)^{m-2}\d s\bigg].
\end{align}
     By the definition of $\lambda$ and from the equality \eqref{torusequ}, it then follows that
	\begin{align}\label{vdp4}
		&\frac{2m\lambda}{\eps}\E\bigg[\int_{t_0}^{t_0+\eps} \bigg(\beta(\mathcal{C}(\Y_{\eps}),(\A+\I)\Y_{\eps})+\min\{\mu,\alpha\}\|(\A+\I)\Y_\eps(s)\|_{\H}^2\bigg)\nonumber\\&\qquad\times
		(1+\|\Y_\eps(s)\|_{\V}^2)^{m-1}\d s\bigg]
		\nonumber\\&\leq\eps+
		\frac{1}{\eps}\E\biggl[\int_{t_0}^{t_0+\eps}\bigg(-\ell(\Y_\eps(s),\a_\eps(s))+\varphi_t(s,\Y_\eps(s))+(\f(s,\a_\eps(s)),\D\varphi(s,\Y_\eps(s)))\nonumber\\&\qquad-\big((\mu\A+\alpha\I)\Y_\eps(s)+\mathcal{B}(\Y_\eps(s))+\beta\mathcal{C}(\Y_\eps(s)),\D\varphi(s,\Y_\eps(s))\big)+\frac12\mathrm{Tr}(\Q\D^2\varphi(s,\Y_\eps(s)))\bigg)\d s \biggr] \nonumber\\&\qquad+
		\frac{1}{\eps}\E\bigg[\int_{t_0}^{t_0+\eps}\bigg(\delta'(s)(1+\|\Y_\eps(s))\|_{\V}^2)^m+m\delta(s)\big(\mathrm{Tr}(\Q_1)+\mathrm{Tr}(\Q)\big)(1+\|\Y_\eps(s))\|_{\V}^2)^{m-1}\bigg)\d s\bigg]
		\nonumber\\&\qquad-
		\frac{2m}{\eps}\E\biggl[\int_{t_0}^{t_0+\eps}
		\delta(s)\big((\mathcal{B}(\Y_\eps(s)),(\A+\I)\Y_\eps(s))(1+\|\Y_\eps(s))\|_{\V}^2)^{m-1}\d s\bigg]
		\nonumber\\&\qquad+\frac{2m}{\eps}
		\E\biggl[\int_{t_0}^{t_0+\eps}\delta(s)
		(\f(s,\a_\eps(s)),(\A+\I)\Y_\eps(s))\big)(1+\|\Y_\eps(s))\|_{\V}^2)^{m-1}\d s\bigg]
		\nonumber\\&\qquad+
		\frac{4m(m-1)}{\eps}\E\bigg[\int_{t_0}^{t_0+\eps}\delta(s)\big(\mathrm{Tr}(\Q_1)+\mathrm{Tr}(\Q)\big)
		(1+\|\Y_\eps(s)\|_{\V}^2)^{m-1}\d s\bigg].
	\end{align}
We now estimate all the terms in the right hand side of \eqref{vdp4} separately. Note that by the assumptions on $\varphi$, we deduce that
\begin{align}\label{unsol1}
	\|\D\varphi(\cdot,\y_\eps)\|_{\H}\leq C(1+\|\y_\eps\|_{\H}), \ \ 
	\|\varphi_t(\cdot,\y_\eps)\|_{\H}\leq C(1+\|\y_\eps\|_{\H}).
\end{align} 
		Using  Hypothesis \ref{valueH} and \eqref{unsol1}, we estimate following:
		\begin{align}
			|\ell(\Y_\eps,\a_\eps)|&\leq C\big(1+\|\Y_\eps\|_{\V}^k\big),\label{vdp6}\\
			|\mathrm{Tr}(\Q\D^2\varphi(\cdot,\Y_\eps))|, \  
			|(\f(\cdot,\a_\eps),\D\varphi(\cdot,\Y_\eps))|&\leq 
			C\big(1+\|\Y_\eps\|_{\H}\big),\label{vdp6.1}\\
			|(\f(\cdot,\a_\eps),(\A+\I)\Y_\eps)|(1+\|\Y_\eps\|_{\V}^2)^{m-1}
			&\leq C\big(1+\|\Y_\eps\|_{\V}^2\big)^m,\label{vdp6.2}\\
			|(\f(\cdot,\a_\eps),\D\varphi(\cdot,\Y_\eps))|&\leq C\big(1+\|\Y_\eps\|_{\H}\big).\label{vdp555}
		\end{align}
		We now consider following two cases for further calculations. 
		\vskip 2mm
		\noindent
		\textbf{Case-I: When $r>3$ in $d\in\{2,3\}.$}
		By using the Cauchy Schwarz and Young's inequalities, \eqref{unsol1} and calculations similar to \eqref{syymB3}, we write
		\begin{align}
		|((\mu\A+\alpha\I)\Y_\eps,\D\varphi(\cdot,\Y_\eps))|&\leq\max\{\mu,
		\alpha\}|((\A+\I)\Y_\eps,\D\varphi(\cdot,\Y_\eps))|\nonumber\\&\leq \frac{\min\{\mu,\alpha\}\lambda}{2}\|(\A+\I)\Y_\eps\|_{\H}^2+C(1+\|\Y_\eps\|_{\H}^2),\label{vdp5}\\
		|(\mathcal{B}(\Y_\eps),(\A+\I)\Y_\eps)|
		&\leq\frac{\min\{\mu,\alpha\}}{2}
		\|(\A+\I)\Y_\eps\|_{\H}^2+\frac{\beta}{4}\||\Y_\eps|^{\frac{r-1}{2}}\nabla\Y_\eps\|_{\H}^2+\varrho^*\|\nabla\Y_\eps\|_{\H}^2,\label{vdp5.1}\\
		 |(\mathcal{B}(\Y_\eps),\D\varphi(\cdot,\Y_\eps))|&\leq C(1+\|\Y_\eps\|_{\H}^2)
		+\frac{\beta\lambda}{4}\||\Y_\eps|^{\frac{r-1}{2}}\nabla\Y_\eps\|_{\H}^2+\varrho^{**}\|\nabla\Y_\eps\|_{\H}^2,\label{vdp5.2}
		\end{align}
		where 
		$\varrho^{**}:=\frac{r-3}{2\lambda(r-1)}\left[\frac{4}{\beta\lambda (r-1)}\right]^{\frac{2}{r-3}}$ and  $\varrho^*:=\frac{r-3}{2\min\{\mu,\alpha\}(r-1)}\left[\frac{4}{\beta
			\min\{\mu,\alpha\} (r-1)}\right]^{\frac{2}{r-3}}$. 
			By using a calculation similar to \eqref{ctsdep7}, \eqref{unsol1} and Remark \ref{rg3L3r}, we find
			\begin{align}\label{vdp55}
				|(\mathcal{C}(\Y_\eps),\D\varphi(\cdot,\Y_\eps))|&\leq	\|\Y_\eps\|_{\wi\L^{r+1}}^{\frac{r+3}{4}}\|\Y\|_{\wi\L^{3(r+1)}}^{\frac{3(r-1)}{4}}\|\D\varphi(\cdot,\Y_\eps)\|_{\H}
				\nonumber\\&\leq C
				\|\Y_\eps\|_{\wi\L^{r+1}}^{\frac{r+3}{4}}\left(\||\Y_\eps|^{\frac{r-1}{2}}\nabla\Y_\eps\|_{\H}^2+\|\Y_\eps\|_{\wi\L^{r+1}}^{r+1}\right)^{\frac{3(r-1)}{4(r+1)}}\|\D\varphi(\cdot,\Y_\eps)\|_{\H}
				\nonumber\\&\leq C
				\left(\||\Y_\eps|^{\frac{r-1}{2}}\nabla\Y_\eps\|_{\H}^2+\|\Y_\eps\|_{\wi\L^{r+1}}^{r+1}\right)^{\frac{r}{r+1}}\|\D\varphi(\cdot,\Y_\eps)\|_{\H}\nonumber\\&\leq C
				\left(\||\Y_\eps|^{\frac{r-1}{2}}\nabla\Y_\eps\|_{\H}^{\frac{2r}{r+1}}+\|\Y_\eps\|_{\wi\L^{r+1}}^r\right)(1+\|\Y_\eps\|_{\H})
				\nonumber\\&\leq
				\frac{\lambda}{4}\||\Y_\eps|^{\frac{r-1}{2}}
				\nabla\Y_\eps\|_{\H}^2+
				C(1+\|\Y_\eps\|_{\H})^{r+1}+\lambda\|\Y_\eps\|_{\wi\L^{r+1}}^{r+1}.
			\end{align}
			Also, from the equality \eqref{toreq}, we write
			\begin{align}\label{vdp5.3}
				(\mathcal{C}(\Y_{\eps}),(\A+\I)\Y_{\eps})\geq
				\||\Y_\eps(s)|^{\frac{r-1}{2}}\nabla\Y_\eps(s)\|_{\H}^2+
			     \|\Y_\eps(s)\|_{\wi\L^{r+1}}^{r+1}.
			\end{align}
			\vskip 2mm
		\noindent
		\textbf{Case-II: When $r=3$ with $2\beta\mu\geq1$ in $d\in\{2,3\}.$}
		From calculations similar to \eqref{syymB3}, we obtain for $0<\theta<1$
			\begin{align}
			|((\mu\A+\alpha\I)\Y_\eps,\D\varphi(\cdot,\Y_\eps))|&\leq \frac{3\min\{\mu,\alpha\}\lambda}{2}\|(\A+\I)\Y_\eps\|_{\H}^2+C(1+\|\Y_\eps\|_{\H}^2),\label{vdp5.555}\\
			|(\mathcal{B}(\Y_\eps),(\A+\I)\Y_\eps)|&\leq\||\Y_\eps|\nabla\Y_\eps\|_{\H}
			\|(\A+\I)\Y_\eps\|_{\H}
			\nonumber\\&\leq\frac{\theta\mu}{2}\|(\A+\I)\Y_\eps\|_{\H}^2+
			\frac{1}{2\theta\mu}\||\Y_\eps|\nabla\Y_\eps\|_{\H}^2,
			\label{vdp5.4}\\
			 |(\mathcal{B}(\Y_\eps),\D\varphi(\cdot,\Y_\eps))|&\leq C(1+\|\Y_\eps\|_{\H}^2)
			+\frac{\beta\lambda}{2}\||\Y_\eps|\nabla\Y_\eps\|_{\H}^2.
			\label{vdp5.8}
		\end{align}
		Similar to \eqref{vdp55}, we obtain the following estimate:
		\begin{align*}
		|(\mathcal{C}(\Y_\eps),\D\varphi(\cdot,\Y_\eps))|\leq
		\frac{\lambda}{2}\||\Y_\eps|\nabla\Y_\eps\|_{\H}^2+
		C(1+\|\Y_\eps\|_{\H})^{r+1}+\lambda\|\Y_\eps\|_{\wi\L^{r+1}}^{r+1}.
		\end{align*}
		For $r=3$, we write \eqref{toreq} as
		\begin{align}\label{vdp5.5}
			(\mathcal{C}(\Y_\eps),(\A+\I)\Y_\eps)\geq\||\Y_\eps|\nabla \Y_{\eps}\|_{\H}^{2} +\|\Y_\eps\|_{\wi\L^{r+1}}^{r+1}.
		\end{align}

	Plugging back the estimates \eqref{vdp6}-\eqref{vdp5.5} into \eqref{vdp4},
	we finally obtain
	\begin{align}\label{diffculty3}
		&\frac{m\lambda}{\eps}\E\bigg[\int_{t_0}^{t_0+\eps} \bigg(\mathfrak{a}_3\||\Y_\eps(s)|^{\frac{r-1}{2}}\nabla\Y_\eps(s)\|_{\H}^2+
		2\beta\|\Y_\eps(s)\|_{\wi\L^{r+1}}^{r+1}+
		\frac{\mathfrak{a}_2}{2}\min\{\mu,\alpha\}\|(\A+\I)\Y_\eps(s)\|_{\H}^2\bigg)
		\nonumber\\&\qquad\times
		(1+\|\Y_\eps(s)\|_{\V}^2)^{m-1}\d s\bigg]
		\nonumber\\&\leq\eps+\frac{1}{\eps} \E\biggl[\int_{t_0}^{t_0+\eps}\bigg(\big(1+\|\Y_\eps(s)\|_{\V}^k\big)+\big(1+\|\Y_\eps(s)\|_{\H}\big)+ \big(1+\|\Y_\eps(s)\|_{\H}^2\big)+\varrho^{**}\|\nabla\Y_\eps(s)\|_{\H}^2\nonumber\\&\qquad+C(1+\|\Y_\eps(s)\|_{\H}^{r+1})\bigg)\d s \biggr] \nonumber\\&\qquad+
		\frac{1}{\eps}\E\biggl[\int_{t_0}^{t_0+\eps}\bigg(\delta'(s)(1+\|\Y_\eps(s)\|_{\V}^2)^m+\varrho^*\delta(s)(1+\|\Y_\eps(s)\|_{\V}^2)^{m}\nonumber\\&\qquad
		+\delta(s)(1+\|\Y_\eps(s)\|_{\V}^2)^m+
		C(1+\|\Y_\eps(s)\|_{\V}^2)^{m-1}\nonumber\\&\qquad
		+(4m^2-3m)\delta(s)\big(\mathrm{Tr}(\Q_1)+\mathrm{Tr}(\Q)\big)(1+\|\Y_\eps(s)\|_{\V}^2)^{m-1}
		\bigg)\d s\biggr],
	\end{align}
	where $\mathfrak{a}_2$ and $\mathfrak{a}_3$ are defined in \eqref{e3cbf1}-\eqref{e3cbf}.
   Hence, by using the uniform energy estimates \eqref{vsee1}-\eqref{ssee1}, we finally obtain the uniform bound
   \begin{align}\label{vdp7}
   	\frac{\lambda}{\eps}\E\bigg[\int_{t_0}^{t_0+\eps}& \bigg(\mathfrak{a}_3
   	\||\Y_\eps(s)|^{\frac{r-1}{2}}\nabla\Y_\eps(s)\|_{\H}^2+
   	2\beta\|\Y_\eps(s)\|_{\wi\L^{r+1}}^{r+1}+
   	\frac{\mathfrak{a}_2}{2}\min\{\mu,\alpha\}\|(\A+\I)\Y_\eps(s)\|_{\H}^2\bigg)\nonumber\\&\qquad\times
   	(1+\|\Y_\eps(s)\|_{\V}^2)^{m-1}\d s\bigg]
   	\leq C,
   \end{align}
	for some constant $C$ independent of $\eps$. Therefore there exists a sequence $\eps_n\to0$ and $t_n\in(t_0,t_0+\eps_n)$ such that
	\begin{align}\label{vdp7.1}
		\E\big[\|\Y_{\eps_n}(t_n)\|_{\V_2}^2\big]\leq C, 
	\end{align}  
	and thus by an application of the Banach-Alaoglu theorem, there exists subsequences, still denoted by $\eps_n$, $t_n$, such that
	\begin{align}\label{weakL}
		\Y_{\eps_n}(t_n)\rightharpoonup\bar{\y} \ \text{ in } \ \mathrm{L}^2(\Omega;\V_2),
	\end{align}
    for some $\bar{y}\in\mathrm{L}^2(\Omega;\V_2)$, which also implies the weak convergence in $\mathrm{L}^2(\Omega;\H)$. However, from \eqref{ctsdep0.1}, we have strong convergence 
    \begin{align}\label{strgL}
    	\Y_{\eps_n}(t_n)\to\y_0 \ \text{ in } \ \mathrm{L}^2(\Omega;\H).
    \end{align}
	Thus by the uniqueness of the weak limit, \eqref{weakL} and \eqref{strgL} together yields 
	\begin{align*}
		\y_0=\bar{\y}\in\V_2.
	\end{align*}
	
	\vskip 0.2mm
	\noindent
	\textbf{To prove the supersolution inequality:}  We now prove the supersolution inequality. For this, we need to `\emph{pass to the limit}' as $\eps\to0$ in \eqref{vdp4}, at least along a subsequence. This operation is rather standard for the terms which uses the convergence in the norms of $\H$ and $\V$, because then we can use the estimates \eqref{vsee1}-\eqref{ssee1}. 
	
	To explain this, let us first discuss the convergence of the cost term appearing in the right hand side of \eqref{vdp4}. Assume that $\|\y_0\|_{\V}\leq r$. Let us define $B_r:=\{\omega\in\Omega\big|\|\Y_\eps(s)\|_{\V}\leq r\}$. Then from conditions \eqref{vh1}-\eqref{vh2} of the Hypothesis \ref{valueH}, H\"older's, Jensen's (see Lemma \ref{lem-modulus}) and Markov's inequalities, \eqref{ssee1} and \eqref{ctsdep0.2}, we calculate
	\begin{align}\label{ss1} 
	&\bigg|\frac{1}{\eps}\E\int_{t_0}^{t_0+\eps}\big[\ell(\Y_\eps(s),\a_\eps(s))-\ell(\y_0,\a_\eps(s))\big]\d s \bigg|\nonumber\\&=
	\bigg|\frac{1}{\eps}\E\int_{t_0}^{t_0+\eps}\big[\ell(\Y_\eps(s),\a_\eps(s))-\ell(\y_0,\a_\eps(s))\big]\mathds{1}_{B_r}\d s
	\nonumber\\&\qquad+
	\frac{1}{\eps}\E\int_{t_0}^{t_0+\eps}\big[\ell(\Y_\eps(s),\a_\eps(s))-\ell(\y_0,\a_\eps(s))\big]\mathds{1}_{\Omega\setminus B_r}\d s\bigg|
	\nonumber\\&\leq
	\frac{1}{\eps}\E\int_{t_0}^{t_0+\eps}\sigma_r\big(\|\Y_\eps(s)-\y_0\|_{\V}\big)\d s+
	\frac{1}{\eps}\E\int_{t_0}^{t_0+\eps}C\big(1+\|\Y_\eps(s)\|_{\V}^k+\|\y_0\|_{\V}^k\big)\mathds{1}_{\Omega\setminus B_r}\d s
	\nonumber\\&\leq
	\sigma_r\bigg(\frac{1}{\eps}\E\int_{t_0}^{t_0+\eps}\|\Y_\eps(s)-\y_0\|_{\V}\d s\bigg)+
	\frac{1}{\eps}\E\int_{t_0}^{t_0+\eps}C\big(1+\|\Y_\eps(s)\|_{\V}^k+\|\y_0\|_{\V}^k\big)\mathds{1}_{\Omega\setminus B_r}\d s \nonumber\\&\leq
	\sigma_r\bigg(\frac{1}{\eps}\int_{t_0}^{t_0+\eps}\big(\E\|\Y_\eps(s)-\y_0\|_{\V}^2\big)^{\frac12}\d s\bigg)\nonumber\\&\qquad+
	\frac{1}{\eps}\int_{t_0}^{t_0+\eps}C\bigg(1+\big(\E\|\Y_\eps(s)\|_{\V}^{2k}\big)^\frac12+\big(\E\|\y_0\|_{\V}^{2k}\big)^{\frac12}\bigg)\P\big(\Omega\setminus B_r\big)\d s\nonumber\\&\leq	\sigma_r\bigg(\frac{1}{\eps}\int_{t_0}^{t_0+\eps}\sqrt{\omega_{\y_0}
	(\eps)}\d s\bigg)+
	C\big(1+\|\y_0\|_{\V}^k\big)\frac{1+\|\y_0\|_{\V}}{r},
	\end{align}
where right hand side is going to zero by letting $\eps\to0$ and then $r\to\infty$. 
Similalry, by using the continuity of $\varphi_t$ (see Definition \ref{testD}), we obtain
\begin{align*}
	\bigg|\frac{1}{\eps}\E\int_{t_0}^{t_0+\eps}\big(\varphi_t(s,\Y_\eps(s))-
	\varphi_t(t_0,\y_0)\big)\d s\bigg|\leq\frac{1}{\eps}\int_{t_0}^{t_0+\eps}
	\E[\omega_1(\eps+\|\Y_\eps(s)-\y_0\|_{\H})]\d s,
\end{align*}
where $\omega_1$ is some local modulus of continuity (see Subsection \ref{defunfmod}) for definition. Let us now discuss the limit as $\eps\to0$ in the following terms:
\begin{align}
&\frac{1}{\eps}\E\int_{t_0}^{t_0+\eps}((\mu\A+\alpha\I)\Y_\eps(s),\D\varphi(s,\Y_\eps(s)))\d s, \label{pas1}\\	&\frac{1}{\eps}\E\int_{t_0}^{t_0+\eps}(\mathcal{B}(\Y_\eps(s)),\D\varphi(s,\Y_\eps(s)))\d s, \label{pas2}\\
&\frac{1}{\eps}\E\int_{t_0}^{t_0+\eps}(\beta\mathcal{C}(\Y_\eps(s)),\D\varphi(s,\Y_\eps(s)))\d s,  \label{pas3}\\ 
&\frac{1}{\eps}\E\int_{t_0}^{t_0+\eps}\big(\f(s,\a_\eps(s)),
\D\varphi(s,\Y_\eps(s)\big)\d s, \label{pas4}\\
&\frac{1}{\eps}\E\int_{t_0}^{t_0+\eps}\delta(s)\big(\f(s,\a_\eps(s)), (\A+\I)\Y_\eps(s)\big)\big)
(1+\|\Y_\eps(s))\|_{\V}^2)^{m-1}\d s, \label{pas5}\\
&\frac{1}{\eps}\E\int_{t_0}^{t_0+\eps}\delta(s)\|(\A+\I)\Y_\eps(s)\|_{\H}^2(1+\|\Y_\eps(s))\|_{\V}^2)^{m-1}\d s, \label{pas6}\\
 &\frac{1}{\eps}\E\int_{t_0}^{t_0+\eps}\delta(s)\big((\mu\A+\alpha\I)\Y_\eps(s),(\A+\I)\Y_\eps(s)\big)(1+\|\Y_\eps(s))\|_{\V}^2)^{m-1}\d s,
\label{pas8}\\
&\frac{1}{\eps}\E\int_{t_0}^{t_0+\eps}\delta(s)(\mathcal{B}(\Y_\eps(s)),(\A+\I)\Y_\eps(s))(1+\|\Y_\eps(s))\|_{\V}^2)^{m-1}\d s, \label{pas9}\\
&\frac{1}{\eps}\E\int_{t_0}^{t_0+\eps}\delta(s)(\mathcal{C}(\Y_\eps(s)),(\A+\I)\Y_\eps(s))(1+\|\Y_\eps(s))\|_{\V}^2)^{m-1}\d s. \label{pas10}
\end{align}
For the rest of the proof, we will use various moduli, denoted by $\upsigma(\cdot)$.
\vskip 0.2mm
\noindent
\textbf{Passing limit into \eqref{pas4}.} It is given that $\f:[0,T]\times\U\to\V$ is uniformly continuous, uniformly for $\a\in\U$, therefore by the definition of modulus of continuity, we write
\begin{align}\label{modf}
	\|\f(s,\a_\eps(s))-\f(t,\a_\eps(s))\|_{\V}\leq\omega_{\f}(|s-t|), \ \text{ for all } \ s,t\in[t_0,t_0+\eps],
\end{align} 
for some modulus of continuity $\omega_{\f}$. Similarly, since $\D\varphi$ is uniformly continuous on $[t_0,t_0+\eps]\times\H$ for every $\eps>0$ (see Definition \ref{testD}), we write
\begin{align}\label{modphi}
	\|\D\varphi(s,\Y_\eps(s)-\D\varphi(t_0,\y_0)\|_{\H}\leq
	\omega_2(\eps+\|\Y_\eps(s)-\y_0\|_{\H}), \ \text{ for all } \ s,t\in[t_0,t_0+\eps],
\end{align}
where $\omega_2$ is some local modulus of continuity. From \eqref{modf}-\eqref{modphi}, we estimate \eqref{pas4} as
\begin{align*}
	&\bigg|\frac{1}{\eps_n}\E\int_{t_0}^{t_0+\eps_n}\big[\big(\f(s,\a_{\eps_n}(s)),
	\D\varphi(s,\Y_{\eps_n}(s)\big)-\big(\f(t_0,\a_{\eps_n}(s)),
	\D\varphi(t_0,\y_0)\big)\big]\d s \bigg|\nonumber\\&\leq
	\bigg|\frac{1}{\eps_n}\E\int_{t_0}^{t_0+{\eps_n}}
	\big[\big(\f(s,\a_{\eps_n}(s)),
	\D\varphi(s,\Y_{\eps_n}(s)-\D\varphi(t_0,\y_0)\big)\big]\d s\bigg|
	\nonumber\\&\qquad+
	\bigg|\frac{1}{\eps_n}\E\int_{t_0}^{t_0+{\eps_n}}
	\big[\big(\f(s,\a_{\eps_n}(s))
	-\f(t_0,\a_{\eps_n}(s)),\D\varphi(t_0,\y_0)\big)\big]\d s\bigg|
	\nonumber\\&\leq
	\frac{C}{\eps_n}\E\int_{t_0}^{t_0+\eps_n}\|\f(s,\a_{\eps_n}(s))\|_{\V}
	\|\D\varphi(s,\Y_{\eps_n}(s)-\D\varphi(t_0,\y_0)\|_{\H}\d s
	\nonumber\\&\quad+
	\frac{C}{\eps_n}\E\int_{t_0}^{t_0+{\eps_n}}\|\f(s,\a_{\eps_n}(s))-\f(t_0,\a_{\eps_n}(s)\|_{\V}\|\D\varphi(t_0,\y_0)\|_{\H}\d s
	\nonumber\\&\leq
	\frac{C}{\eps_n}\E\int_{t_0}^{t_0+\eps_n}
	\omega_2(\eps_n+\|\Y_{\eps_n}(s)-\y_0\|_{\H})\d s+
	\frac{C}{\eps_n}\E\int_{t_0}^{t_0+{\eps_n}}\omega_{\f}(\eps_n)\d s
	\nonumber\\&\leq\upsigma(\eps_n).
\end{align*}

	\vskip 0.2mm
	\noindent
	\textbf{Passing limit into \eqref{pas6} and \eqref{pas8}.} 
	From \eqref{vdp7} and H\"older's inequality, we first notice that
	\begin{align}\label{vdp8}
	&\E\left\|\frac{1}{\eps}\int_{t_0}^{t_0+\eps}\sqrt{\delta(s)}
	(\A+\I)\Y_\eps(s)(1+\|\Y_\eps(s))\|_{\V}^2)^{\frac{m-1}{2}}\d s \right\|_{\H}^2\nonumber\\&\leq\E
	\frac{1}{\eps}\int_{t_0}^{t_0+\eps}
	\delta(s)\|(\A+\I)\Y_\eps(s)\|_{\H}^2(1+\|\Y_\eps(s))\|_{\V}^2)^{m-1}\d s
	\leq C.
	\end{align}
	Therefore, an application of the Banach-Alaoglu theorem guarantees the existence of a sequence $\eps_n\to0$ and $\wi\y\in\mathrm{L}^2(\Omega;\H)$ such that 
	\begin{align}\label{wknm}
	\wi\y_n:=\frac{1}{\eps_n}\int_{t_0}^{t_0+\eps_n}\sqrt{\delta(s)}(\A+\I)\Y_{\eps_n}(s)(1+\|\Y_{\eps_n}(s))\|_{\V}^2)^{\frac{m-1}{2}}
	\d s\rightharpoonup\wi\y \  \text{ in } \ \mathrm{L}^2(\Omega;\H),
	\end{align}
	as $n\to\infty$. 
	Arguing similarly as in \eqref{ss1}, one can prove that 
	\begin{align}\label{wknmdiff}
	(\A+\I)^{-1}\wi\y_n&=\frac{1}{\eps_n}\int_{t_0}^{t_0+\eps_n}\sqrt{\delta(s)}\Y_{\eps_n}(s)(1+\|\Y_{\eps_n}(s))\|_{\V}^2)^{\frac{m-1}{2}}
	\d s\\&\to\sqrt{\delta(t_0)} \y_0\left(1+\|\y_0\|_{\V}^2\right)^{\frac{m-1}{2}}
	\end{align}
  in $\mathrm{L}^2(\Omega;\H)$ as $n\to\infty$. Therefore, by the uniqueness of the weak limit, it follows that  $$\wi\y=\sqrt{\delta(t_0)}(\A+\I) \y_0\left(1+\|\y_0\|_{\V}^2\right)^{\frac{m-1}{2}}.$$
  Moreover, in view of \eqref{ssee1}, \eqref{ctsdep0.1}-\eqref{ctsdep0.2}, one can also verify the following limits:
  \begin{align}
  	\frac{1}{\eps_n}\int_{t_0}^{t_0+\eps_n}\delta(s)\|\Y_{\eps_n}(s)\|_{\H}^2(1+\|\Y_{\eps_n}(s))\|_{\V}^2)^{m-1}\d s&\to\delta(t_0) \|\y_0\|_{\H}^2\left(1+\|\y_0\|_{\V}^2\right)^{m-1}
  	\label{wknm1.1}
  	\end{align}
  and 
  	\begin{align}
  	\frac{1}{\eps_n}\int_{t_0}^{t_0+\eps_n} \delta(s)\|\nabla\Y_{\eps_n}(s)\|_{\H}^2(1+\|\Y_{\eps_n}(s))\|_{\V}^2)^{m-1}\d s&\to\delta(t_0) \|\nabla\y_0\|_{\H}^2 \left(1+\|\y_0\|_{\V}^2\right)^{m-1},\label{wknm1.2}
  \end{align} 
  in $\mathrm{L}^2(\Omega;\H)$ as $n\to\infty$. From \eqref{vdp8} and \eqref{wknm}, using Jensen's inequality and  the weak lower semicontinuity property of norm, we obtain
 \begin{align}\label{wknm1}
 	&\liminf_{n\to\infty}\E\frac{1}{\eps_n}\int_{t_0}^{t_0+\eps_n}\delta(s) \|(\A+\I)\Y_{\eps_n}(s)\|_{\H}^2(1+\|\Y_{\eps_n}(s))\|_{\V}^2)^{m-1}
 	\d s\nonumber\\&\geq\liminf_{n\to\infty} \E\left\|\frac{1}{\eps_n}\int_{t_0}^{t_0+{\eps_n}}
 	\sqrt{\delta(s)}(\A+\I)\Y_{\eps_n}(s)(1+\|\Y_{\eps_n}(s))\|_{\V}^2)^{\frac{m-1}{2}}\d s \right\|_{\H}^2\nonumber\\&\geq
 	\delta(t_0)\|(\A+\I)\y_0\|_{\H}^2(1+\|\y_0\|_{\V}^2)^{m-1}.
 \end{align}
This will take care of the term \eqref{pas6}. A similar argument as we performed  above yields 
 \begin{align}\label{wknm2}
 	\frac{1}{\eps_n}\int_{t_0}^{t_0+\eps_n}(\A+\I)\Y_{\eps_n}(s)\d s\rightharpoonup(\A+\I)\y_0 \  \text{ in } \ \mathrm{L}^2(\Omega;\H)
 	\  \text{ as } \ n\to\infty.
 \end{align}
Moreover, since $\|\A\Y_{\eps_n}\|_{\H}^2=\|(\A+\I)\Y_{\eps_n}\|_{\H}^2-\|\Y_{\eps_n}\|_{\H}^2-2\|\nabla\Y_{\eps_n}\|_{\H}^2$, therefore by using \eqref{wknm1.1}-\eqref{wknm1}, we have the following lower bound:
\begin{align}\label{wknm2.2}
	&\liminf_{n\to\infty}\E\frac{1}{\eps_n}\int_{t_0}^{t_0+\eps_n}\delta(s) \|\A\Y_{\eps_n}(s)\|_{\H}^2(1+\|\Y_{\eps_n}(s))\|_{\V}^2)^{m-1}
	\d s\nonumber\\&\geq
	\delta(t_0)\|\A\y_0\|_{\H}^2(1+\|\y_0\|_{\V}^2)^{m-1}.
\end{align}
Finally, in view of \eqref{wknm1.1}-\eqref{wknm1.2} and \eqref{wknm2.2}, the following limit is immediate:
\begin{align}\label{wknmpas8}
	&\liminf_{n\to\infty}\E\frac{1}{\eps_n}\int_{t_0}^{t_0+\eps_n}\delta(s)\big((\mu\A+\alpha\I)\Y_\eps(s),(\A+\I)\Y_\eps(s)\big)
	(1+\|\Y_\eps(s))\|_{\V}^2)^{m-1}\d s
	 \nonumber\\&\geq
	\delta(t_0)\big((\mu\A+\alpha\I)\y_0,(\A+\I)\y_0\big)
	(1+\|\y_0\|_{\V}^2)^{m-1}.
\end{align}


\vskip 0.2mm
\noindent
\textbf{Passing limit into the linear term \eqref{pas1}.} 
Let us denote by $\omega_{\varphi}$, a modulus of continuity of $\D\varphi$. Then by using H\"older's inequality, \eqref{vdp7}, \eqref{ctsdep0.1} and \eqref{wknm2}, we calculate
\begin{align}\label{wknm3}
&\bigg|\frac{1}{\eps_n}\E\int_{t_0}^{t_0+\eps_n}\big((\mu\A+\alpha\I)
\Y_{\eps_n}(s),\D\varphi(s,\Y_{\eps_n}(s))\big)\d s-
\big((\mu\A+\alpha\I)\y_0,\D\varphi(t_0,\y_0)\big)\bigg|
\nonumber\\&\leq
\bigg|\frac{1}{\eps_n}\E\int_{t_0}^{t_0+\eps_n}\big((\mu\A+\alpha\I)
\Y_{\eps_n}(s),\D\varphi(s,\Y_{\eps_n}(s))-\D\varphi(t_0,\y_0)\big)\d s \bigg|\nonumber\\&\quad+
\bigg|\frac{1}{\eps_n}\E\int_{t_0}^{t_0+\eps_n}\big((\mu\A+\alpha\I)
\big(\Y_{\eps_n}(s)-\y_0\big),\D\varphi(t_0,\y_0)\big)\d s \bigg|\nonumber\\&\leq
\frac{1}{\eps_n}\int_{t_0}^{t_0+\eps_n}\big(\E\|(\mu\A+\alpha\I)
\Y_{\eps_n}(s)\|_{\H}^2\big)^{\frac12}\bigg(\E\big(\omega_\varphi\big(\eps_n+\|\Y_{\eps_n}(s)-\y_0\|_{\H}\big)\big)^2\bigg)^{\frac12}\d s
\nonumber\\&\quad+
\bigg|\E\bigg(\frac{1}{\eps_n}\int_{t_0}^{t_0+\eps_n}
(\mu\A+\alpha\I)\big(\Y_{\eps_n}(s)-\y_0\big)\d s,\D\varphi(t_0,\y_0)
\bigg)\bigg|\nonumber\\&\leq
\left(\frac{1}{\eps_n}\int_{t_0}^{t_0+\eps_n}\E\|(\mu\A+\alpha\I)
\Y_{\eps_n}(s)\|_{\H}^2\d s\right)^{\frac12} \left(\frac{1}{\eps_n}\int_{t_0}^{t_0+\eps_n}\E\big(\omega_\varphi\big(\eps_n+\|\Y_{\eps_n}(s)-\y_0\|_{\H}\big)\big)^2\d s\right)^{\frac12}
\nonumber\\&\quad+
\bigg|\E\bigg(\frac{1}{\eps_n}\int_{t_0}^{t_0+\eps_n}
(\mu\A+\alpha\I)\big(\Y_{\eps_n}(s)-\y_0\big)\d s,\D\varphi(t_0,\y_0)
\bigg)\bigg|\nonumber\\&\leq\upsigma(\eps_n).
\end{align}

\vskip 0.2mm
\noindent
\textbf{Passing limit into the term \eqref{pas5}.} From the Hypothesis \ref{fhyp} and \eqref{ctsdep0.2} together with the definition of modulus of continuity, one can conclude that 
{\small
\begin{align*} 
	&\bigg|\frac{1}{\eps_n}\E\int_{t_0}^{t_0+\eps_n}\delta(s)\big(\f(s,\a_{\eps_n}(s)), (\A+\I)\Y_{\eps_n}(s)\big)
	(1+\|\Y_{\eps_n}(s))\|_{\V}^2)^{m-1}\d s
	\nonumber\\&\quad-
	\frac{1}{\eps_n}\E\int_{t_0}^{t_0+\eps_n}\delta(t_0)\big(\f(s,\a_{\eps_n}(s)), (\A+\I)\y_0\big)\big)
	(1+\|\y_0\|_{\V}^2)^{m-1}\d s\bigg|\nonumber\\&\leq
	\frac{1}{\eps_n}\E\int_{t_0}^{t_0+\eps_n}\delta(s)\bigg|\big((\A+\I)^{\frac12}\f(s,\a_{\eps_n}(s)),(\A+\I)^{\frac12}\Y_{\eps_n}(s)-(\A+\I)^{\frac12}\y_0\big)\bigg|(1+\|\Y_{\eps_n}(s))\|_{\V}^2)^{m-1}\d s
	\nonumber\\&\quad+
	\frac{1}{\eps_n}\E\int_{t_0}^{t_0+\eps_n}\delta(s)\big|\big((\A+\I)^{\frac12}\f(s,\a_{\eps_n}(s)),(\A+\I)^{\frac12}\y_0\big)\big|\bigg|(1+\|\Y_{\eps_n}(s))\|_{\V}^2)^{m-1}-(1+\|\y_0\|_{\V}^2)^{m-1}\bigg|\d s
	\nonumber\\&\quad+
	\frac{1}{\eps_n}\E\int_{t_0}^{t_0+\eps_n}\big|\delta(s)-\delta(t_0)\big|
	\big|\big((\A+\I)^{\frac12}\f(s,\a_{\eps_n}(s)),(\A+\I)^{\frac12}\y_0\big)\big|(1+\|\y_0\|_{\V}^2)^{m-1}\d s
	\nonumber\\&\leq
	C\left(\frac{1}{\eps_n}\int_{t_0}^{t_0+\eps_n}\E\|\Y_{\eps_n}(s)-\y_0\|_{\V}^2\d s\right)^{\frac12} \left(\int_{t_0}^{t_0+\eps_n}\E((1+\|\Y_{\eps_n}(s))\|_{\V}^2)^{2(m-1)})
	\d s\right)^{\frac12}
	\nonumber\\&\quad+
	\frac{C}{\eps_n}\E\int_{t_0}^{t_0+\eps_n}\bigg|(1+\|\Y_{\eps_n}(s))\|_{\V}^2)^{m-1}-(1+\|\y_0\|_{\V}^2)^{m-1}\bigg|\d s
	\nonumber\\&\quad+
	\frac{C}{\eps_n}\int_{t_0}^{t_0+\eps_n}\big|\delta(s)-\delta(t_0)\big|
	\d s
	\nonumber\\&\leq\upsigma(\eps_n).
\end{align*}}

\vskip 0.2mm
\noindent
\textbf{Passing limit into the Navier-Stokes nonlinearity \eqref{pas2}.} 
 By using \eqref{syymB}, we write
\begin{align}\label{wknm5}
	&\bigg|\frac{1}{\eps_n}\E\int_{t_0}^{t_0+\eps_n}(\mathcal{B}(\Y_{\eps_n}(s)),\D\varphi(s,\Y_{\eps_n}(s)))\d s-
	(\mathcal{B}(\y_0),\D\varphi(t_0,\y_0))\bigg|
	\nonumber\\&\leq
	\underbrace{\bigg|\frac{1}{\eps_n}\E\int_{t_0}^{t_0+\eps_n}
	(\mathcal{B}(\Y_{\eps_n}(s)),\D\varphi(s,\Y_{\eps_n}(s))-\D\varphi(t_0,\y_0))\d s \bigg|}_{:=J_1}
	\nonumber\\&\quad+
	\underbrace{\bigg|\frac{1}{\eps_n}\E\int_{t_0}^{t_0+\eps_n}
      (\mathcal{B}(\Y_{\eps_n}(s)-\y_0,\Y_{\eps_n}(s)),\D\varphi(t_0,\y_0))\d s \bigg|}_{:=J_2}
	\nonumber\\&\quad+
	\underbrace{\bigg|\frac{1}{\eps_n}\E\int_{t_0}^{t_0+\eps_n}(\mathcal{B}(\y_0,\Y_{\eps_n}(s)-\y_0),\D\varphi(t_0,\y_0))\d s \bigg|}_{:=J_3}.
\end{align}	
By using Agmon's and H\"older's inequalities, and energy estimates \eqref{vsee1}-\eqref{ssee1}, we calculate
\begin{align}\label{J1B}
	J_1&\leq\frac{1}{\eps_n}\E\int_{t_0}^{t_0+\eps_n}
     \|\mathcal{B}(\Y_{\eps_n}(s))\|_{\H}\left(\omega_\varphi\big(\eps_n+\|\Y_{\eps_n}(s)-\y_0\|_{\H}\big)\right)\d s
     \nonumber\\&\leq
      \frac{1}{\eps_n}\E\int_{t_0}^{t_0+\eps_n}
     \|\Y_{\eps_n}(s)\|_{\H}^{1-\frac{d}{4}}\|(\A+\I)\Y_{\eps_n}(s)\|_{\H}^{\frac{d}{4}}\|\nabla\Y_{\eps_n}(s)\|_{\H}
     \left(\omega_\varphi\big(\eps_n+\|\Y_{\eps_n}(s)-\y_0\|_{\H}\big)
     \right)\d s     
     \nonumber\\&\leq\frac{1}{\eps_n}\int_{t_0}^{t_0+\eps_n}
     \left(\E\|\Y_{\eps_n}(s)\|_{\H}^2\right)^{\frac{4-d}{8}}
     \left(\E\|(\A+\I)\Y_{\eps_n}(s)\|_{\H}^{2}\right)^{\frac{d}{8}}
     \left(\E\|\nabla\Y_{\eps_n}(s)\|_{\H}^4\right)^{\frac14}\nonumber\\&\qquad\times
     \left(\E\big(\omega_\varphi\big(\eps_n+\|\Y_{\eps_n}(s)-\y_0\|_{\H}
     \big)\big)^{4}\right)^{\frac14}\d s
     \nonumber\\&\leq
     \left(\frac{1}{\eps_n}\int_{t_0}^{t_0+\eps_n}
     \E\|\Y_{\eps_n}(s)\|_{\H}^2\right)^{\frac{4-d}{8}}\left(\frac{1}{\eps_n}\int_{t_0}^{t_0+\eps_n}\E\|(\A+\I)\Y_{\eps_n}(s)\|_{\H}^{2}\d s
     \right)^{\frac{d}{8}}\nonumber\\&\qquad\times
     \left(\frac{1}{\eps_n}\int_{t_0}^{t_0+\eps_n}
     \E\|\nabla\Y_{\eps_n}(s)\|_{\H}^{2}\d s\right)^{\frac{1}{4}} \left(\frac{1}{\eps_n}\int_{t_0}^{t_0+\eps_n}
     \E\big(\omega_\varphi\big(\eps_n+\|\Y_{\eps_n}(s)-\y_0\|_{\H}
     \big)\big)^4\d s\right)^{\frac{1}{4}}
     \nonumber\\&\leq C
     \left(\frac{1}{\eps_n}\int_{t_0}^{t_0+\eps_n}\E\|(\A+\I)\Y_{\eps_n}(s)\|_{\H}^{2}\d s\right)^{\frac{d}{8}}
     \left(\frac{1}{\eps_n}\int_{t_0}^{t_0+\eps_n} \E\big(\omega_\varphi\big(\eps_n+\|\Y_{\eps_n}(s)-\y_0\|_{\H}
     \big)\big)^{4}\d s\right)^{\frac{1}{4}} \nonumber\\&\leq\upsigma(\eps_n).
\end{align}
Similalry, we calculate
\begin{align}\label{J2B}
	J_2&\leq\frac{1}{\eps_n}\E\int_{t_0}^{t_0+\eps_n}
	\|\mathcal{B}(\Y_{\eps_n}(s)-\y_0,\Y_{\eps_n}(s))\|_{\H}\left(\omega_\varphi\big(\eps_n+\|\Y_{\eps_n}(s)-\y_0\|_{\H}\big)\right)\d s
	\nonumber\\&\leq
	 \frac{1}{\eps_n}\E\int_{t_0}^{t_0+\eps_n}
	\|\Y_{\eps_n}(s)-\y_0\|_{\H}^{1-\frac{d}{4}}\|(\A+\I)(\Y_{\eps_n}(s)-\y_0)\|_{\H}^{\frac{d}{4}}\|\nabla\Y_{\eps_n}(s)\|_{\H}
	\nonumber\\&\qquad\times
	\left(\omega_\varphi\big(\eps_n+\|\Y_{\eps_n}(s)-\y_0\|_{\H}\big)
	\right)\d s \nonumber\\&\leq
	\frac{C}{\eps_n}\E\int_{t_0}^{t_0+\eps_n}
	\|\Y_{\eps_n}(s)-\y_0\|_{\H}^{1-\frac{d}{4}}\left(\|(\A+\I)\Y_{\eps_n}(s)\|_{\H}^{\frac{d}{4}}+\|(\A+\I)\y_0\|_{\H}^{\frac{d}{4}}\right)\|\nabla\Y_{\eps_n}(s)\|_{\H}
	\nonumber\\&\qquad\times
	\left(\omega_\varphi\big(\eps_n+\|\Y_{\eps_n}(s)-\y_0\|_{\H}\big)
	\right)\d s
	\nonumber\\&\leq
	 \left(\frac{1}{\eps_n}\int_{t_0}^{t_0+\eps_n}
	\E\|\Y_{\eps_n}(s)-\y_0\|_{\H}^2 \d s\right)^{\frac{4-d}{8}} \bigg[\left(\frac{1}{\eps_n}\int_{t_0}^{t_0+\eps_n}\E\|(\A+\I)\Y_{\eps_n}(s)\|_{\H}^{2}\d s
	\right)^{\frac{d}{8}}\nonumber\\&\quad+\left(\frac{1}{\eps_n}\int_{t_0}^{t_0+\eps_n}\E\|(\A+\I)\y_0\|_{\H}^{2}\d s\right)^{\frac{d}{8}}\bigg] \left(\frac{1}{\eps_n}\int_{t_0}^{t_0+\eps_n}
	\E\|\nabla\Y_{\eps_n}(s)\|_{\H}^{2}\d s\right)^{\frac{1}{4}}\nonumber\\&\qquad\times
	 \left(\frac{1}{\eps_n}\int_{t_0}^{t_0+\eps_n}
	\E\big(\omega_\varphi\big(\eps_n+\|\Y_{\eps_n}(s)-\y_0\|_{\H}
	\big)\big)^4\d s\right)^{\frac{1}{4}}
	\nonumber\\&\leq\upsigma(\eps_n).
\end{align}
Also, we estimate the final term in \eqref{wknm5} as
\begin{align}\label{J3B}
	J_3&\leq\frac{1}{\eps_n}\E\int_{t_0}^{t_0+\eps_n}
	\|\mathcal{B}(\y_0,\Y_{\eps_n}(s)-\y_0)\|_{\H}\left(\omega_\varphi\big(\eps_n+\|\Y_{\eps_n}(s)-\y_0\|_{\H}\big)\right)\d s
	\nonumber\\&\leq
	 \frac{1}{\eps_n}\E\int_{t_0}^{t_0+\eps_n}
	\|\y_0\|_{\H}^{1-\frac{d}{4}}\|(\A+\I)\y_0\|_{\H}^{\frac{d}{4}}
	\|\nabla(\Y_{\eps_n}(s)-\y_0)\|_{\H}
	\left(\omega_\varphi\big(\eps_n+\|\Y_{\eps_n}(s)-\y_0\|_{\H}\big)
	\right)\d s \nonumber\\&\leq 
	\frac{1}{\eps_n}\int_{t_0}^{t_0+\eps_n}
	\left(\E\big[\|\y_0\|_{\H}^{4-d}\|(\A+\I)\y_0\|_{\H}^{d}\big]\right)^{\frac14}
	\left(\E\|\nabla(\Y_{\eps_n}(s)-\y_0)\|_{\H}^2\right)^{\frac12}
\nonumber	\\&\qquad\times
	\left(\E\big(\omega_\varphi\big(\eps_n+\|\Y_{\eps_n}(s)-\y_0\|_{\H}
	\big)\big)^{4}\right)^{\frac14}\d s
	\nonumber\\&\leq
	\|\y_0\|_{\H}^{4-d}\|(\A+\I)\y_0\|_{\H}^{d}
	\left(\frac{1}{\eps_n}\int_{t_0}^{t_0+\eps_n}
	\E\|\nabla(\Y_{\eps_n}(s)-\y_0)\|_{\H}^{2}\d s\right)^{\frac{1}{2}} \nonumber\\&\qquad\times
	\left(\frac{1}{\eps_n}\int_{t_0}^{t_0+\eps_n}
	\E\big(\omega_\varphi\big(\eps_n+\|\Y_{\eps_n}(s)-\y_0\|_{\H}
    \big)\big)^4\d s\right)^{\frac{1}{4}}
    \nonumber\\&\leq
    \|\y_0\|_{\H}^{4-d}\|(\A+\I)\y_0\|_{\H}^{d}
    \left(\frac{1}{\eps_n}\int_{t_0}^{t_0+\eps_n}
    \omega_{\y_0}(\eps_n)\d s\right)^{\frac{1}{2}} \nonumber\\&\qquad\times
    \left(\frac{1}{\eps_n}\int_{t_0}^{t_0+\eps_n}
    	\E\big(\omega_\varphi\big(\eps_n+\|\Y_{\eps_n}(s)-\y_0\|_{\H}
    	\big)\big)^4\d s\right)^{\frac{1}{4}}
    \nonumber\\&\leq\upsigma(\eps_n).
\end{align}
Similar calculations can be performed for \eqref{pas9} (see Appendix \ref{wknBA1} for a detailed explanation) to obtain 
\begin{align*}
\bigg|\frac{1}{\eps_n}\E\int_{t_0}^{t_0+\eps_n}&\delta(s)
\big(\mathcal{B}(\Y_\eps(s)),\A\Y_\eps(s)\big)
(1+\|\Y_\eps(s))\|_{\V}^2)^{m-1} \d s
\nonumber\\&-
\delta(t_0)\big(\mathcal{B}(\y_0),\A\y_0\big)(1+\|\y_0\|_{\V}^2)^{m-1}
\bigg|\leq\upsigma(\eps_n).
\end{align*}

\vskip 0.2mm
\noindent
\textbf{Passing limit into Forchheimer nonlinearity \eqref{pas3}.} 
We now deal with the term \eqref{pas3}. For this, we write 
\begin{align}\label{wknm4}
	&\bigg|\frac{1}{\eps_n}\E\int_{t_0}^{t_0+\eps_n}(\mathcal{C}(\Y_{\eps_n}(s)),\D\varphi(s,\Y_{\eps_n}(s)))\d s-
	(\mathcal{C}(\y_0),\D\varphi(t_0,\y_0))\bigg|
	\nonumber\\&\leq
	\bigg|\frac{1}{\eps_n}\E\int_{t_0}^{t_0+\eps_n}(\mathcal{C}
	(\Y_{\eps_n}(s)),\D\varphi(s,\Y_{\eps_n}(s))-\D\varphi(t_0,\y_0)) 
	\d s\bigg|\nonumber\\&\quad+
	\bigg|\frac{1}{\eps_n}\E\int_{t_0}^{t_0+\eps_n}(\mathcal{C}
	(\Y_{\eps_n}(s))-\mathcal{C}(\y_0),\D\varphi(t_0,\y_0))\d s \bigg|.
\end{align}	
From the calculation \eqref{vdp55}, H\"older's inequality 
and energy estimates \eqref{ssee1}, we calculate
	\begin{align}\label{wknm4.1}
		&\bigg|\frac{1}{\eps_n}\E\int_{t_0}^{t_0+\eps_n}(\mathcal{C}
		(\Y_{\eps_n}(s)),\D\varphi(s,\Y_{\eps_n}(s))-\D\varphi(t_0,\y_0)) 
		\d s\bigg|\nonumber\\&\leq C
		\bigg|\frac{1}{\eps_n}\E\int_{t_0}^{t_0+\eps_n}
		\left(\||\Y_{\eps_n}(s)|^{\frac{r-1}{2}}\nabla\Y_{\eps_n}(s)\|_{\H}^2+\|\Y_{\eps_n}\|_{\wi\L^{r+1}}^{r+1}\right)^{\frac{r}{r+1}}\omega_\varphi\big(\eps_n+\|\Y_{\eps_n}(s)-\y_0\|_{\H}\big)\d s\bigg|\nonumber\\&\leq
		C\bigg|\frac{1}{\eps_n}\E\int_{t_0}^{t_0+\eps_n} \left(\||\Y_{\eps_n}(s)|^{\frac{r-1}{2}}\nabla\Y_{\eps_n}(s)\|_{\H}^{\frac{2r}{r+1}}+
		\|\Y_{\eps_n}(s)\|_{\wi\L^{r+1}}^r\right)\omega_\varphi\big(\eps_n+\|\Y_{\eps_n}(s)-\y_0\|_{\H}\big)\d s\bigg|\nonumber\\&\leq
		\frac{C}{\eps_n}\int_{t_0}^{t_0+\eps_n} \big(\E\||\Y_{\eps_n}(s)|^{\frac{r-1}{2}}
		\nabla\Y_{\eps_n}(s)\|_{\H}^2\big)^{\frac{r}{r+1}}
		\left(\E\big(\omega_\varphi\big(\eps_n+\|\Y_{\eps_n}(s)-\y_0\|_{\H}\big)\big)^{r+1}\right)^{\frac{1}{r+1}}\d s\nonumber\\&+
		\frac{C}{\eps_n}\int_{t_0}^{t_0+\eps_n} \big(\E\|\Y_{\eps_n}(s)\|_{\wi\L^{r+1}}^{r+1}\big)^{\frac{r}{r+1}}
		\left(\E\big(\omega_\varphi\big(\eps_n+\|\Y_{\eps_n}(s)-\y_0\|_{\H}\big)\big)^{r+1}\right)^{\frac{1}{r+1}}\d s\nonumber\\&\leq C
		\left(\frac{1}{\eps_n}\int_{t_0}^{t_0+\eps_n}
		\E\||\Y_{\eps_n}(s)|^{\frac{r-1}{2}}\nabla\Y_{\eps_n}(s)\|_{\H}^2\d s \right)^{\frac{r}{r+1}}
		\nonumber\\&\qquad\times\left(\frac{1}{\eps_n}\int_{t_0}^{t_0+\eps_n}
		\E\big(\omega_\varphi\big(\eps_n+\|\Y_{\eps_n}(s)-\y_0\|_{\H}\big)
		\big)^{r+1}\d s\right)^{\frac{1}{r+1}}\nonumber\\&+
		C\left(\frac{1}{\eps_n}\int_{t_0}^{t_0+\eps_n}
		\E\|\Y_{\eps_n}(s)\|_{\wi\L^{r+1}}^{r+1}\d s\right)^{\frac{r}{r+1}} \left(\frac{1}{\eps_n}\int_{t_0}^{t_0+\eps_n}
		\E\big(\omega_\varphi\big(\eps_n+\|\Y_{\eps_n}(s)-\y_0\|_{\H}\big)
		\big)^2\d s\right)^{\frac{1}{2}}.
 	\end{align}
Similarly, by using  Taylor's formula, H\"older's and interpolation inequalities, Sobolev embedding $\V\hookrightarrow\wi\L^6$, and \eqref{ctsdep0.2}, we estimate
	\begin{align}\label{wknm4.2}
	&\bigg|\frac{1}{\eps_n}\E\int_{t_0}^{t_0+\eps_n}(\mathcal{C}
	(\Y_{\eps_n}(s))-\mathcal{C}(\y_0),\D\varphi(t_0,\y_0))\d s \bigg|\nonumber\\&\leq
	\frac{1}{\eps_n}\E\int_{t_0}^{t_0+\eps_n}\||\Y_{\eps_n}(s)|^{r-1}\Y_{\eps_n}(s)-|\y_0|^{r-1}\y_0\|_{\H}\|\D\varphi(t_0,\y_0)\|_{\H}\d s \nonumber\\&\leq
	\frac{1}{\eps_n}\E\int_{t_0}^{t_0+\eps_n}
	\||\Y_{\eps_n}(s)|+|\y_0|\|_{\wi\L^{3(r+1)}}^{r-1}\|\Y_{\eps_n}(s)-\y_0\|_{\wi\L^{\frac{6(r+1)}{r+5}}}\|\D\varphi(t_0,\y_0)\|_{\H}\d s
	\nonumber\\&\leq
	\frac{C}{\eps_n}\int_{t_0}^{t_0+\eps_n}
	\left(\E\||\Y_{\eps_n}(s)|+|\y_0|\|_{\wi\L^{3(r+1)}}^{r+1}\right)^{\frac{r-1}{r+1}}\left(\E\|\Y_{\eps_n}(s)-\y_0\|_{\V}^{\frac{r+1}{2}}
	\right)^{\frac{2}{r+1}}\d s
	\nonumber\\&\leq
	\frac{C}{\eps_n}\int_{t_0}^{t_0+\eps_n}
	\left(\E\||\Y_{\eps_n}(s)|+|\y_0|\|_{\wi\L^{3(r+1)}}^{r+1}\right)^{\frac{r-1}{r+1}}\left(\E\|\Y_{\eps_n}(s)-\y_0\|_{\V}^{2}\right)^{\frac{1}{r+1}}
	\left(\E\|\Y_{\eps_n}(s)-\y_0\|_{\V}^{r-1}\right)^{\frac{1}{r+1}}\d s
	\nonumber\\&\leq\left(\sup\limits_{s\in[t_0,t_0+\eps_n]} \left(\E\|\Y_{\eps_n}(s)-\y_0\|_{\V}^{r-1}\right)^{\frac{1}{r+1}}
	\right)\left(\frac{1}{\eps_n}\int_{t_0}^{t_0+\eps_n}
	\E\||\Y_{\eps_n}(s)|+|\y_0|\|_{\wi\L^{3(r+1)}}^{r+1}\d s \right)^ {\frac{r-1}{r+1}}
	\nonumber\\&\quad\times\left(\frac{1}{\eps_n}\int_{t_0}^{t_0+\eps_n}
	\E\|\Y_{\eps_n}(s)-\y_0\|_{\V}^2\d s \right)^{\frac12} 
	\nonumber\\&\leq C
	\left(\frac{1}{\eps_n} \int_{t_0}^{t_0+\eps_n}
	\sqrt{\omega_{\y_0}(\eps_n)}\d s\right)
	\nonumber\\&\leq\upsigma(\eps_n),
	\end{align}
	where in the last line of the above inequality, we have used \eqref{ssee1} and \eqref{ctsdep0.2}. In a similar way, one can establish that (see Appendix \ref{wknCA1} for a detailed explanation) 
	\begin{align*}
		\bigg|\frac{1}{\eps_n}\E\int_{t_0}^{t_0+\eps_n}&\delta(s)
		\big(\mathcal{C}(\Y_\eps(s)),(\A+\I)\Y_\eps(s)\big)
		(1+\|\Y_\eps(s))\|_{\V}^2)^{m-1} \d s
		\nonumber\\&-
		\delta(t_0)\big(\mathcal{C}(\y_0),(\A+\I)\y_0\big)
		(1+\|\y_0\|_{\V}^2)^{m-1}\bigg|
		\leq\upsigma(\eps_n),
	\end{align*}
	for the values of $r$ given  in Table \ref{Table2}. Finally, we combine all the above convergences \eqref{ss1}-\eqref{pas10} and establish the supersolution inequality.
		 
\vskip 0.2mm
	\noindent
	\textbf{Passing limit into \eqref{vdp3.1}: Supersolution inequality.} For the sake of convenience, we rewrite \eqref{vdp3.1}, along a subsequence as follows:
{\small	\begin{align}\label{vdp3.11}
	&\eps_n-\frac{1}{\eps_n}\E\int_{t_0}^{t_0+\eps_n}\ell(\Y_{\eps_n}(s),\a_{\eps_n}(s))\d s
	\nonumber\\&\geq
	-\frac{1}{\eps_n}\E\bigg[\int_{t_0}^{t_0+\eps_n}\bigg(\varphi_t(s,\Y_{\eps_n}(s))-\big((\mu\A+\alpha\I)\Y_{\eps_n}(s)+\mathcal{B}(\Y_{\eps_n}(s))+\beta\mathcal{C}(\Y_{\eps_n}(s)),\D\varphi(s,\Y_{\eps_n}(s))\big)\bigg)\d s\bigg]
	\nonumber\\&\quad-
	\E\bigg[\int_{t_0}^{t_0+\eps_n}\bigg(\big(\f(s,\a_{\eps_n}(s)),\D\varphi(s,\Y_{\eps_n}(s))\big)+\frac12\mathrm{Tr}(\Q\D^2\varphi(s,\Y_{\eps_n}(s)))\bigg)\d s\bigg]
	\nonumber\\&\quad-
	\frac{2m}{\eps_n}\E\bigg[\int_{t_0}^{t_0+\eps_n}\delta(s)
	\big(\f(s,\a_{\eps_n}(s)),(\A+\I)\Y_{\eps_n}(s)\big)\big)(1+\|\Y_{\eps_n}(s))\|_{\V}^2)^{m-1}\d s\bigg]
	\nonumber\\&\quad
	-\frac{1}{\eps_n}\E\bigg[\int_{t_0}^{t_0+\eps_n}\bigg(\delta'(s)(1+\|\Y_{\eps_n}(s))\|_{\V}^2)^m+m\delta(s)\big(\mathrm{Tr}(\Q_1)+\mathrm{Tr}(\Q)\big)(1+\|\Y_{\eps_n}(s))\|_{\V}^2)^{m-1}\bigg)\d s\bigg]
	\nonumber\\&\quad-
	\frac{2m(m-1)}{\eps_n}\E\bigg[\int_{t_0}^{t_0+\eps_n}\delta(s)
	\|\Q^{\frac12}(\A+\I)\Y_{\eps_n}(s)\|_{\H}^2
	(1+\|\Y_{\eps_n}(s)\|_{\V}^2)^{m-2}\d s\bigg]
	\nonumber\\&\quad+\frac{2m}{\eps_n}\E\bigg[\int_{t_0}^{t_0+\eps_n}\delta(s)\big((\mu\A+\alpha\I)\Y_{\eps_n}(s),(\A+\I)\Y_{\eps_n}(s)\big)
	(1+\|\Y_{\eps_n}(s))\|_{\V}^2)^{m-1}\d s\bigg]
	\nonumber\\&\quad+
	\frac{2m}{\eps_n}\E\bigg[\int_{t_0}^{t_0+\eps_n}\delta(s)(\mathcal{B}(\Y_{\eps_n}(s)),(\A+\I)\Y_{\eps_n}(s))
	(1+\|\Y_{\eps_n}(s))\|_{\V}^2)^{m-1}
	\d s\bigg]
	\nonumber\\&\quad+
	\frac{2m\beta}{\eps_n}\E\bigg[\int_{t_0}^{t_0+\eps_n}\delta(s)(\mathcal{C}(\Y_{\eps_n}(s)),(\A+\I)\Y_{\eps_n}(s))(1+\|\Y_{\eps_n}(s))\|_{\V}^2)^{m-1}\d s \bigg].
\end{align}}
	On utilizing all convergences \eqref{ss1}-\eqref{pas10}, rearranging the terms and using the definition of test function $\psi(\cdot,\cdot)=\varphi(\cdot,\cdot)+
	\delta(\cdot)(1+\|\cdot\|_{\V}^2)^m$, we finally arrive at (along a subsequence) the following inequality:

	\begin{align*}
		&-\psi_t(t_0,\y_0)+\big((\mu\A+\alpha\I)\y_0,\D\psi(t_0,\y_0)\big)+
		\big(\mathcal{B}(\y_0),\D\psi(t_0,\y_0)\big)+\big(\mathcal{C}(\y_0),\D\psi(t_0,\y_0)\big)
		\nonumber\\&\quad+\frac{1}{\eps_n}\E
		\bigg[\int_{t_0}^{t_0+\eps_n}\big[(\f(t_0,\a_{\eps_n}(s)),-\D\psi(t_0,\y_0))+\ell(\y_0,\a_{\eps_n}(s))\big]\d s\bigg] \nonumber\\&\quad-\frac12\mathrm{Tr}(\Q\D^2\psi(t_0,\y_0))
		\leq\upsigma(\eps_n).
	\end{align*}
    Finally, on taking infimum over $\a\in\U$ inside the integral and then passing $n\to\infty$, we get the supersolution inequality \eqref{eqn-superso}. With this, we have shown that the value function \eqref{valueF} is the viscosity solution of the HJB equation \eqref{HJBE}, and this completes the proof.
		\end{proof}

	\begin{appendix}\renewcommand{\thesection}{\Alph{section}}
		\numberwithin{equation}{section}
		\section{Some useful estimates}\label{useca}
		The aim of this appendix is to justify the convergences given in \eqref{viscdef4}, \eqref{pas9} and \eqref{pas10}. 
		\subsection{Explanation of \eqref{viscdef4} in Theorem \ref{comparison}}\label{expoca} 
 		Let us define 
		 \begin{align*}
		 \mathfrak{J}_n:=\D\varphi_n(t_n,\y_n)+ \frac{1}{\eps}\mathbf{Q}_N(\bar{\y}-\bar{\x})+ \frac{2}{\eps}\mathbf{Q}_N(\y_n-\bar{\y}).
		 \end{align*}
		 Then \eqref{viscdef3} can be written as 
		 \begin{align}\label{viscdef3.1}
			&-\frac{\delta}{2}K_\upgamma e^{K_\upgamma(T-t_n)} (1+\|\y_n\|_{\V}^2)^m+
			\underbrace{(\varphi_n)_t(t_n,\y_n)+\frac12\mathrm{Tr}(\Q\D^2\varphi_n(t_n,\y_n)+2\Q\mathbf{Q}_N)}_{\text{ Term I}} 
			\nonumber\\&\quad-
			\underbrace{\min\{\mu,\alpha\}\big((\A+\I)\y_n,\mathfrak{J}_n\big)}_{\text{Term II}}- \underbrace{\big(\mathcal{B}(\y_n), \mathfrak{J}_n\big)}_{\text{Term III}}-
			\underbrace{\beta\big(\mathcal{C}(\y_n), \mathfrak{J}_n\big)}_{\text{Term IV}}+
			\underbrace{F\big(t_n,\y_n, \mathfrak{J}_n\big)}_{\text{Term V}}
			\nonumber\\&\geq\frac{\upgamma}{2T^2}+
			\underbrace{\beta m\delta e^{K_\upgamma(T-t_n)} (1+\|\y_n\|_{\V}^2)^{m-1}\||\y_n|^{\frac{r-1}{2}}\nabla\y_n\|_{\H}^2}_{\text{Term VI}}
			\nonumber\\&\quad+\underbrace{\frac{\min\{\mu,\alpha\} m\delta}{2} e^{K_\upgamma(T-t_n)} (1+\|\y_n\|_{\V}^2)^{m-1}\|(\A+\I)\y_n\|_{\H}^2}_{\text{ Term VII}}
			\nonumber\\&\quad+
			\underbrace{2\beta m\delta e^{K_\upgamma(T-t_n)}(1+\|\y_n\|_{\V}^2)^{m-1} \|\y_n\|_{\wi\L^{r+1}}^{r+1}}_{\text{Term VIII}}.
		\end{align}

		Since convergent and weakly convergent sequences are bounded, therefore by utilizing \eqref{fnPQ1}, \eqref{conv1}-\eqref{conv2}, and from Step-II of Theorem \ref{comparison}, we conclude the following bound:
		\begin{align}\label{ot1}
		 \|\mathfrak{J}_n\|_{\H}&\leq\bigg\|\D\varphi_n(t_n,\y_n)+\frac{1}{\eps}\mathbf{Q}_N(\bar{\y}-\bar{\x})+\frac{2}{\eps}\mathbf{Q}_N(\y_n-\bar{\y})\bigg\|_{\H}
		\nonumber\\&=
		\bigg\|\D\varphi_n(t_n,\y_n)-\frac{1}{\eps}\mathbf{P}_N(\bar{\y}-\bar{\x})+\frac{1}{\eps}(\bar{\y}-\bar{\x})+\frac{2}{\eps}\mathbf{Q}_N(\y_n-\bar{\y})\bigg\|_{\H}\nonumber\\&\leq 
		\bigg\|\D\varphi_n(t_n,\y_n)-\frac{1}{\eps}\mathbf{P}_N(\bar{\y}-\bar{\x})\bigg\|_{\H}+\bigg\|\frac{1}{\eps}(\bar{\y}-\bar{\x})+\frac{2}{\eps}\mathbf{Q}_N(\y_n-\bar{\y})\bigg\|_{\H}
		\nonumber\\&\leq C(\eps,\|\bar{\y}\|_{\V},\|\bar{\x}\|_{\V}).
		\end{align} 
		By using the Cauchy-Schwarz and Young's inequalities, we estimate the following (see calculations \eqref{vdp5} and \eqref{vdp55}):
		\begin{align}
			|(\mathcal{C}(\y_n),\mathfrak{J}_n)|&\leq C\left(\||\y_n|^{\frac{r-1}{2}}\nabla\y_n\|_{\H}^{\frac{2r}{r+1}}+\|\y_n\|_{\wi\L^{r+1}}^r\right)
			\|\mathfrak{J}_n\|_{\H}\nonumber\\&\leq
			\frac{ m\delta}{4}\||\y_n|^{\frac{r-1}{2}}\nabla\y_n\|_{\H}^2+
			C\|\mathfrak{J}_n\|_{\H}^{r+1}+m\delta\|\y_n\|_{\wi\L^{r+1}}^{r+1},
			\label{appdac1}\\
			|(\mathcal{B}(\y_n),\mathfrak{J}_n)|&\leq \frac12\|\mathfrak{J}_n\|_{\H}^2
			+\frac{\beta m\delta}{4}\||\y_n|^{\frac{r-1}{2}}\nabla\y_n\|_{\H}^2
			+\varrho^{**}\|\nabla\y_n\|_{\H}^2,\label{appdac2}\\
			\text{ and } |((\A+\I)\y_n,\mathfrak{J}_n)|&\leq \frac{m\delta}{4}\|(\A+\I)\y_n\|_{\H}^2+C_2\|\mathfrak{J}_n\|_{\H}^2\label{appdac3},
		\end{align}\
 where $\varrho^{**}=\frac{r-3}{m\delta(r-1)}\left[\frac{4}{\beta m\delta (r-1)}\right]^{\frac{2}{r-3}}$. On combining \eqref{ot1},\eqref{appdac1}-\eqref{appdac3} in \eqref{viscdef3.1}, we deduce 
 \begin{align}\label{appdac4}
 		&- \frac{\delta}{2}K_\upgamma e^{K_\upgamma(T-t_n)} (1+\|\y_n\|_{\V}^2)^m+
 	(\varphi_n)_t(t_n,\y_n)+\frac12\mathrm{Tr}(\Q\D^2\varphi_n(t_n,\y_n))+\mathrm{Tr}(\Q\mathbf{Q}_N)\nonumber\\&
 	\quad+\varrho^{**}\|\y_n\|_{\V}^2+
 	F\big(t_n,\y_n, \mathfrak{J}_n\big)
 	+C\|\mathfrak{J}_n\|_{\H}^2+
 	C\|\mathfrak{J}_n\|_{\H}^{r+1}
 	\nonumber\\&\geq\frac{\upgamma}{2T^2}+\frac{\beta m\delta}{2} (1+\|\y_n\|_{\V}^2)^{m-1}\||\y_n|^{\frac{r-1}{2}}\nabla\y_n\|_{\H}^2
 	\nonumber\\&\quad+\frac{\min\{\mu,\alpha\} m\delta}{4} (1+\|\y_n\|_{\V}^2)^{m-1}\|(\A+\I)\y_n\|_{\H}^2\nonumber\\&\quad+\beta m\delta e^{K_\upgamma(T-t_n)}(1+\|\y_n\|_{\V}^2)^{m-1} \|\y_n\|_{\wi\L^{r+1}}^{r+1}.
 \end{align}
  We now estimate the bound for $F\big(t_n,\y_n, \mathfrak{J}_n\big)$. Let us assume that $\|\y_n\|_{\V}\leq r$, $\|\bar{\y}\|_{\V}\leq r$ and $\|\bar{\x}\|_{\V}\leq r$, where $r>0$ is arbitrary. Then, from assumptions \eqref{F1}-\eqref{F3} of Hypothesis \ref{hypF14}, \eqref{conv1}-\eqref{conv2} and \eqref{eqn-v-norm-conv}, we calculate
 \begin{align}\label{F5}
& \bigg|F\bigg(t_n,\y_n,\D\varphi_n(t_n,\y_n)+\frac{1}{\eps}\mathbf{Q}_N(\bar{\y}-\bar{\x})+\frac{2}{\eps}\mathbf{Q}_N(\y_n-\bar{\y})\bigg)\bigg|
 \nonumber\\&\leq
 \bigg|F\bigg(t_n,\y_n,\D\varphi_n(t_n,\y_n)+\frac{1}{\eps}\mathbf{Q}_N(\bar{\y}-\bar{\x})+\frac{2}{\eps}\mathbf{Q}_N(\y_n-\bar{\y})\bigg)
 \nonumber\\&\quad-
 \underbrace{F\bigg(t_n,\y_n,\D\varphi_n(t_n,\y_n)+\frac{1}{\eps}\mathbf{Q}_N(\bar{\y}-\bar{\x})\bigg)\bigg|}_{\text{ using } \ \eqref{F2}}
 \nonumber\\&\quad+\underbrace{\bigg|
 	F\bigg(t_n,\y_n,\D\varphi_n(t_n,\y_n)+\frac{1}{\eps}\mathbf{Q}_N(\bar{\y}-\bar{\x})\bigg)- F\big(t_n,\y_n,\D\varphi_n(t_n,\y_n)\big)\bigg|}_{\text{using} \ \eqref{F2}}
 \nonumber\\&\quad+
 \underbrace{ \bigg|F\big(t_n,\y_n,\D\varphi_n(t_n,\y_n)\big)-F\bigg(t_n,\y_n,\frac{1}{\eps}(\bar{\y}-\bar{\x})\bigg)\bigg|}_{\text{using} \ \eqref{F2}}
  \nonumber\\&\quad+
  \underbrace{\bigg|F\bigg(t_n,\y_n,\frac{1}{\eps}(\bar{\y}-\bar{\x})\bigg)-
  	F\bigg(t_n,\bar{\y},\frac{1}{\eps}(\bar{\y}-\bar{\x})\bigg)\bigg|}_{\text{ using} \ \eqref{F1}}
   \nonumber\\&\quad+
  \underbrace{\bigg|F\bigg(t_n,\bar{\y},\frac{1}{\eps}(\bar{\y}-\bar{\x})\bigg)-F\bigg(\bar{t},\bar{\y},\frac{1}{\eps}(\bar{\y}-\bar{\x})\bigg)\bigg|}_{\text{ using } \ \eqref{F3}}
   \nonumber\\&\quad+
   \bigg|F\bigg(\bar{t},\bar{\y},\frac{1}{\eps}(\bar{\y}-\bar{\x})\bigg)\bigg|
   \nonumber\\&\leq
   \omega\left(\frac{2}{\eps}(1+\|\y_n\|_{\V})\|\mathbf{Q}_N(\y_n-\bar{\y})\|_{\H}\right)+ \omega\left(\frac{1}{\eps}(1+\|\y_n\|_{\V})\|\mathbf{Q}_N(\bar{\y}-\bar{\x})\|_{\H}\right)
   \nonumber\\&\quad+
   \omega\left(\frac{1}{\eps}(1+\|\y_n\|_{\V})\left\|\D\varphi_n(t_n,\y_n)
   -\frac{1}{\eps}(\bar{\y}-\bar{\x})\right\|_{\H}\right)
   \nonumber\\&\quad+
   \omega_r(\|\y_n-\bar{\y}\|_{\V})+\omega\left(\frac{1}{\eps}\|\y_n-\bar{\y}\|_{\V}\|\bar{\y}-\bar{\x}\|_{\H}\right)+\omega_r(|t_n-\bar{t}|)
    \nonumber\\&\quad+
   \bigg|F\bigg(\bar{t},\bar{\y},\frac{1}{\eps}(\bar{\y}-\bar{\x})\bigg)\bigg|.
 \end{align}
Substituting \eqref{F5} into \eqref{appdac4}, we deduce the following estimate:
 \begin{align}\label{appdac5}
	&\frac{\upgamma}{2T^2}+\frac{\beta m\delta}{2} (1+\|\y_n\|_{\V}^2)^{m-1}\||\y_n|^{\frac{r-1}{2}}\nabla\y_n\|_{\H}^2
	+\beta m\delta e^{K_\upgamma(T-t_n)}(1+\|\y_n\|_{\V}^2)^{m-1} \|\y_n\|_{\wi\L^{r+1}}^{r+1}
	\nonumber\\&\quad+\frac{\min\{\mu,\alpha\} m\delta}{4} (1+\|\y_n\|_{\V}^2)^{m-1}\|(\A+\I)\y_n\|_{\H}^2
	\nonumber\\&\leq
	- \frac{\delta}{2}K_\upgamma e^{K_\upgamma(T-t_n)} (1+\|\y_n\|_{\V}^2)^m+
		(\varphi_n)_t(t_n,\y_n)+\frac12\mathrm{Tr}(\Q\D^2\varphi_n(t_n,\y_n))+\mathrm{Tr}(\Q\mathbf{Q}_N)\nonumber\\&
	\quad+\varrho^{**}\|\y_n\|_{\V}^2+C\|\mathfrak{J}_n\|_{\H}^2+
	C\|\mathfrak{J}_n\|_{\H}^{r+1}
	\nonumber\\&\quad+
	 \omega\left(\frac{2}{\eps}(1+\|\y_n\|_{\V})\|\mathbf{Q}_N(\y_n-\bar{\y})\|_{\H}\right)+ \omega\left(\frac{1}{\eps}(1+\|\y_n\|_{\V})\|\mathbf{Q}_N(\bar{\y}-\bar{\x})\|_{\H}\right)
	\nonumber\\&\quad+
	\omega\left(\frac{1}{\eps}(1+\|\y_n\|_{\V})\left\|\D\varphi_n(t_n,\y_n)
	-\frac{1}{\eps}(\bar{\y}-\bar{\x})\right\|_{\H}\right)
	\nonumber\\&\quad+
	\omega_r(\|\y_n-\bar{\y}\|_{\V})+\omega\left(\frac{1}{\eps}\|\y_n-\bar{\y}\|_{\V}\|\bar{\y}-\bar{\x}\|_{\H}\right)+\omega_r(|t_n-\bar{t}|)
	\nonumber\\&\quad+
	\bigg|F\bigg(\bar{t},\bar{\y},\frac{1}{\eps}(\bar{\y}-\bar{\x})\bigg)\bigg|.
\end{align}
 Finally, using the fact that convergent sequences are bounded (see \eqref{conv1}-\eqref{conv2}), we conclude that the left hand side of \eqref{appdac5} is bounded independent of $n$.
 Therefore, by an application of the Banach-Alaoglu theorem and then using the uniqueness of weak limits, we have the following weak convergence as $n\to\infty$:
 \begin{align}\label{weakconv1}
 	\A\y_n\rightharpoonup\A\bar{\y} \  \text{ in } \ \D(\A).
 \end{align}
 Also, from \eqref{conv1} and \eqref{eqn-v-norm-conv}, we conclude the following convergence in $\V$:
 \begin{align}\label{conv3}
 	 \mathfrak{J}_n&:=\D\varphi_n(t_n,\y_n)+ \frac{1}{\eps}\mathbf{Q}_N(\bar{\y}-\bar{\x})+ \frac{2}{\eps}\mathbf{Q}_N(\y_n-\bar{\y})\nonumber\\&\to
 	 \frac{1}{\eps}(\bar{\y}_N-\bar{\x}_N) +\frac{1}{\eps}\mathbf{Q}_N(\bar{\y}-\bar{\x})=\frac{1}{\eps}(\bar{\y}-\bar{\x}) \ \text{ as } \ n\to\infty.
 \end{align}
  Now we justify \eqref{viscdef3.1} as `$n\to\infty$'. For this, we separately discuss the limit of each of the terms in \eqref{viscdef3.1}.
 \vskip 0.2mm
 \noindent
 \textbf{Passing limit in Term II of \eqref{viscdef3.1}.} From \eqref{weakconv1}, \eqref{conv3} and the fact that $\|\A\y_n\|_{\H}$ is bounded independent of $n$, we estimate
 \begin{align}\label{Aconv}
 	&\bigg|(\A\y_n,\mathfrak{J}_n)-\bigg(\A\bar{\y}, \frac{\bar{\y}-\bar{\x}}{\eps}\bigg)\bigg|\nonumber\\&\leq
 	\bigg|\bigg(\A\y_n,\mathfrak{J}_n-\frac{\bar{\y}-\bar{\x}}{\eps}\bigg)\bigg|+\bigg|\bigg(\A\y_n-\A\bar{\y}, \frac{\bar{\y}-\bar{\x}}{\eps}\bigg)\bigg|
 	\nonumber\\&\leq
     \|\A\y_n\|_{\H}\left\|\mathfrak{J}_n-\frac{\bar{\y}-\bar{\x}}{\eps}\right\|_{\V}+
 	\bigg|\bigg(\A\y_n-\A\bar{\y}, \frac{\bar{\y}-\bar{\x}}{\eps}\bigg)\bigg|
 	\nonumber\\&\to0, \  \text{ as } \ n\to\infty.
 \end{align}
		
 \vskip 0.2mm
\noindent
\textbf{Passing limit in Term III of \eqref{viscdef3.1}.} 
By using H\"older's and Agmon's inequalities, \eqref{eqn-v-norm-conv} and, \eqref{conv3}, we estimate
\begin{align}\label{Bconv2}
	\bigg|\bigg(\mathcal{B}(\y_n),\mathfrak{J}_n-\frac{\bar{\y}-\bar{\x}}{\eps}\bigg)\bigg|&\leq C\|\nabla\y_n\|_{\H}\|\y_n\|_{\H}^{1-\frac{d}{4}}\|\y_n\|_{\H_{\mathrm{p}}^2}^{\frac{d}{4}}\left\|\mathfrak{J}_n-\frac{\bar{\y}-\bar{\x}}{\eps}\right\|_{\V}
	\nonumber\\&\to0, \  \text{ as } \ n\to\infty.
\end{align} 
 where we have used the fact that $\|\A\y_n\|_{\H}$ is bounded independent of $n$. Further, by using H\"older's inequality, Sobolev embedding and \eqref{eqn-v-norm-conv}, we compute
 \begin{align}\label{Bconv3}
 	&\bigg|\bigg(\mathcal{B}(\y_n)-\mathcal{B}(\bar{\y}), \frac{\bar{\y}-\bar{\x}}{\eps}\bigg)\bigg|
 	\nonumber\\&\leq\bigg|\bigg(\mathcal{B}(\bar{\y},\y_n-\bar{\y}), \frac{\bar{\y}-\bar{\x}}{\eps}\bigg)\bigg|+\bigg|\bigg(\mathcal{B}(\y_n-\bar{\y},\y_n), \frac{\bar{\y}-\bar{\x}}{\eps}\bigg)\bigg|
 	\nonumber\\&\leq
 	\frac{1}{\eps}\|\bar{\y}\|_{\wi\L^4}\|\nabla(\y_n-\bar{\y})\|_{\H}\|\bar{\y}-\bar{\x}\|_{\wi\L^4}+
 	\frac{1}{\eps}\|\y_n-\bar{\y}\|_{\wi\L^4}\|\nabla\bar{\y}\|_{\H}\|\bar{\y}-\bar{\x}\|_{\wi\L^4}
 	\nonumber\\&\to0, \  \text{ as } \ n\to\infty.
 \end{align}
 Combining \eqref{Bconv2} and \eqref{Bconv3} yields the following convergence:
 \begin{align}\label{Bconv1}
 	&\bigg|(\mathcal{B}(\y_n),\mathfrak{J}_n)-\bigg(\mathcal{B}(\bar{\y}), \frac{\bar{\y}-\bar{\x}}{\eps}\bigg)\bigg|\nonumber\\&\leq
 	\bigg|\bigg(\mathcal{B}(\y_n),\mathfrak{J}_n-\frac{\bar{\y}-\bar{\x}}{\eps}\bigg)\bigg|+\bigg|\bigg(\mathcal{B}(\y_n)-\mathcal{B}(\bar{\y}), \frac{\bar{\y}-\bar{\x}}{\eps}\bigg)\bigg|
 	\nonumber\\&\to0, \  \text{ as } \ n\to\infty.
 \end{align}
 
  \vskip 0.2mm
 \noindent
 \textbf{Passing limit in Term IV of \eqref{viscdef3.1}.} 
 Since, we have $\y_n\to\bar{\y}$ as $n\to\infty$ in $\wi\L^{r+1}$, therefore from \cite[Theorem 4.9, pp. 94]{HB}, we conclude the following a.e. convergence (along a subsequence) as $n\to\infty$:
 \begin{align*}
 	\y_n(\xi)\to\bar{\y}(\xi) \  \text{ for a.e. } \ \xi\in\Omega,
 \end{align*}
 and therefore $|\y_n(\xi)|^{r-1}\y_n(\xi)\to|\bar{\y}(\xi)|^{r-1}\bar{\y}(\xi)$ as $n\to\infty$, for a.e. $\xi\in\Omega$. 
 Further,  for the values of $r$  given in Table \ref{Table1}, by the application of Taylor's formula \cite[Theorem 7.9.1]{PGC}, H\"older's and interpolation inequalities, and Sobolev embedding $\V\hookrightarrow\wi\L^6$, we calculate the following:
 \begin{align}\label{Cconv4}
  \bigg|\bigg(\mathcal{C}(\y_n)-\mathcal{C}(\bar{\y}), \frac{\bar{\y}-\bar{\x}}{\eps}\bigg)\bigg|&\leq\frac{C}{\eps}
  \||\y_n|^{r-1}+|\bar{\y}|^{r-1}\|_{\wi\L^3}\|\y_n-\bar{\y}\|_{\wi\L^6}
  \|\bar{\y}-\bar{\x}\|_{\H}
  \nonumber\\&\leq\frac{C}{\eps}
  \big(\|\y_n\|_{\wi\L^{3(r-1)}}^{r-1}+\|\bar{\y}\|_{\wi\L^{3(r-1)}}^{r-1}\big)\|\y_n-\bar{\y}\|_{\wi\L^6}\|\bar{\y}-\bar{\x}\|_{\H}
  \nonumber\\&\leq \frac{C}{\eps}
  \big(\|\y_n\|_{\wi\L^{r+1}}\|\y_n\|_{\wi\L^{3(r+1)}}^{r-2}+
  \|\bar{\y}\|_{\wi\L^{r+1}}\|\bar{\y}\|_{\wi\L^{3(r+1)}}^{r-2}\big)
  \nonumber\\&\quad\times\|\y_n-\bar{\y}\|_{\V}\|\bar{\y}-\bar{\x}\|_{\H}
   \nonumber\\&\to0, \  \text{ as } \ n\to\infty.
 \end{align}
 For the values of $r$ given in Table \ref{Table1}, we conclude by using interpolation inequality and \eqref{appdac5} that
	\begin{align}\label{Cconv2}
		\|\mathcal{C}(\y_n)\|_{\H}\leq\|\y_n\|_{\widetilde{\L}^{r+1}}^{\frac{r+3}{4}}\|\y_n\|_{\widetilde{\L}^{3(r+1)}}^{\frac{3(r-1)}{4}}\leq C,
	\end{align}
	where the constant $C$ is independent of $n$. 
Therefore, from \eqref{Cconv2} and \eqref{conv3}, we calculate
		\begin{align}\label{Cconv3}
			\bigg|\bigg(\mathcal{C}(\y_n),\mathfrak{J}_n-\frac{\bar{\y}-\bar{\x}}{\eps}\bigg)\bigg|&\leq\|\mathcal{C}(\y_n)\|_{\H}\left\|\mathfrak{J}_n-\frac{\bar{\y}-\bar{\x}}{\eps}\right\|_{\H}
			\nonumber\\&\leq\|\y_n\|_{\widetilde{\L}^{r+1}}^{\frac{r+3}{4}}\|\y_n\|_{\widetilde{\L}^{3(r+1)}}^{\frac{3(r-1)}{4}}\left\|\mathfrak{J}_n-\frac{\bar{\y}-\bar{\x}}{\eps}\right\|_{\V}
			\nonumber\\&\to0, \  \text{ as } \ n\to\infty. 
		\end{align}
   Therefore, utilizing \eqref{Cconv4} and \eqref{Cconv3}, we finally conclude
   \begin{align}\label{Cconv1}
   	&\bigg|(\mathcal{C}(\y_n),\mathfrak{J}_n)-\bigg(\mathcal{C}(\bar{\y}), \frac{\bar{\y}-\bar{\x}}{\eps}\bigg)\bigg|\nonumber\\&\leq
   	\bigg|\bigg(\mathcal{C}(\y_n),\mathfrak{J}_n-\frac{\bar{\y}-\bar{\x}}{\eps}\bigg)\bigg|+\bigg|\bigg(\mathcal{C}(\y_n)-\mathcal{C}(\bar{\y}), \frac{\bar{\y}-\bar{\x}}{\eps}\bigg)\bigg|
   	\nonumber\\&\to0 \  \text{ as } \ n\to\infty.
   \end{align}
We cannot justify the terms VI and VII of \eqref{viscdef3.1} as $n\to\infty$, since it requires strong convergence. Therefore, we cannot pass the limit as $n\to\infty$ in \eqref{viscdef3.1}. So, on taking $\limsup$ as $n\to\infty$ in \eqref{viscdef3.1}, together with \eqref{conv3}, \eqref{Aconv}, \eqref{Bconv1} and \eqref{Cconv1}, we finally deduce \eqref{viscdef4}.

	\subsection{Explanation of \eqref{pas9} in Theorem \ref{extunqvisc}}\label{wknBA1}
	Let us now estimate the limit of the term \eqref{pas9} as $n\to\infty$. We write 
		\begin{align}\label{wknBA}
			&\bigg|\frac{1}{\eps_n}\E\int_{t_0}^{t_0+\eps_n}\delta(s)
			\big(\mathcal{B}(\Y_{\eps_n}(s)),\A\Y_{\eps_n}(s)\big)(1+\|\Y_{\eps_n}(s))\|_{\V}^2)^{m-1} 
			\d s\nonumber\\&\quad-
			\delta(t_0)\big(\mathcal{B}(\y_0),\A\y_0\big)(1+\|\y_0\|_{\V}^2)^{m-1}\bigg|\nonumber\\&\leq 
			\frac{1}{\eps_n}\E\int_{t_0}^{t_0+\eps_n}\delta(s)\bigg|\big(\mathcal{B}(\Y_{\eps_n}(s)),\A\Y_{\eps_n}(s)\big)-\big(\mathcal{B}(\y_0),\A\y_0\big)\bigg|(1+\|\Y_{\eps_n}(s))\|_{\V}^2)^{m-1}\d s
			\nonumber\\&\quad+
			\frac{1}{\eps_n}\E\int_{t_0}^{t_0+\eps_n}\delta(s)\big|\big(\mathcal{B}(\y_0),\A\y_0\big)\big|\bigg|(1+\|\Y_{\eps_n}(s))\|_{\V}^2)^{m-1}-(1+\|\y_0\|_{\V}^2)^{m-1}\bigg|\d s
			\nonumber\\&\quad+
			\frac{1}{\eps_n}\E\int_{t_0}^{t_0+\eps_n}\big|\delta(s)-\delta(t_0)\big|
			\big|\big(\mathcal{B}(\y_0),\A\y_0\big)\big|(1+\|\y_0\|_{\V}^2)^{m-1}\d s
			\nonumber\\&\leq
			\underbrace{\frac{1}{\eps_n}\E\int_{t_0}^{t_0+\eps_n}\delta(s)\bigg|\big(\mathcal{B}(\Y_{\eps_n}(s))-\mathcal{B}(\y_0),\A\Y_{\eps_n}(s)\big)\bigg|(1+\|\Y_{\eps_n}(s))\|_{\V}^2)^{m-1}\d s}_{\text{$J_4$}}
			\nonumber\\&\quad+
			\underbrace{	 \frac{1}{\eps_n}\E\int_{t_0}^{t_0+\eps_n}\delta(s)\bigg|\big(\mathcal{B}(\y_0),\A\Y_{\eps_n}(s)-\A\y_0\big)\bigg|(1+\|\Y_{\eps_n}(s))\|_{\V}^2)^{m-1}\d s}_{\text{$J_5$}}
			\nonumber\\&\quad+
			\underbrace{\frac{1}{\eps_n}\E\int_{t_0}^{t_0+\eps_n}\delta(s)\big|\big(\mathcal{B}(\y_0),\A\y_0\big)\big|\bigg|(1+\|\Y_{\eps_n}(s))\|_{\V}^2)^{m-1}-(1+\|\y_0\|_{\V}^2)^{m-1}\bigg|\d s}_{\text{$J_6$}}
			\nonumber\\&\quad+
			\underbrace{ \frac{1}{\eps_n}\E\int_{t_0}^{t_0+\eps_n}|\delta(s)-\delta(t_0)|
				\big|\big(\mathcal{B}(\y_0),\A\y_0\big)\big|(1+\|\y_0\|_{\V}^2)^{m-1}\d s}_{\text{$J_7$}}.
		\end{align}
		By using Taylor's formula and H\"older's inequality, we estimate $J_4$ as
		\begin{align*}
			J_4&\leq\frac{1}{\eps_n}\E\int_{t_0}^{t_0+\eps_n}\delta(s)\bigg|\big(\mathcal{B}(\Y_{\eps_n}(s))-\mathcal{B}(\y_0),\A\Y_{\eps_n}(s)\big)\bigg|(1+\|\Y_{\eps_n}(s))\|_{\V}^2)^{m-1}\d s
			\nonumber\\&\leq
			\frac{1}{\eps_n}\E\int_{t_0}^{t_0+\eps_n}\delta(s)\bigg|\big(\mathcal{B}(\Y_{\eps_n}(s)-\y_0,\Y_{\eps_n}(s)),\A\Y_{\eps_n}(s)\big)\bigg|(1+\|\Y_{\eps_n}(s))\|_{\V}^2)^{m-1}\d s
			\nonumber\\&\quad+
			\frac{1}{\eps_n}\E\int_{t_0}^{t_0+\eps_n}\delta(s)\bigg|\big(\mathcal{B}(\y_0,\Y_{\eps_n}(s)-\y_0),\A\Y_{\eps_n}(s)\big)
			\bigg|(1+\|\Y_{\eps_n}(s))\|_{\V}^2)^{m-1}\d s
			\nonumber\\&\leq
			\frac{1}{\eps_n}\E\int_{t_0}^{t_0+\eps_n}\delta(s)\|\mathcal{B}(\Y_{\eps_n}(s)-\y_0,\Y_{\eps_n}(s))\|_{\H}\|\A\Y_{\eps_n}(s)\|_{\H}(1+\|\Y_{\eps_n}(s))\|_{\V}^2)^{m-1}\d s
			\nonumber\\&\quad+
			\frac{1}{\eps_n}\E\int_{t_0}^{t_0+\eps_n}\delta(s)\|\mathcal{B}(\y_0,\Y_{\eps_n}(s)-\y_0)\|_{\H}\|\A\Y_{\eps_n}(s)\|_{\H}(1+\|\Y_{\eps_n}(s))\|_{\V}^2)^{m-1}\d s
			\nonumber\\&\leq
			\frac{1}{\eps_n}\E\int_{t_0}^{t_0+\eps_n}\delta(s)
			\|\Y_{\eps_n}(s)-\y_0\|_{\H}^{1-\frac{d}{4}}\|(\A+\I)(\Y_{\eps_n}(s)-\y_0)\|_{\H}^{\frac{d}{4}}\|\nabla\Y_{\eps_n}(s)\|_{\H}
			\|\A\Y_{\eps_n}(s)\|_{\H}\nonumber\\&\qquad\qquad\times(1+\|\Y_{\eps_n}(s))\|_{\V}^2)^{m-1}\d s
			\nonumber\\&\quad+
			\frac{1}{\eps_n}\E\int_{t_0}^{t_0+\eps_n}\delta(s)\|\y_0\|_{\H}^{1-\frac{d}{4}}\|(\A+\I)\y_0\|_{\H}^{\frac{d}{4}}
			\|\nabla(\Y_{\eps_n}(s)-\y_0)\|_{\H}\|\A\Y_{\eps_n}(s)\|_{\H}\nonumber\\&\qquad\qquad\times(1+\|\Y_{\eps_n}(s))\|_{\V}^2)^{m-1}\d s
			\nonumber\\&\leq
			\frac{C}{\eps_n}\E\int_{t_0}^{t_0+\eps_n}
			\|\Y_{\eps_n}(s)-\y_0\|_{\H}^{1-\frac{d}{4}}\left(\|(\A+\I)\Y_{\eps_n}(s)\|_{\H}^{\frac{d}{4}+1}+\|(\A+\I)\y_0\|_{\H}^{\frac{d}{4}}\right)
			\nonumber\\&\qquad\qquad\times(1+\|\Y_{\eps_n}(s))\|_{\V}^2)^{m}\d s
			\nonumber\\&\quad+
			\frac{C}{\eps_n}\E\int_{t_0}^{t_0+\eps_n}\|\y_0\|_{\H}^{1-\frac{d}{4}}\|(\A+\I)\y_0\|_{\H}^{\frac{d}{4}}
			\|\nabla(\Y_{\eps_n}(s)-\y_0)\|_{\H}\|\A\Y_{\eps_n}(s)\|_{\H}\nonumber\\&\qquad\qquad\times(1+\|\Y_{\eps_n}(s))\|_{\V}^2)^{m-1}\d s
			\nonumber\\&\leq
			C\left(\frac{1}{\eps_n}\int_{t_0}^{t_0+\eps_n}
			\E\|\Y_{\eps_n}(s)-\y_0\|_{\H}^4 \d s\right)^{\frac{4-d}{16}}  \left(\frac{1}{\eps_n}\int_{t_0}^{t_0+\eps_n}\E\|(\A+\I)\Y_{\eps_n}(s)\|_{\H}^{2}\d s
			\right)^{\frac{4+d}{8}}\nonumber\\&\qquad\times
			\left(\frac{1}{\eps_n}\int_{t_0}^{t_0+\eps_n}\E
			(1+\|\Y_{\eps_n}(s))\|_{\V}^2)^{\frac{16m}{4-d}}\d s\right)^{\frac{4-d}{16}} \nonumber\\&\quad+
			C\left(\frac{1}{\eps_n}\int_{t_0}^{t_0+\eps_n}
			\E\|\Y_{\eps_n}(s)-\y_0\|_{\H}^4 \d s\right)^{\frac{4-d}{16}}  \left(\frac{1}{\eps_n}\int_{t_0}^{t_0+\eps_n}\E
			(1+\|\Y_{\eps_n}(s))\|_{\V}^2)^{\frac{8m}{4+d}}\d s\right)^{\frac{4+d}{8}}
			\nonumber\\&\quad+
			C\left(\frac{1}{\eps_n}\int_{t_0}^{t_0+\eps_n}
			\E\|\Y_{\eps_n}(s)-\y_0\|_{\V}^4 \d s\right)^{\frac{1}{4}}  \left(\frac{1}{\eps_n}\int_{t_0}^{t_0+\eps_n}\E\|\A\Y_{\eps_n}(s)\|_{\H}^{2}\d s\right)^{\frac{1}{2}}\nonumber\\&\qquad\times
			\left(\frac{1}{\eps_n}\int_{t_0}^{t_0+\eps_n}\E
			(1+\|\Y_{\eps_n}(s))\|_{\V}^2)^{4m}\d s\right)^{\frac{1}{4}}
			\nonumber\\&\leq\upsigma(\eps_n),
		\end{align*}
		where the last inequality follows from \eqref{ssee1}, \eqref{ctsdep0.1}-\eqref{ctsdep0.2} and \eqref{vdp7}. By using H\"older's inequality, weak convergence \eqref{wknm2} and \eqref{ssee1}, we estimate $J_5$ as 
		\begin{align*}
			J_5&\leq\frac{C}{\eps_n}\E\int_{t_0}^{t_0+\eps_n}|(\mathcal{B}(\y_0),\A\Y_{\eps_n}(s)-\A\y_0)|(1+\|\Y_{\eps_n}(s))\|_{\V}^2)^{m-1}\d s
			\nonumber\\&\leq
			C\left(\frac{1}{\eps_n}\int_{t_0}^{t_0+\eps_n}
			\E|(\mathcal{B}(\y_0),\A\Y_{\eps_n}(s)-\A\y_0)|^2\d s\right)^{\frac12}
			\left(\frac{1}{\eps_n}\int_{t_0}^{t_0+\eps_n}
		\E	(1+\|\Y_{\eps_n}(s))\|_{\V}^2)^{2m}\d s\right)^{\frac{1}{2}}
			\nonumber\\&\leq\upsigma(\eps_n).
		\end{align*}
		The convergence of the terms $J_6$ and $J_7$ in \eqref{wknBA} as $n\to\infty$ can be proved similarly and it follows from \eqref{ctsdep0.2} and the property of the modulus of continuity. With this, the right hand side of \eqref{wknBA} approaches to $0$ as $n\to\infty$.
		
		\subsection{Explanation of \eqref{pas10} in Theorem \ref{extunqvisc}}\label{wknCA1}
	      Let us now estimate the limit of the term \eqref{pas10} as $n\to\infty$. We write 
			\begin{align}\label{wknCA}
				&\bigg|\frac{1}{\eps_n}\E\int_{t_0}^{t_0+\eps_n}\delta(s)\big(\mathcal{C}(\Y_{\eps_n}(s)),(\A+\I)\Y_{\eps_n}(s)\big)(1+\|\Y_{\eps_n}(s))\|_{\V}^2)^{m-1} 
				\d s\nonumber\\&\quad-
				\delta(t_0)\big(\mathcal{C}(\y_0),(\A+\I)\y_0\big)(1+\|\y_0\|_{\V}^2)^{m-1}\bigg|\nonumber\\&\leq 
				\frac{1}{\eps_n}\E\int_{t_0}^{t_0+\eps_n}\delta(s)\bigg|\big(\mathcal{C}(\Y_{\eps_n}(s)),(\A+\I)\Y_{\eps_n}(s)\big)-\big(\mathcal{C}(\y_0),(\A+\I)\y_0\big)\bigg|(1+\|\Y_{\eps_n}(s))\|_{\V}^2)^{m-1}\d s
				\nonumber\\&\quad+
				\frac{1}{\eps_n}\E\int_{t_0}^{t_0+\eps_n}\delta(s)\big|\big(\mathcal{C}(\y_0),(\A+\I)\y_0\big)\big|\bigg|(1+\|\Y_{\eps_n}(s))\|_{\V}^2)^{m-1}-(1+\|\y_0\|_{\V}^2)^{m-1}\bigg|\d s
				\nonumber\\&\quad+
				\frac{1}{\eps_n}\E\int_{t_0}^{t_0+\eps_n}\big|\delta(s)-\delta(t_0)\big|
				\big|\big(\mathcal{C}(\y_0),(\A+\I)\y_0\big)\big|(1+\|\y_0\|_{\V}^2)^{m-1}\d s
				\nonumber\\&\leq
				\underbrace{\frac{1}{\eps_n}\E\int_{t_0}^{t_0+\eps_n}\delta(s)\bigg|\big(\mathcal{C}(\Y_{\eps_n}(s))-\mathcal{C}(\y_0),(\A+\I)\Y_{\eps_n}(s)\big)\bigg|(1+\|\Y_{\eps_n}(s))\|_{\V}^2)^{m-1}\d s}_{\text{$\mathcal{J}_1$}}
				\nonumber\\&\quad+
				\underbrace{	 \frac{1}{\eps_n}\E\int_{t_0}^{t_0+\eps_n}\delta(s)\bigg|\big(\mathcal{C}(\y_0),(\A+\I)\Y_{\eps_n}(s)-(\A+\I)\y_0\big)\bigg|(1+\|\Y_{\eps_n}(s))\|_{\V}^2)^{m-1}\d s}_{\text{$\mathcal{J}_2$}}
				\nonumber\\&\quad+  
				\underbrace{\frac{1}{\eps_n}\E\int_{t_0}^{t_0+\eps_n}\delta(s)\big|\big(\mathcal{C}(\y_0),(\A+\I)\y_0\big)\big|\bigg|(1+\|\Y_{\eps_n}(s))\|_{\V}^2)^{m-1}-(1+\|\y_0\|_{\V}^2)^{m-1}\bigg|\d s}_{\text{$\mathcal{J}_3$}}
				\nonumber\\&\quad+
				\underbrace{ \frac{1}{\eps_n}\E\int_{t_0}^{t_0+\eps_n}|\delta(s)-\delta(t_0)|
					\big|\big(\mathcal{C}(\y_0),(\A+\I)\y_0\big)\big|(1+\|\y_0\|_{\V}^2)^{m-1}\d s}_{\text{$\mathcal{J}_4$}}.
			\end{align}
			By using Taylor's formula, H\"older's and interpolation inequalities, we estimate $\mathcal{J}_1$ as
			\begin{align*}
			\mathcal{J}_1&\leq\frac{1}{\eps_n}\E\int_{t_0}^{t_0+\eps_n}\delta(s)\bigg|\big(\mathcal{C}(\Y_{\eps_n}(s))-\mathcal{C}(\y_0),(\A+\I)\Y_{\eps_n}(s)\big)\bigg|(1+\|\Y_{\eps_n}(s))\|_{\V}^2)^{m-1}\d s
			\nonumber\\&\leq
			\frac{C}{\eps_n}\E\int_{t_0}^{t_0+\eps_n}
			\||\Y_{\eps_n}(s)|^{r-1}+|\y_0|^{r-1}\|_{\wi\L^3}
			\|\Y_{\eps_n}(s)-\y_0\|_{\wi\L^6}
			\|(\A+\I)\Y_{\eps_n}(s)\|_{\H}
			\nonumber\\&\qquad\times(1+\|\Y_{\eps_n}(s))\|_{\V}^2)^{m-1}\d s
			\nonumber\\&\leq
			\frac{C}{\eps_n}\E\int_{t_0}^{t_0+\eps_n}
			\left(\|\Y_{\eps_n}(s)\|_{\wi\L^{3(r-1)}}^{r-1}+\|\y_0\|_{\wi\L^{3(r-1)}}^{r-1}\right)\|\Y_{\eps_n}(s)-\y_0\|_{\wi\L^6}
			\|(\A+\I)\Y_{\eps_n}(s)\|_{\H}
			\nonumber\\&\qquad\times(1+\|\Y_{\eps_n}(s))\|_{\V}^2)^{m-1}\d s. \end{align*}
			
			We now consider the following cases to estimate the most difficult expression $\mathcal{J}_1$, because of which we are getting restrictions on $r$.
			\vskip 2mm
			\noindent
			\textbf{Case-I: For $d=2$ and $r\in(3,\infty)$ .}
			By using the Sobolev embedding $\V\hookrightarrow\wi\L^{r+1}$, for any $r\in(3,\infty)$, H\"older's inequality, \eqref{ssee1} and \eqref{ctsdep0.2}, we estimate $J_4$ as follows:
			{\small	\begin{align*}
				\mathcal{J}_1&\leq\frac{1}{\eps_n}\int_{t_0}^{t_0+\eps_n}\delta(s)\bigg|\big(\mathcal{C}(\Y_{\eps_n}(s))-\mathcal{C}(\y_0),(\A+\I)\Y_{\eps_n}(s)\big)\bigg|(1+\|\Y_{\eps_n}(s))\|_{\V}^2)^{m-1}\d s
				\nonumber\\&\leq
				\frac{C}{\eps_n}\int_{t_0}^{t_0+\eps_n}
				\left(\|\Y_{\eps_n}(s)\|_{\V}^{r-1}+\|\y_0\|_{\V}^{r-1}\right)\|\Y_{\eps_n}(s)-\y_0\|_{\V}
				\|(\A+\I)\Y_{\eps_n}(s)\|_{\H}(1+\|\Y_{\eps_n}(s))\|_{\V}^2)^{m-1}\d s
				\nonumber\\&\leq C
				\left(\int_{t_0}^{t_0+\eps_n}\|\Y_{\eps_n}(s)\|_{\V}^{4(r-1)}
				\d s \right)^{\frac14}
				\left(\int_{t_0}^{t_0+\eps_n}
				\|(\A+\I)\Y_{\eps_n}(s)\|_{\H}^{2}(1+\|\Y_{\eps_n}(s)\|_{\V}^2)^{2(m-1)}\d s\right)^{\frac12}
				\nonumber\\&\quad\times
				\left(\frac{1}{\eps_n}\int_{t_0}^{t_0+\eps_n}
				\|\Y_{\eps_n}(s)-\y_0\|_{\V}^4\d s \right)^{\frac14}
				\nonumber\\&\quad+C
				\left(\int_{t_0}^{t_0+\eps_n}
				\|(\A+\I)\Y_{\eps_n}(s)\|_{\H}^{2}(1+\|\Y_{\eps_n}(s)\|_{\V}^2)^{2(m-1)}\d s\right)^{\frac12}
				\nonumber\\&\quad\times
				\left(\frac{1}{\eps_n}\int_{t_0}^{t_0+\eps_n}
				\|\Y_{\eps_n}(s)-\y_0\|_{\V}^2\d s \right)^{\frac12}
				\nonumber\\&\leq\upsigma(\eps_n).
				\end{align*}}
			\vskip 2mm
			\noindent
			\textbf{Case-II: For $d=3$ and $r\in(3,5)$.} Similar calculations as  above yield 
			{\small		\begin{align*}
				&\mathcal{J}_1\leq\frac{C}{\eps_n}\E\int_{t_0}^{t_0+\eps_n}
				\|\Y_{\eps_n}(s)\|_{\wi\L^{3(r-1)}}^{r-1}\|\Y_{\eps_n}(s)-\y_0\|_{\V}
				\|(\A+\I)\Y_{\eps_n}(s)\|_{\H}(1+\|\Y_{\eps_n}(s))\|_{\V}^2)^{m-1}\d s
				\nonumber\\&\quad+
				\frac{C}{\eps_n}\E\int_{t_0}^{t_0+\eps_n}
				\|\Y_{\eps_n}(s)-\y_0\|_{\V}\|(\A+\I)\Y_{\eps_n}(s)\|_{\H}
				(1+\|\Y_{\eps_n}(s))\|_{\V}^2)^{m-1}\d s
				\nonumber\\&\leq
				\frac{C}{\eps_n}\E\int_{t_0}^{t_0+\eps_n}
				\|\Y_{\eps_n}(s)\|_{\wi\L^{r+1}}\|\Y_{\eps_n}(s)\|_{\wi\L^{3(r+1)}}^{r-2}\|\Y_{\eps_n}(s)-\y_0\|_{\V}
				\|(\A+\I)\Y_{\eps_n}(s)\|_{\H}(1+\|\Y_{\eps_n}(s))\|_{\V}^2)^{m-1}\d s
				\nonumber\\&\quad+
				\frac{C}{\eps_n}\E\int_{t_0}^{t_0+\eps_n}
				\|\Y_{\eps_n}(s)-\y_0\|_{\V}\|(\A+\I)\Y_{\eps_n}(s)\|_{\H}
				(1+\|\Y_{\eps_n}(s))\|_{\V}^2)^{m-1}\d s
				\nonumber\\&\leq C
				\left(\int_{t_0}^{t_0+\eps_n}\E \|\Y_{\eps_n}(s)\|_{\V}^{\frac{4(r+1)}{5-r}}\d s \right)^{\frac{5-r}{4(r+1)}}
				\left(\int_{t_0}^{t_0+\eps_n}
				\E\left(\|(\A+\I)\Y_{\eps_n}(s)\|_{\H}^{2}(1+\|\Y_{\eps_n}(s))\|_{\V}^2)^{2(m-1)}\right)\d s\right)^{\frac12}
				\nonumber\\&\quad\times
				\left(\int_{t_0}^{t_0+\eps_n}\|\Y_{\eps_n}(s)\|_{\wi\L^{3(r+1)}}^{r+1}\d s\right)^{\frac{r-2}{r+1}}
				\left(\frac{1}{\eps_n}\int_{t_0}^{t_0+\eps_n}\E
				\|\Y_{\eps_n}(s)-\y_0\|_{\V}^{\frac{4(r+1)}{5-r}}\d s \right)^{\frac{5-r}{4(r+1)}}
				\nonumber\\&\quad+
				\left(\frac{1}{\eps_n}\int_{t_0}^{t_0+\eps_n}\E
				\|\Y_{\eps_n}(s)-\y_0\|_{\V}^2\d s \right)^{\frac12}
				\left(\int_{t_0}^{t_0+\eps_n}
				\E\left(\|(\A+\I)\Y_{\eps_n}(s)\|_{\H}^{2}(1+\|\Y_{\eps_n}(s))\|_{\V}^2)^{2(m-1)}\right)\d s\right)^{\frac12}
				\nonumber\\&\leq\upsigma(\eps_n),
			\end{align*}}
			where in the last step we have used Sobolev embedding $\V\hookrightarrow\wi\L^{r+1}$ for any $r\in(3,5)$, uniform energy estimate \eqref{ssee1}, \eqref{ctsdep0.1} and \eqref{vdp7}. 
		
	   The convergence of other terms in \eqref{wknCA} can be shown similarly as we have demonstrated for the convective terms in \eqref{wknBA}.

	\section{Infinite horizon problem}
In applications, there are optimal control problems where one needs to study the case of infinite horizon (that is, $T=+\infty$). The primary motivation for studying the infinite horizon problem arises from economics, where the long-term performance is much more relevant than short-term gains. One can refer to Ramsey \cite{FR} for one of the earliest economic optimization problems where the dynamical system is observed over an unbounded time interval.

In this section, we briefly discuss the concept of viscosity solution for the HJB equation in the context of the infinite horizon optimal control problem and its properties. Unlike the finite time horizon case, where the HJB equation is time-dependent (or parabolic), the infinite horizon case gives rise to the stationary (or elliptic) HJB equation. 

We consider an infinite horizon problem described by the state equation \eqref{stap} with a cost functional of the form 
 \begin{align}\label{costF1}
	J(t,\y;\a(\cdot))=
	\E\left\{\int_t^{+\infty}e^{-\lambda(s-t)}[\ell(\Y(s;t,\y,\a(\cdot)),\a(s))] \d s\right\},
\end{align}
where $\ell:\H\times\U\to\R$ is the running cost which is independent of $t$ and $\lambda>0$ is a discount factor. We assume that the reference probability space and solutions (or state) are defined on $[t,+\infty)$. The optimal control problem is to find an admissible control $\a(\cdot)\in\mathscr{U}_t^{\nu}$ which minimizes the cost functional \eqref{costF1}.

We define the value function $\mathcal{V}:\H\to\mathbb{R}$ corresponding to the cost functional \eqref{costF1}, as
	\begin{align}\label{valueINF}
		\mathcal{V}(t,\y):=\inf\limits_{\a(\cdot)\in\mathscr{U}_t^{\nu}} J(t,\y;\a(\cdot)).
	\end{align}
We are interested in the viscosity solution of the following infinite-dimensional second order stationary HJB equation related to the stochastic optimal control problem \eqref{stap} and \eqref{costF1}
 	\begin{equation}\label{HJBE1}
	\left\{
	\begin{aligned}
	\lambda u-\frac12\mathrm{Tr}(\Q\D^2u)+ &(\mu\A\y+\mathcal{B}(\y)+\alpha\y+
		\beta\mathcal{C}(\y),\D u)\\+F(\y,\D u)&=0, \ \text{ in } \ \H.
	\end{aligned}
	\right.
\end{equation}
where $F$ is the Hamiltonian function, which is defined by
\begin{align}\label{hamfunc1}
	F(\y,\p):=\inf\limits_{\a\in\U} \left\{(\f(\a),\p)+\ell(\y,\a)\right\}.
\end{align}

Let us now we provide the definition of test function and viscosity solution for the stationary HJB equation \eqref{HJBE1}. 
\begin{definition}
	A function $\uppsi:\H\to\R$ is said to be the \emph{test function} for \eqref{HJBE1} if $\uppsi(\y)=\upvarphi(\y)+\mathfrak{h}(1+\|\y\|_{\V}^2)^m$, where
	\begin{itemize}
	\item $\upvarphi\in\mathrm{C}^{2}(\H)$ and is such that $\D\upvarphi$ and $\D^2\upvarphi$ are uniformly continuous on $\H$;
	\item $\mathfrak{h}>0$.
\end{itemize}
\end{definition}
In the definition given below, we assume that $F:\V\times\H\to\R$.
\begin{definition}
	A weakly sequentially upper-semicontinuous (respectively, lower-semicontinuous) function $u:\V\to\R$ is called \emph{a viscosity subsolution} (respectively, \emph{supersolution}) of \eqref{HJBE1} if whenever $u-\uppsi$ has global maximum (respectively, $u+\uppsi$ has a global minimum) at a point $\y\in\V$ for every test function $\uppsi$, then $\y\in\V_2$ and 
	\begin{align*}
		&\lambda u-\frac12\mathrm{Tr}(\Q\D^2\uppsi(t,\y))\\&\quad+
		(\mu\A\y+\mathcal{B}(\y)+\alpha\y+\beta\mathcal{C}(\y) ,\D\uppsi(\y))+F(\y,\D\uppsi(\y))\leq0
	\end{align*}
	(respectively, 
	\begin{align*}
		&\lambda u+\frac12\mathrm{Tr}(\Q\D^2\uppsi(t,\y))\\&-
		(\mu\A\y+\mathcal{B}(\y)+\alpha\y+\beta\mathcal{C}(\y),\D\uppsi(t,\y))+F(\y,-\D\uppsi(\y))\geq0.)
	\end{align*}
	
	A \emph{viscosity solution} of \eqref{HJBE1} is a function which is both viscosity subsolution and viscosity supersolution.
\end{definition}

We now state the comparison principle for the case of stationary HJB equation \eqref{HJBE1}, which can be proved in a similar way as we have proved the comparison principle for the time dependent case (see Theorem \ref{comparison}).  
\begin{theorem}\label{compsta}
	Let Hypothesis \ref{hypF14} holds for the Hamiltonian $F:\V\times\H\to\R$ given in \eqref{hamfunc1}. Let $u:\V\to\R$ and $v:\V\to\R$ be respectively, viscosity subsolution and supersolution of \eqref{HJBE1}. We further assume that 	
	\begin{align*}
		u(\y), -v(\y)\leq C(1+\|\y\|_{\V}^k), 
	\end{align*}
	for some constant $C>0$ and $k\geq0$. Then, for $r$ in Table \ref{Table1}, we have $$u\leq v\ \text{ on } \ \V.$$ 
\end{theorem}

Finally, we state a result which refers to the existence and uniqueness of viscosity solutions and can be proved in a similar way as we have proved Theorem \ref{extunqvisc} (for the time-dependent case).
\begin{theorem}
Let us assume that the running cost $\ell(\cdot)$ satisfies assumptions \eqref{vh1}-\eqref{vh2} of  Hypothesis \ref{valueH}.  Moreover, we assume that $\f:\U\to\V$ is bounded and continuous. Then, for the values of $r$ given in Table \ref{Table2}, the value function $\mathcal{V}$ defined in \eqref{valueINF} is the unique viscosity solution of the HJB equation \eqref{HJBE1} within the class of viscosity solutions $\u$ satisfying
\begin{align*}
	|u(\y)|\leq C(1+\|\y\|_{\V}^k), \  \ \y\in\V,
\end{align*} 
for some $k\geq0$.
\end{theorem}

\medskip\noindent
\textbf{Acknowledgments:} The first author would like to thank Ministry of Education, Government of India - MHRD for financial assistance. Support for M. T. Mohan's research received from the National Board of Higher Mathematics (NBHM), Department of Atomic Energy, Government of India (Project No. 02011/13/2025/NBHM(R.P)/R\&D II/1137).

	\end{appendix}


\begin{thebibliography}{99}
		
		\bibitem{F.T.} F. Abergel and R. Temam, On some control problems in fluid mechanics, \emph{Theor. Comput. Fluid Dyn.}, \textbf{1} (1990), 303--325.
		
		  
		 \bibitem{VBGDP} V. Barbu and G. Da Prato, \emph{Hamilton–Jacobi Equations in Hilbert Spaces}, Res. Notes Math., Pitman, Boston, 1983.
		  
		 \bibitem{VSB} V. S. Borkar, 
		 \emph{Optimal Control of Diffusion Processes}, Pitman Res. Notes Math. Ser., \textbf{203}, New York, 1989.
		 
		 
		\bibitem{SiB} S. Breneis, On variation functions and their moduli of continuity, \emph{J. Math. Anal. Appl.}, {491} (2) (2020), 124349.
		 
		\bibitem{ZBSP} Z. Brzezniak and S. Peszat,
		Space-time continuous solutions to SPDE's driven by a homogeneous Wiener process,
		\emph{Studia Math.}, \textbf{137}(3) (1999), 261--299.
		
		\bibitem{ZBSP1} Z. Brzezniak and S. Peszat,
		Stochastic two dimensional Euler equations, \emph{Ann. Probab.},
		\textbf{29}(4) (2001), 1796--1832.
		
		\bibitem{HB} H. Brezis, \emph{Functional analysis, Sobolev Spaces and Partial Differential Equations}, Universitext Springer, New York, 2011. 
		
		\bibitem{Schan} S. Chandrasekhar, \emph{Hydrodynamic and Hydromagnetic Stability}, Dover Publications, New York, 1981.
		
		\bibitem{PGC} P. G. Ciarlet, \emph{Linear and Nonlinear Functional Analysis with Applications}, SIAM Philadelphia, Philadelphia (2013).
			
		\bibitem{MGL} M. G. Crandall and P. L. Lions, Viscosity solutions of Hamilton-Jacobi equations, \emph{Trans. Amer. Math. Soc.}, \textbf{277}(1) (1983), 1--42.
		
		\bibitem{MGEL} M. G. Crandall, L. C. Evans and P. L. Lions, Some properties of viscosity solutions of Hamilton-Jacobi equations,
		\emph{Trans. Amer. Math. Soc.} \textbf{282}(2) (1984), 487--502.
		
		\bibitem{MGL1} M. G. Crandall and P. L. Lions, Hamilton-Jacobi equations in infinite dimensions-I: Uniqueness of viscosity solutions,
		\emph{J. Funct. Anal.}, \textbf{62}(3) (1985), 379--396.
		
		\bibitem{MGL2} M. G. Crandall and P. L. Lions, Hamilton-Jacobi equations in infinite dimensions-II: Existence of viscosity solutions,
		\emph{J. Funct. Anal.}, \textbf{65}(3) (1986), 368--405.
		
		\bibitem{MGL3} M. G. Crandall and P. L. Lions, Hamilton-Jacobi equations in infinite dimensions-III,
		\emph{J. Funct. Anal.}, \textbf{68}(2) (1986), 214--247.
		
		\bibitem{MGL4}  M. G. Crandall and P. L. Lions, Viscosity solutions of Hamilton-Jacobi equations in infinite dimensions-IV: Hamiltonians with unbounded linear terms,
		\emph{J. Funct. Anal.}, \textbf{90}(2) (1990), 237--283.
		
		\bibitem{MGHL} M. G. Crandall, H. Ishii and P. L. Lions, User's guide to viscosity solutions of second order partial differential equations,
		\emph{Bull. Amer. Math. Soc.} \textbf{27}(1) (1992), 1--67.
		
		\bibitem{gPaD} G. Da Prato and A. Debussche, Dynamic programming for the stochastic Navier-Stokes equations,
		Special issue for R. Temam's 60th birthday, \emph{M2AN Math. Model. Numer. Anal.}, \textbf{34} (2) (2000), 459--475.
		
		\bibitem{GDPAD} G. Da Prato and A. Debussche, Ergodicity for the 3D stochastic Navier-Stokes equations, \emph{J. Math.
		Pures Appl. (9)} \textbf{82}(8) (2003), 877--947.
		
		 \bibitem{gdp} G. Da Prato and J. Zabczyk, \emph{Stochastic Equations in Infinite Dimensions},
		Second edition, Encyclopedia of mathematics and its applications, 152, Cambridge University Press, Cambridge, 2014.
		
		\bibitem{FGSSA} F. Gozzi, S. S. Sritharan and A. Swiech,
		Bellman equations associated to the optimal feedback control of stochastic Navier-Stokes equations,
		\emph{Comm. Pure Appl. Math.} \textbf{58}(5) (2005), 671--700.
		
		\bibitem{FGSSA1} F. Gozzi, S. S. Sritharan and A. Swiech,
		Viscosity solutions of dynamic-programming equations for the optimal control of the two-dimensional Navier-Stokes equations,
		\emph{Arch. Ration. Mech. Anal.} \textbf{163}(4) (2002), 295--327.
		
		\bibitem{MDG} M. D. Gunzburger, \emph{Flow control}, Proceedings of the Workshop on Period of Concentration in Flow Control Held at the University of Minnesota, 1992, The IMA Volumes in Mathematics and Its Applications, \textbf{68}, Springer-Verlag, New York, 1995.
		
		\bibitem{JFTGK} J. Feng and T. G. Kurtz, \emph{Large Deviations for Stochastic Processes}, Mathematical Surveys and Monographs, \textbf{131}, American Mathematical Society, Providence, RI, 2006.
		
		\bibitem{GFAS} G. Fabbri, F. Gozzi and A. Swiech, \emph{Stochastic Optimal Control in Infinite Dimension: Dynamic Programming and HJB Equations}, Probability Theory and Stochastic Modelling, \textbf{82}, Springer, Cham, 2017. 
		
		\bibitem{FKS} R. Farwig, H. Kozono and H. Sohr,
		An $L^q$-approach to Stokes and Navier--Stokes equations in general domains, \emph{Acta Math.}, \textbf{195} (2005), 21--53.
		
		\bibitem{WHF} W. H. Fleming and H. M. Soner, \emph{Controlled Markov Processes and Viscosity Solutions}, \textbf{25},	Springer, New York, 2006.
		
		\bibitem{WHF1} W. H. Fleming and R. W. Rishel, \emph{Deterministic and Stochastic Optimal Control},
		Springer-Verlag, Berlin-New York, 1975.
		
		
		\bibitem{DFHM} D. Fujiwara and H. Morimoto, An $L^r$-theorem of the Helmholtz decomposition of vector fields, \emph{J. Fac. Sci. Univ. Tokyo Sect. IA Math.}, \textbf{24}(3) (1977), 685--700.
		
		\bibitem{SGMM1} S. Gautam and M. T. Mohan, Kolmogorov equations for 2D stochastic convective Brinkman-Forchheimer equations: Analysis and Applications, \emph{Submitted}.
		\url{https://arxiv.org/pdf/2412.20948}
		
		\bibitem{SMTM} S. Gautam and M. T. Mohan, On the convective Brinkman-Forccheimer equations, \emph{Accepted in Dynamics
			of Partial Differential Equations (2025)}.
		
		
		 
		\bibitem{FFFG} F. Flandoli and F. Gozzi, Kolmogorov equation associated to a stochastic Navier-Stokes equation, \emph{J. Funct. Anal.}, \textbf{160}(1) (1998), 312--336.
		
		\bibitem{KWH}	K. W. Hajduk and J. C. Robinson, Energy equality for the 3D critical convective Brinkman--Forchheimer equations, \emph{Journal of Differential Equations}, {\bf 263} (2017), 7141--7161.
		
		\bibitem{CHJS} C. Henoch and J. Stace, Experimental investigation of a salt water turbulent boundary layer modified by an applied streamwise magnetohydrodynamic body force, \emph{Phys.
		Fluids} \textbf{7} (1995) 1371--1383.
		
		
		\bibitem{Hsh1} H. Ishii, Uniqueness of unbounded viscosity solution of Hamilton-Jacobi equations, \emph{Indiana Univ. Math. J.},
		 \textbf{33}(5) (1984), 721--748.
		
		\bibitem{Hsh} H. Ishii, Viscosity solutions of nonlinear second-order partial differential equations in Hilbert spaces, \emph{Comm. Partial Diflerential Equations} \textbf{18} (3-4) (1993), 601--650.
		
		 \bibitem{KKMTM} K. Kinra and M. T. Mohan,  Random attractors and invariant measures for stochastic convective Brinkman--Forchheimer equations on 2D and 3D unbounded domains, \emph{Discrete Contin. Dyn. Syst. Ser. B}, {\bf 29}(1) (2024),  377--425.
		 
		 \bibitem{NVK} N. V. Krylov, Controlled Diffusion Processes, \textbf{14}, Springer-Verlag, New York-Berlin, 1980.
		
		\bibitem{HHK} H. H. Kuo, Gaussian Measures in Banach Spaces,
		Lecture Notes in Math., \textbf{463},
		Springer-Verlag, Berlin-New York, 1975.
		
		\bibitem{PL1} P. L. Lions, On the Hamilton-Jacobi-Bellman equations, \emph{Acta Appl Math} \textbf{1}, 171141 (1983).
		
		\bibitem{PL2} P. L. Lions, Optimal control of diffusion processes and Hamilton-Jacobi-Bellman equations. I. The dynamic programming principle and applications, 
		\emph{Comm. Partial Differential Equations} \textbf{8}(10) (1983), 1101--1174.
		
		\bibitem{PL3} P. L. Lions, Optimal control of diffusion processes and Hamilton-Jacobi-Bellman equations. II. Viscosity solutions and uniqueness, \emph{Comm. Partial Differential Equations} \textbf{8}(11) (1983), 1229--1276.
		
		\bibitem{PL4} P. L. Lions, \emph{Viscosity solutions of fully nonlinear second-order equations and optimal stochastic control in infinite dimensions, Part-II:  Optimal Control of Zakai's Equation}, in Stochastic Partial Differential Equations and Applications 11, Lecture Notes in Math., \textbf{1390}, Springer, Berlin, 1989.
		
		
		 \bibitem{MTT} P. A. Markowich, E. S. Titi and S. Trabelsi, Continuous data assimilation for the three dimensional Brinkman-Forchheimer-extended Darcy model, \emph{Nonlinearity}, \textbf{29}(4) (2016), 1292--1328. 
		
		\bibitem{CMMR} C. Marinelli and M. R\"ockner, 
		On the maximal inequalities of Burkholder, Davis and Gundy,
		\emph{Expo. Math.}, \textbf{34}(1) (2016), 1--26.
		
		\bibitem{HKM}  H. K. Moffatt, \emph{Magnetic Field Generation in Electrically Conducting Fluids}, Cambridge University Press, New York, 1978.
		
		\bibitem{MTM8} M. T. Mohan, Stochastic convective Brinkman-Forchheimer equations. \url{https://arxiv.org/pdf/2007.09376}. 
		
		\bibitem{MTM9} M. T. Mohan, Martingale solutions of two and three dimensional stochastic convective Brinkman-Forchheimer equations forced by Lévy noise. \url{https://arxiv.org/pdf/2109.05510}.
		
		\bibitem{MT5} M. T. Mohan, $\mathrm{L}^p$-solutions of deterministic and stochastic convective Brinkman-Forchheimer equations, \emph{ Anal. Math. Phys.}, \textbf{11}(4) (2021),  Paper No. 164, 33 pp.
		
%
		
		\bibitem{FR} F. Ramsey, A mathematical theory of saving, \emph{Economic Journal}, \textbf{38} (1928), 543--549.
		
		\bibitem{JCR} J. C. Robinson, J. L. Rodrigo and W. Sadowski, \emph{The Three-Dimensional Navier--Stokes equations, classical theory}, Cambridge University Press, Cambridge, UK, 2016.
		
		 \bibitem{JCR1} J. C. Robinson, \emph{Infinite-Dimensional Dynamical Systems: An Introduction to Dissipatives Parabolic PDEs and the Theory of Global Attractors}, Cambridge University Press, 2001.
		
		\bibitem{JCRW} J. C. Robinson and W. Sadowski,
		A local smoothness criterion for solutions of the 3D Navier-Stokes equations, \emph{Rend. Semin. Mat. Univ. Padova},
		 \textbf{131} (2014), 159--178.
		
		\bibitem{MRBZ} M. R\"ockner, S. Byron and X. Zhang, Yamada-Watanabe theorem for stochastic evolution equations in infinite dimensions, \emph{Condensed Matter Physics} (2008).
		
		\bibitem{Kshi} K. Shimano, A class of Hamilton-Jacobi equations with unbounded coefficients in Hilbert spaces,
		\emph{Appl. Math. Optim.}, \textbf{45}(1) (2002), 75--98.
		
		
		\bibitem{AS2} A. \'Swiech, A PDE approach to large deviations in Hilbert spaces, \emph{Stochastic Process. Appl.} \textbf{119}(4) (2009), 1081--1123.
		
		\bibitem{AS3Z} A. \'Swiech and J. Zabczyk, Large deviations for stochastic PDE with L\'evy noise, \emph{J. Funct. Anal.}, \textbf{260}(3) (2011), 674--723.
		
		\bibitem{AS5Z} A. \'Swiech and J. Zabczyk, Uniqueness for integro-PDE in Hilbert spaces, \emph{Potential Anal.} \textbf{38}(1) (2013), 233--259.
		
		\bibitem{AS4Z} A. \'Swiech and J. Zabczyk, Integro-PDE in Hilbert spaces: existence of viscosity solutions, \emph{Potential Anal.}, \textbf{45}(4) (2016), 703--736.
		
		
		\bibitem{SSS1} S. S. Sritharan,
		Dynamic programming of the Navier-Stokes equations,
		\emph{Systems Control Lett.} \textbf{16}(4) (1991), 299--307.
		
		\bibitem{SSS} S. S. Sritharan, An introduction to deterministic and stochastic control of viscous flow, \emph{Optimal Control of Viscous Flow}, 1--42, SIAM, Philadelphia, PA, 1998.
		
		\bibitem{Te} R. Temam,  \emph{Navier--Stokes Equations: Theory and Numerical Analysis}, North-Holland, Amsterdam, 1984.
		
	
	    \bibitem{JYXYZ} J. Yong and X. Y. Zhou, \emph{Stochastic controls:
	    	Hamiltonian systems and HJB equation}, 
	    Springer-Verlag, New York, 1999.
	\end{thebibliography}
\end{document}